\documentclass[11pt]{article}
\usepackage[utf8]{inputenc}
\usepackage[colorlinks=true, linkcolor=blue, urlcolor=blue, citecolor=blue]{hyperref}
\usepackage{geometry}

\usepackage{amsfonts}
\usepackage{amsfonts,latexsym}
\usepackage{amsmath}
\usepackage{amsthm}
\usepackage{amssymb}
\usepackage{mathrsfs}
\usepackage{xcolor}
\usepackage{cite}
\usepackage{verbatim}
\usepackage{enumitem}
\usepackage{amsmath, amssymb, booktabs} 
\usepackage{array}  
\usepackage{xcolor}  
\usepackage{hyperref}  
\usepackage{xcolor}  
\usepackage{hyperref}  
\hypersetup{
	allcolors = blue  
}

\newtheorem{theorem}{\indent Theorem}[section]
\numberwithin{theorem}{section}  
\newtheorem{proposition}[theorem]{\indent Proposition}
\newtheorem{definition}[theorem]{\indent Definition}
\newtheorem{lemma}[theorem]{\indent Lemma}
\newtheorem{remark}[theorem]{\indent Remark}
\newtheorem{assumption}[theorem]{\indent Assumption}
\newtheorem{corollary}[theorem]{\indent Corollary}

\numberwithin{equation}{section}

\begin{document}
\title{Lagrangian chaos for the 2D Boussinesq equations with a degenerate random forcing}
\author{ Dengdi Chen$^{1}$, Yan Zheng$^{1}$ \thanks{\emph{E-mail addresses:} yanzhengyl@163.com} \\
{\small \it $^{1}$College of Sciences, National University of Defense Technology,}\\
{\small \it Changsha, 410073,  People's Republic
of China} }
\date{}

\maketitle
\begin{abstract}
We demonstrate that Lagrangian flow for the 2D Boussinesq equations under degenerate noise exhibit chaotic behavior characterized by the strict positivity of the top Lyapunov exponent, where the degenerate noise acts only on a few Fourier modes of the temperature equation. To achieve this, we overcome difficulties arising from the degeneracy of noise and its intricate interaction with the nonlinear terms. This is accomplished by introducing a solution-dependent manifold spanning condition to establish probabilistic spectral 
bound on a cone for the Malliavin matrix associated with the extended system. Additionally, the approximate controllability of the extended system is realized by constructing smooth controls based on shear and cellular flows.
\end{abstract}
\tableofcontents
\section{Introduction}
In this paper, we study the stochastic flow of diffeomorphisms $\mathbf{x}_t:\mathbb{T}^2\to \mathbb{T}^2$, $t\ge 0$, defined by the random ODE
\begin{equation} \label{Lagrange1}
	\frac{d}{dt}\mathbf{x}_t=\mathbf{u}_t(\mathbf{x}_t),\quad \mathbf{x}_0=x.
\end{equation}
Here, the random velocity field $\mathbf{u}_t:\mathbb{T}^2\to \mathbb{T}^2$ at time $t\ge 0$ evolves according to the following stochastic Boussinesq equations 
\begin{equation} \label{Boussinesq}
 \begin{aligned}
 \left\{
\begin{array}{lr}
d\mathbf{u}+(\mathbf{u}\cdot \nabla\mathbf{u})dt=(-\nabla p+\nu_1\Delta\mathbf{u}+\mathbf{g}\theta)dt,\\
d\theta+(\mathbf{u}\cdot \nabla\theta)dt=\nu_2\Delta\theta dt+\sigma_\theta dW,\\
\nabla\cdot\mathbf{u}=0,
\end{array}\right.
\end{aligned}
\end{equation}
where $p$ denotes the (density-normalized) pressure and $\theta$ denotes the temperature of the viscous incompressible fluid. And the parameters $\nu_1,\nu_2>0$ are respectively the kinematic viscosity and thermal
diffusivity of the fluid and $\mathbf{g}={(0,g)}^T$ with $g\neq 0$ is the product of the gravitational
constant and the thermal expansion coefficient. The spatial variable $x=(x_1,x_2)$ belongs
to a two-dimensional torus $\mathbb{T}^2$. That is, we impose periodic boundary conditions in space.
We consider a degenerate stochastic forcing $\sigma_\theta dW$, which acts only on a few Fourier
modes and exclusively through the temperature equation.

We prove that the dynamical system defined via \eqref{Lagrange1} possesses a strictly positive top Lyapunov exponent. More precisely, we establish:
\begin{theorem} \label{L1}
With white noise acting only on the two largest standard modes of the temperature equation,
$$
\sigma_\theta dW=\alpha_1cosx_1dW^1+\alpha_2sinx_1dW^2+\alpha_3cosx_2dW^3+\alpha_4sinx_2dW^4,
$$
there exists a deterministic constant $\lambda_+>0$ for the system \eqref{Lagrange1}, depending on the Boussinesq equations \eqref{Boussinesq} such that the following limit holds:
\begin{equation} \label{lambda+}
	\lim\limits_{t \to \infty}\frac{1}{t}\mathop{\rm log}|D_x\mathbf{x}_t|=\lambda_+  \,\,\,\, \text{for} \,\,\, \,\mu^1\times \mathbb{P}-a.e.\, (\mathbf{u}_0,\theta_0, x,w),
\end{equation} 
where $D_x\mathbf{x}_t$ refers to the Jacobian matrix of $\mathbf{x}_t:\mathbb{T}^2\to \mathbb{T}^2$ taken at $x$. Here, we refer to $(\mathbf{u}_t,\theta_t,\mathbf{x}_t)$ as the Lagrangian process  associated with $(\mathbf{u}_t,\theta_t)$, and let $\mu^1$ denote its corresponding unique stationary measure. The phenomenon where Lagrangian flow $\mathbf{x}_t$ exhibits chaotic characteristics due to exponential sensitivity to initial data (i.e., $\lambda_+>0$) is sometimes referred to as Lagrangian chaos. 
\end{theorem}
\begin{remark} \label{1.2}
In fact, we can present the following stronger result, which helps support findings on the mixing of passive scalars (c.f. \cite{JEMS,BBP22}): 
\begin{equation} \label{lambda+-P}
	\lim\limits_{t \to \infty}\frac{1}{t}\mathop{\rm log}|D_x\mathbf{x}_t\nu|=\lambda_+>0 \,\, \text{for} \,\,\, \mu^P\times \mathbb{P}-a.e.\, (\mathbf{u}_0,\theta_0, x,v,w),
\end{equation}
where $\mu^P$ denotes the unique stationary measure for projective process (c.f. Definition \ref{Lagrange-process-def}). 
\end{remark} 	
\subsection{Background and motivations} \label{Relevant literature}
In the early 20th century, Rayleigh \cite{JFA-79} proposed the use of the governing equations formulated by Boussinesq \cite{JFA-11} to investigate buoyancy-driven thermal convection mechanisms, with the goal of explaining the experimental observations made by Bénard \cite{JFA-7}. These equations describe convective phenomena within a fluid layer confined between two parallel plates, where the temperature $\theta$ is maintained at a fixed value $\theta_t$ at the top and $\theta_b$ at the bottom, with $\theta_t\le \theta_b$ (i.e., heated from below). 
Today, this classical Boussinesq equations serve as a fundamental framework across diverse physical domains, including climate and weather systems, plate tectonics, and stellar internal dynamics (see \cite{JFA-10, JFA-13, JFA-37, JFA-49, JFA-95} and references therein). The specific form of the Boussinesq equations considered in this work, namely \eqref{Boussinesq} supplemented with periodic boundary conditions, is sometimes termed the `homogeneous Rayleigh–Bénard' system in the physics literature. To circumvent the potential emergence of unbounded solutions within this periodic setting \cite{JFA-14}, we examine system \eqref{Boussinesq} under isothermal conditions.

We now briefly outline the motivation for introducing this type of degenerate stochastic forcing in \eqref{Boussinesq}. Firstly, as early as the 19th century, J.V. Boussinesq conjectured that turbulence could not be fully characterized by deterministic methods alone, advocating for a stochastic framework \cite{JFA-103}. Secondly, numerous researchers have proposed that the Navier-Stokes equations subjected to degenerate white noise forcing can serve as a substitute model for the idealized `general' stirring at large scales assumed in fundamental turbulence theories, and this setting is common in the turbulence literature (cf. \cite{JFA-43, JFA-92, JFA-108} and references therein).

Unlike research on well-posedness, ergodicity and large deviations etc. for the stochastic Boussinesq equations (cf. \cite{long-1,long-2,long-3,long-4,long-5,long-6,long-7,long-8,long-9,long-10,JFA} and references therein), the dynamical properties of its solutions—particularly chaotic behavior offering key insights into turbulence \cite{BJPV98}—remain poorly understood, with rigorous mathematical results notably absent.

The study of chaotic dynamics in dynamical systems has been a major focus for decades. Chaos is generally characterized by sensitive dependence on initial conditions, signified by a positive top Lyapunov exponent. In random dynamical systems, noise often induces an averaging effect that facilitates chaos analysis compared to purely deterministic settings \cite{You13, BBPS23}. 
Current research on chaos in stochastic fluid models—including chaos in the velocity field (Eulerian chaos) and chaos in particle trajectories (Lagrangian chaos)—concentrates primarily on the Navier–Stokes equations. 
For Eulerian chaos, the problem for the full 2D stochastic Navier–Stokes system remains unresolved; only its Galerkin approximation has been rigorously proven to exhibit chaotic behavior \cite{BPS24}.

The physics literature extensively investigates both Eulerian and Lagrangian chaos and their interconnections \cite{AGM96, BJPV98, CFVP91, FdCN01, GV94}, which also underscores the importance of studying Lagrangian chaos. However, rigorous mathematical results in this area remain scarce. Bedrossian et al. \cite{JEMS} provided the first rigorous proof of Lagrangian chaos for time-continuous velocity fields generated by physical fluid models (e.g., Stokes or Navier-Stokes), where the Navier-Stokes system was driven by a white-in-time noise that is non-degenerate at high Fourier modes. More precisely, the noise is assumed to be of the form $Q\dot{W}_t(x) = \sum_{k \in \mathbb{Z}\backslash\mathbf{0}} q_k e_k(x) \gamma_k \dot{W}_t^k$, where $\{e_k\gamma_k\}_{k \in\mathbb{Z}\backslash\mathbf{0}}$ forms a basis for the Hilbert space of square integrable, mean-zero, divergence-free vector fields on $\mathbb{T}^2$, and $\{q_k\}$ is a sequence of real numbers such that $c|k|^{-\alpha}\le |q_k|\le C|k|^{-\alpha}$ for $|k|\ge L$ with  some constants $c,C,L>0$ and $\alpha>5$. Crucially, this type of noise ensures the Lagrangian process's strong Feller property, which is essential to their proof: (a) Existence of $\lambda_+$. The strong Feller property enables proving ergodicity of the Lagrangian process, ensuring the existence of the exponent $\lambda_+$. Combined with the volume-preserving nature, this further yields $\lambda_+\ge 0$. (b) Positivity of $\lambda_+$. By Furstenberg's criterion \cite{Fur63, Led86, re89}, $\lambda_+=0$ implies the almost sure existence of a measurable deterministic invariant structure under the dynamics of the triplet $(u_t,x_t,\mathcal{A}_t)$ (c.f. \eqref{invariant-measure}). 
For finite-dimensional systems, if one demonstrates sufficient nondegeneracy in the distribution of $\mathcal{A}_t$, this excludes such structures and proves that $\lambda_+>0$. However, this strategy fails for the infinite-dimensional 2D Navier–Stokes system. Under the strong Feller assumption, \cite{JEMS} derived a refined Furstenberg criterion limiting possible invariant structures to two continuous types. Approximate controllability arguments subsequently excluded their existence.

Building upon this foundation, Cooperman and Rowan \cite{NS} refined the results of \cite{JEMS} by introducing a novel and ingenious method to construct continuous invariant family of projective measures. 
This approach circumvents the need for the strong Feller property and instead relies on smoothness estimates of the transition functions for the triplet matrix process. For the 2D Navier-Stokes system driven by highly degenerate white noise, these estimates were rigorously verified using Malliavin calculus techniques from \cite{HM06}, and thus Cooperman and Rowan established the existence of Lagrangian chaos. Recently, using the framework of \cite{NS}, Nersesyan et al. \cite{V2} derived an analogous Furstenberg's criterion for the 2D Navier-Stokes system driven by highly degenerate, bounded, non-Gaussian noise and similarly proved the existence of Lagrangian chaos. Moreover, based on the Furstenberg's criterion established in \cite{JEMS}, Agresti \cite{PE} similarly established Lagrangian chaos for the 3D stochastic primitive equations with nondegenerate noise.

This paper focuses on proving Lagrangian chaos for the 2D Boussinesq equations driven by highly degenerate white noise. To the best of our knowledge, no prior rigorous results exist concerning chaos in the Boussinesq equations.

\subsection{Mathematical setting and notations} \label{notaion}
Let $\mathbb{T}^2={[0,2\pi]}^2$ denote the period box. 
At numerous points throughout this paper, consideration will also be given to the equivalent vorticity formulation of equation \eqref{Boussinesq}.
Namely, if we denote $\omega:=\nabla^\bot\cdot \mathbf{u}=\partial_1u_2-\partial_2u_1$, then by a standard calculation we obtain
\begin{equation} \label{Boussinesq-w}
	\begin{aligned}
		\left\{
		\begin{array}{lr}
			d\omega+(\mathbf{u}\cdot \nabla\omega-\nu_1\Delta\omega)dt=g\partial_1\theta dt,\\
			d\theta+(\mathbf{u}\cdot \nabla\theta-\nu_2\Delta\theta)dt=\sigma_\theta dW.
		\end{array}\right.
	\end{aligned}
\end{equation}
To close the system \eqref{Boussinesq-w}, we calculate $\mathbf{u}$ from $\omega$ by Biot-Savart law, that is $\mathbf{u}=K*\omega$, where $K$ is the Biot-Savart kernel,
so that $\nabla^\bot\cdot \mathbf{u}=\omega$ and $\nabla\cdot\mathbf{u}=0$.

We next introduce a functional setting for the system \eqref{Boussinesq-w}, which represents a standard procedure. Define
$$
\hat{H}:=\left\{U:={(\omega,\theta)}^T\in {L^2(\mathbb{T}^2)}^2:\int_{\mathbb{T}^2}\omega dx=\int_{\mathbb{T}^2}\theta dx=0\right\}
$$
to be the Hilbert space of square integrable, mean-zero, divergence-free vector fields on $(\mathbb{T}^2)^2$, equipped with the norm
\begin{align}\label{JFA-2.4}
\|U\|_{\hat{H}}^2:=\frac{\lambda_1 \nu_1 \nu_2}{g^2}\|\omega\|_{L^2}^2+\|\theta\|_{L^2}^2,
\end{align}
where $\lambda_1=1$ (we omit $\lambda_1$ below) is the principal eigenvalue of $-\Delta$ on $\hat{H}$. Observe that
this norm is equivalent to the standard norm on the space ${L^2(\mathbb{T}^2)}^2$. The associated inner product on $\hat{H}$ is denoted by $\langle\cdot,\cdot\rangle_{\hat{H}}$. Note furthermore that the zero mean property in the definition of $\hat{H}$ is maintained by the flow \eqref{Boussinesq-w}. The higher order Sobolev spaces are denoted
$$
\hat{H}^s:=\left\{U:={(\omega,\theta)}^T\in {W^{s,2}(\mathbb{T}^2)}^2:\int_{\mathbb{T}^2}\omega dx=\int_{\mathbb{T}^2}\theta dx=0\right\} \quad \text{for any } s \geq 0.
$$
where $W^{s,2}(\mathbb{T}^2)$ is classical Sobolev space, which is equivalently denoted by $H^s$. And $\hat{H}^s$ is equipped with the norm
\begin{align}\label{2.3}
\|U\|_{\hat{H}^s}^2:=\frac{\nu_1 \nu_2}{g^2}\|\omega\|_{W^{s,2}}^2+\|\theta\|_{W^{s,2}}^2.
\end{align}


Additionally, in many parts of the manuscript, we also utilize the orthogonal basis in the $\hat{H}^4$ space, along with its finite-dimensional subspaces and their corresponding projection operators. Fix the trigonometric basis:
\begin{equation}\label{temp-base}
\sigma_j^m := 
\begin{cases} 
	(0, \cos(j \cdot x))^T, & m = 0, \\
	(0, \sin(j \cdot x))^T, & m = 1,
\end{cases}
\quad \text{and} \quad
\psi_j^m := 
\begin{cases} 
	(\cos(j \cdot x), 0)^T, & m = 0,\\
	(\sin(j \cdot x), 0)^T, & m = 1,
\end{cases}
\end{equation}
where \( j \in \mathbb{Z}^2 \) and \( x \in \mathbb{T}^2 \). Here, $\sigma_j^{m+m'}$ and $\psi_j^{m+m'}$ mean that $\sigma_j^{{m+m'}(\mod2)}$ and $\psi_j^{m+m'(\mod2)}$ respectively, a convention that will be frequently employed in Section \ref{smoothing estimate} and Subsection \ref{var-assum5.5}. And, we also denote
\[
\mathbb{Z}_+^2 := \left\{j = (j_1, j_2) \in \mathbb{Z}^2 : j_1 > 0 \text{ or } j_1 = 0, \, j_2 > 0 \right\}.
\]

Next, we reformulate system \eqref{Boussinesq-w} into a functional form by introducing the following abstract operators associated with each term in the equations. For $U:=(\omega,\theta)$ and $\tilde{U}:=(\tilde\omega,\tilde\theta )$, let $A : D(A) = \hat{H}^{n+2} \subset \hat{H}^n \to \hat{H}^n $ be the linear, symmetric, positive definite operator defined by
\[
A U := \left( -\nu_1 \Delta \omega,\, -\nu_2 \Delta \theta \right)^T,
\]
for any \( U \in \hat{H}^{n+2} \). Note that \( A \) is the infinitesimal generator of a semigroup \( e^{-tA} : \hat{H}^n \to \hat{H}^{n+2} \).
For the nonlinear inertial term define \( B : \hat{H}^{n+1} \times \hat{H}^{n+1} \to \hat{H}^{n} \) by
$$
B(U, \tilde{U}) := \left( (K * \omega) \cdot \nabla \tilde{\omega},\, (K * \omega) \cdot \nabla \tilde{\theta} \right)^T, 
$$
for \( U, \tilde{U} \in \hat{H}^{n+1} \). It is well known that 
$$
\| K * \omega \|_{\hat{H}^s} \leq C \| \omega \|_{\hat{H}^{s-1}}, 
$$
and since \( \hat{H}^{n+2} \hookrightarrow \hat{H}^{n} \), we indeed have that \( B(U, \tilde{U}) \in \hat{H}^{n} \). 
Finally, for the buoyancy term define \( G : \hat{H}^{n+1} \to \hat{H}^{n} \) by
$$
G U := \left( g \partial_x \theta,\, 0 \right)^T, 
$$
for \( U \in \hat{H}^{n+1} \).

Next we focus on the stochastic forcing terms appearing in \eqref{Boussinesq-w}. We introduce a finite set $\mathcal{Z}\subset \mathbb{Z}_+^2$ which represents the forced directions in Fourier space (In fact, $\mathcal{Z}=\{(1,0),\, (0,1)\}$). The driving noise process \( W := (W^{j,m})_{j \in \mathcal{Z}, m=0,1} \) is a \( d := 2 \cdot |\mathcal{Z}| \)-dimensional Brownian motion defined relative to a filtered probability space \( (\Omega, \mathcal{F}, \{\mathcal{F}_t\}_{t \geq 0}, \mathbb{P}) \). 
Denoting by \(\{e_j^m\}_{j \in \mathcal{Z}, m=0,1}\) the standard basis of \(\mathbb{R}^{2|\mathcal{Z}|}\) and considering a sequence of non-zero constants \(\{\alpha_j^m\}_{j \in \mathcal{Z}, m=0,1}\), we define a linear map \(\sigma_\theta : \mathbb{R}^{2|\mathcal{Z}|} \to \hat{H}^4\) such that
$$
\sigma_\theta e_j^m := \alpha_j^m \sigma_j^m \quad \text{for any } j \in \mathcal{Z},\, m \in \{0,1\},
$$
where \(\sigma_j^m\) are the basis elements defined in \eqref{temp-base}. Denote the Hilbert-Schmidt norm of \(\sigma_\theta\) by
$$
\|\sigma_\theta\|^2 := \|\sigma_\theta^* \sigma_\theta\| = \sum_{\substack{j \in \mathcal{Z} \\ m \in \{0,1\}}} (\alpha_j^m)^2.
$$
We consider a stochastic forcing of the form
\begin{equation}\label{noise-form}
\sigma_\theta \, dW := \sum_{\substack{j \in \mathcal{Z} \\ m \in \{0,1\}}} \alpha_j^m \sigma_j^m \, dW^{j,m}.
\end{equation}
The index \(\theta\) of \(\sigma_\theta\) indicates that \(\sigma_\theta\) attains nontrivial values only in the second component, i.e., only the \(\theta\) component is directly forced.\\
\textbf{Note.} Throughout, we fix the physical constants $\nu_1,\nu_2>0\, (\kappa:=\min\{\nu_1,\nu_2\}<2),g\neq 0$ and non-zero noise coefficients ${(\alpha_j^m)}_{j\in \mathcal{Z},m= {0,1}}$. Below analytic constants $C$ etc. may change line by line and they implicitly depend on $\nu_1,\nu_2>0,g\neq 0$, and $\alpha_j^m\neq 0$. All other parameter dependencies are indicated explicitly, e.g. ``there exists a constant $C(p)>0$ such that $\ldots$'' means the constant $C$ depends on $p,\nu_1,\nu_2,g$ and $\alpha_j^m$.

With these preliminaries in hand, the equations \eqref{Boussinesq-w} may be written as an abstract stochastic evolution equation on \( \hat{H}^n \):
\begin{equation}\label{Boussinesq-w1}
	dU_t + \left( AU_t + B(U_t,U_t) \right)dt = GU_t\,dt + \sigma_\theta dW, \quad U_0 = U,
\end{equation}
where \( U \in \hat{H}^n \). We say that \( U_t\) is a solution of \eqref{Boussinesq-w1} if it is \( \mathcal{F}_t \)-adapted,
\[
U_t \in C\big([0, \infty); \hat{H}^n\big) \cap L^2_{\mathrm{loc}}\big([0, \infty); \hat{H}^{n+1}\big) \quad \text{a.s.},
\]
and \( U_t \) satisfies \eqref{Boussinesq-w1} in the mild sense, that is,
\begin{equation}
	U_t = e^{-At}U_0 - \int_{0}^{t} e^{-A(t-s)} \big( B(U_s,U_s) - GU_s \big) \, \mathrm{d}s + \int_{0}^{t} e^{-A(t-s)} \sigma_\theta \, \mathrm{d}W_s.
\end{equation}
Observe that \( B, G : \hat{H}^{n+1} \to \hat{H}^n \), and since \( e^{-tA} \) maps \( \hat{H}^n \) to \( \hat{H}^{n+2} \), this formulation is consistent. We now introduce a simplified formulation of the equation \eqref{Boussinesq-w1}, which will be frequently employed in subsequent developments, particularly for H\"{o}rmander-type Lie bracket computations. Setting
\begin{equation}\label{F(U)}
	F(U):=-AU-B(U,U)+GU,
\end{equation}
then equation \eqref{Boussinesq-w1} can be writen as
\begin{equation*}
	dU_t = F(U_t)\,dt + \sigma_\theta dW, \quad U_0 = U.
\end{equation*}

The following proposition delineates the well-posedness and ergodicity properties pertaining to the process $(U_t)$ that we utilize in this manuscript.
\begin{proposition} \emph{{}\cite[Proposition~2.2]{JFA}\emph{}}\label{base-posedness-ergodicity}
Fix $\nu_1,\nu_2>0,g\neq 0$, and a filtered probability space \( (\Omega, \mathcal{F}, \{\mathcal{F}_t\}_{t \geq 0},\mathbb{P}) \), the following holds for any $T>0$:
\begin{enumerate}[label=(\alph*), leftmargin=3em]
	\item For all functions \( U \in \hat{H}^n \) and with probability 1, there exists a unique mild solution \( U_t \in C\big([0, T]; \hat{H}^n\big) \) with \( U_0 = U \). As a function of the noise sample \( \omega \in \Omega \), the solution \( U_t \) is measurable, \( \mathcal{F}_t \)-adapted, and belongs to \( L^p\big(\Omega; C([0, T]; \hat{H}^n)\big) \) for all \( p \geq 1 \). Lastly, \( (U_t) \) itself is a Feller Markov process on \( \hat{H}^n \).
	\item The Markov process \( U_t \) admits a unique Borel stationary measure \( \mu \) in \( \hat{H}^n \).
\end{enumerate}
\end{proposition} 


Next, we introduce the definitions of the base process and the Markov process of interest considered in this manuscript. Here, the Markov property of these processes is ensured by verifying the independent increments hypothesis {\textbf{(H1)}} in Subsection \ref{rds}.
\begin{definition} \label{base-process-def}
The base process refers to the Markov process \(\{U_t\}_{t\geq 0}\) on \(\hat{H}^n\) which solves the (vorticity-form) incompressible stochastic Boussinesq equation \eqref{Boussinesq-w1} with initial data given by \(U_0\). As previously mentioned, we write $\mathbf{u}_t := \nabla^\perp \Delta^{-1} \omega_t$ to refer to the velocity field induced by the vorticity \(\omega_t\), and equivalently we may refer to \(\{(\mathbf{u}_t,\theta_t)\}_{t\geq 0}\) as the base process as well.
\end{definition}

\begin{definition} \label{Lagrange-process-def}
Given a geodesically complete Riemannian manifold \(({\Sigma},\hat{g})\) and an associated vector field map \(\Theta: H^{n+1}(\mathbb{T}^2;\mathbb{R}^2) \to \mathcal{X}_\mathrm{loc}^1(\Sigma)\) (where \(\mathcal{X}_\mathrm{loc}^1({\Sigma})\) denotes the set of \(C_{\mathrm{loc}}^{1}\) vector fields on \({\Sigma}\))-associating to any velocity field \(\mathbf{u}_t\) a vector field \(\Theta_{\mathbf{u}_t}\) on \(\Sigma\)-we define the associated Markov process \(\{\mathbf{u}_{t},\theta_t,\mathbf{p}_{t}\}_{t\geq 0}\) where \((\mathbf{u}_{t},\theta_t)\) is the base process and \(\mathbf{p}_{t}\) solves
\begin{equation} \label{Lagrange-process-eq}
	\dot{\mathbf{p}}_{t} = \Theta_{\mathbf{u}_{t}}(\mathbf{p}_{t}),
\end{equation}
with initial condition given by \(\mathbf{p}_{0} \in \Sigma\). There are several special processes\footnote{We note that one could simply pass all the results through the natural projection \(S^{1} \to R P^{1}\) and consider the `true' projective process on \(R P^{1}\) throughout.} that we will focus on listed in Table~1. 
\end{definition}
\begin{table}[ht]
	\centering
	\caption{Several associated Markov processes of interest}
	\label{tab:processes}
	\begin{tabular}{llll}
		\toprule
		\textbf{Process} & \textbf{Form} & \textbf{Manifold} & \textbf{Vector Field} \\
		\midrule
		Lagrangian & 
		\( (\mathbf{u}_t,\theta_t,\mathbf{x}_t) \)& 
		\( \Sigma^1 := \mathbb{T}^2 \) & 
		\( \Theta_{\mathbf{u}}^{1}(x) := \mathbf{u}(x) \)  \\
		
		
		\addlinespace[0.3em]
		Tangent & 
		\( (\mathbf{u}_t,\theta_t,\mathbf{x}_t,\tau_t)\)& 
		\( \Sigma^T := \mathbb{T}^2 \times (\mathbb{R}^2 \setminus \{0\}) \) & 
		\( \Theta_{\mathbf{u}}^{T}(x,\tau) := (\mathbf{u}(x),\tau \cdot \nabla \mathbf{u}(x)) \)  \\
		
		\addlinespace[0.3em]
		Projective & 
		\( (\mathbf{u}_t,\theta_t,\mathbf{x}_t,v_t)\)& 
		\( \Sigma^P := \mathbb{T}^2 \times S^1 \) & 
		\( \Theta_{\mathbf{u}}^{P}(x,v) := \left(\mathbf{u}(x),\ v \cdot \nabla \mathbf{u}(x) \cdot v^{\perp}v^{\perp}\right) \)  \\
		
		\addlinespace[0.3em]
		Jacobian & 
		\( (\mathbf{u}_t,\theta_t,\mathbf{x}_t,A_t)\) & 
		\( \Sigma^J := \mathbb{T}^2 \times \mathrm{SL}_2(\mathbb{R}) \) & 
		\( \Theta_{\mathbf{u}}^{J}(x,A) := (\mathbf{u}(x), A\nabla \mathbf{u}(x)) \) \\
		\bottomrule
	\end{tabular}
\end{table}

\begin{remark} \label{L-process-wp}
	Note that when \( n=4 \), the flow \( \mathbf{p}^t \) is a well-defined diffeomorphism since the velocity field \( \mathbf{u}_t \) belongs to ${H}^5(\mathbb{T}^2)$ (so it is at least \( C^2 \) by Sobolev embedding). This gives rise to an \( \mathcal{F}_t \)-adapted, Feller Markov process \( (\mathbf{u}_t,\theta_t, \mathbf{p}_t) \) on \( H^5(\mathbb{T}^2) \times H^5(\mathbb{T}^2) \times \Sigma\) defined by \eqref{Lagrange-process-eq}. Therefore, we only need \( n=4\).
 
	It is noteworthy, however, that while well-posedness and the Markov property require only $n=4$, the subsequent establishment of the asymptotic gradient estimates for the Markov semigroup (c.f. Theorem \ref{unit-time-n})--which is crucial for proving the ergodicity of the Lagrangian process--necessitates the initial data to belong to $\hat{H}^6(\mathbb{T}^2)$, i.e., $n=6$. Apart from this specific step, all other aspects of the analysis are developed and considered exclusively for $n=4$.
\end{remark}
\begin{remark} \label{conservative}
It is imperative to emphasize that the Lagrangian flow for the incompressible stochastic fluid equations is conservative. This fundamental fact underpins all subsequent investigations into Lagrangian chaos. While the proof is straightforward, we provide it below to ensure the mathematical rigor of this paper.

The Jacobian matrix $A_t$ satisfies $\dot{A}_{t} = D\mathbf{u}_t(\mathbf{x}_t)A_t$, furthermore, its determinant satisfies 
$$
\dot \det(A_{t}) = \det(A_{t}) \cdot tr(A_t^{-1}\dot A_{t})=\det(A_{t}) \cdot tr(A_t^{-1}\cdot D\mathbf{u}_t(\mathbf{x}_t)\cdot A_t)=\det(A_{t}) \cdot tr(D\mathbf{u}_t)=0.
$$
The last inequality employs the incompressibility condition $\nabla\cdot\mathbf{u}=0$, from which we obtain $\det(A_{t})=1$, implying that the Lagrangian flow is conservative.
\end{remark}

\subsection{Outline of the proof and contributions}
Below is a fairly detailed exposition of the proof of Theorem \ref{L1}, the main result of this paper. This exposition also highlights several contributions of the present manuscript. The proof of Theorem \ref{L1} is primarily inspired by \cite{JEMS, NS, re89}, and proceeds in three steps:\\
\textbf{Step1: Verification of framework for random dynamical systems and assumptions.} Classical well-posedness theory and regularity estimates (see e.g. \cite{JFA-71,JFA-99,long-7,long-3}) establish that the base, Lagrangian, projective, tangent and Jacobian processes are all Markovian, arise as continuous random dynamical systems in the framework of Section \ref{rds}, and the derivative cocycle associated with the Lagrangian flow admits an integrability condition \eqref{integrability}. Unlike the Navier-Stokes equations, establishing regularity estimates for the Boussinesq equation requires distinct weighted treatments for the vorticity and temperature equations.\\
\textbf{Step2: Ergodicity of the Lagrangian process and projective process.} We now focus on the ergodicity of processes $(\mathbf{u}_t,\theta_t,\mathbf{x}_t)$ and $(\mathbf{u}_t,\theta_t,\mathbf{x}_t,v_t)$ (c.f. Corollary \ref{erogdicty}). The existence of stationary measures follows from the compactness of manifolds $\Sigma^1$, $\Sigma^P$ and the super-Lyapunov property \eqref{Super-Lyapunov-property}, via the Krylov-Bogolyubov theorem.  For uniqueness, the degenerate nature of noise \eqref{noise-form} leads us to employ the asymptotically strong Feller framework established by Martin Hairer etc. \cite{HM06,HM11a}. 
Therefore, it suffices to prove that the extended system satisfies the time asymptotic gradient estimate and weak irreducibility required (c.f. Corollary \ref{long-time} and Proposition \ref{irreducibility}, respectively) by the framework above. By manifold compactness, the former reduces to establishing a unit-time asymptotic gradient estimate on the Markov semigroup (compared to \cite{JFA}, this yields substantial simplification—particularly in constructing the control process $v_{t}^{\beta}$, see \eqref{v11})—crucially contingent on obtaining a nondegenerate bound on a cone for the associated Malliavin matrix $\mathcal{M}_T$ (c.f. Theorem \ref{bounds-Malliavin}). The latter follows directly from an approximate controllability argument.

It should be emphasized that, unlike previous studies on the long-time behavior of the 2D stochastic Boussinesq equations (e.g. \cite{JFA,long-5,long-9} etc.), we need to additionally consider $\mathbf{x}_t$ and $v_t$ induced by the base process $(\mathbf{u}_t,\theta_t)$. Crucially, equation \eqref{Lagrange-process-eq} contains no explicit noise term—that is, no direct noise source—yet randomness propagates from the temperature $\theta_t$ (directly noise-driven) to the velocity field $\mathbf{u}_t$, and subsequently to the vector field $\mathbf{p}_t$ on manifold $\Sigma$, see equations \eqref{Boussinesq} and \eqref{Lagrange-process-eq}. Consequently, 
our case exhibits stronger degeneracy and requires further accounts for the geometric specifics on manifold $T\Sigma$.

The ergodicity of the Lagrangian process, combined with integrability conditions, guarantees the existence of Lyapunov exponents through the multiplicative ergodic theorem (c.f. Theorem \ref{MET}). Specifically, for $\mu^1$-almost every $(\mathbf{u}_0,\theta_0,\mathbf{x}_0)$, the limit $\lim\limits_{t \to \infty}\frac{1}{t}\mathop{\rm log}|D_x\mathbf{x}_t|=\lambda_+$ exists. Furthermore, the ergodicity of the projective process combined with the random multiplicative ergodic theorem  (c.f. \cite[Theorem~III.1.2]{PE-50}) yields a stronger result that holds $\mu^P$-almost surely. 
Here, $\mu^1$ and $\mu^P$ denote the unique stationary measures for the Lagrangian process and the projective process, respectively.\\
\textbf{Step3: Furstenberg’s criterion and approximate controllability.} Next, we establish the positivity of the top Lyapunov exponent $\lambda_+$. As noted earlier, this is typically proven via the Furstenberg's criterion. This criterion indicates that for conservative systems, an alternative scenario to the positivity of the top Lyapunov exponent is the existence of an almost surely invariant structure under the dynamics of the quadruplet $(\mathbf{u}_t,\theta_t,\mathbf{x}_t,A_t)$. Therefore, we require two conditions: (i) The family of invariant projective measures possesses suitable continuity; (ii) The system exhibits approximate controllability. The former enables the application and refinement of the Furstenberg's criterion, restricting possible invariant structures to two continuous types, while the latter precludes the existence of these two types to ensure the positivity of the top Lyapunov exponent. 

Regarding condition (i), continuity can be achieved through two distinct approaches: (a) Smoothing estimates of transition semigroups of the extended system (e.g., strong Feller property \cite{JEMS,PE}, asymptotic strong Feller property \cite{NS}, Hölder-type smoothing estimates \cite{V2}) for the Lagrangian and projective processes; (b) Weak continuity via compactness of the phase space and uniform mixing in total variation distance \cite{Bou88,BCZG23}. 

Analysis in Step 2 reveals that for the Boussinesq equations, the Lagrangian process 
$(\mathbf{u}_t,\theta_t,\mathbf{x}_t)$ and projective process $(\mathbf{u}_t,\theta_t,\mathbf{x}_t,v_t)$ satisfy merely the asymptotic strong Feller property and mix exclusively in the weaker dual-Lipschitz metric. We therefore adopt the framework of \cite{NS} to establish continuity of the family of invariant projective measures (c.f. Proposition \ref{continuous projective measures}), thereby deriving an analogous Furstenberg's criterion (c.f. Theorem \ref{Furstenberg}). Thus, only condition (ii) requires verification. We verify the approximate controllability condition (c.f. Definition \ref{def:controllability}) by constructing smooth controls, which completes the proof of Theorem \ref{L1}. Notably, compared with the literature considering only the process 
$(\mathbf{u}_t,\theta_t,\mathbf{x}_t)$ \cite{JFA} or the Navier-Stokes equations \cite{NS, V2}, the verification of controllability here becomes more subtle and intricate due to heightened degeneracy and increased equation complexity.

It is evident that the principal challenges in this work lie in establishing the probabilistic spectral bound on a cone for Malliavin matrix associated with the extended system and verifying its approximate controllability condition. We now outline these difficulties and our corresponding resolution strategies.

(1) 
Substantial modifications to the computations in \cite[Sections 5–6]{JFA} are required to adapt them to the extended system, with analysis conducted from both the $\hat{H}^n$-component and $T\Sigma$-component perspectives. We first state that for any \( \mathfrak{p} \in \hat{H}^n\times T_{\mathbf{p}_{\scriptscriptstyle T}}\Sigma \) and on the set \( \Omega^* \) with $\mathbb{P}({(\Omega^*)}^c) \leq C\exp(\eta V^{n+2}(U_0)) o_\epsilon(1)$ one has,
\begin{align}\label{JFA-2.18}
	\langle \mathcal{M}_T\mathfrak{p}, \mathfrak{p} \rangle \leq \epsilon \|\mathfrak{p}\|^2 \Longrightarrow
	\begin{cases}
		\max\limits_{|j|\le N, m\in \{0,1\}}\sup\limits_{t\in [T/2,T]}|\langle \mathcal{J}^*_{t,T}\mathfrak{p}, \sigma_j^m \rangle| \leq o_\epsilon(1) \|\mathfrak{p}\|,\\
		\max\limits_{|j|\le N, m\in \{0,1\}}|\langle  \psi_j^m + J_{j,m}^N(U_T),\mathfrak{p} \rangle| \leq o_\epsilon(1) \|\mathfrak{p}\|,
	\end{cases}
\end{align}
where $V^{n+2}(U_0)$ is the $n+2$-order version of the super Lyapunov function as in Definition \ref{Vndef}, $\mathcal{M}_T$ is the Malliavin matrix, $\mathcal{J}^*_{t,T}$ is the adjoint of the linearization of process $(U_t,\mathbf{p}_t)$, and $J_{j,m}^N(U_T)$ is defined as in \eqref{Jm}. Using \eqref{JFA-2.18}, we prove that with high probability the Sobolev norm of the $\hat{H}^n$-component of $\mathfrak{p}$ does not grow excessively, thereby controlling its {\textbf{overlap in the $\hat{H}^n$ direction}}. 
However, unlike the Navier-Stokes equations (c.f. \cite{NS}), it is crucial to note that the estimates \eqref{JFA-2.18} generated by H\"{o}rmander-type Lie brackets now explicitly depend on the stochastic process $U_T$. Consequently, we must carefully balance the control of the energy of $U_T$ with ensuring the validity of Proposition \ref{NS-prop5.18}.

Then, we also wish to establish an analogous implication relationship for $T\Sigma$-component. However, 
we can only obtain the following expression on the set $\Omega^*$:
\begin{align*}
	&\langle \mathcal{M}_T\mathfrak{p}, \mathfrak{p} \rangle \leq \epsilon \|\mathfrak{p}\|^2\\ &\Longrightarrow
	\max_{|j|\le N, m\in \{0,1\}}\left|\left\langle [F(\overline U_T),Z_j^m(\overline U_T)]-\Big(\Theta_{Z_j^m(\overline U_T)}(\mathbf{p}_T)+\Theta_{[U_T,Y_j^m(\overline U_T)]_x}(\mathbf{p}_T)\Big),\mathfrak{p}\right\rangle\right| \leq o_\epsilon(1) \|\mathfrak{p}\|,
\end{align*}
where $\overline{U}:=U-\sigma_\theta W$, $Y_j^m(U):=[F(U),\sigma_j^m]$ and $Z_j^m(U):=[F(U),Y_j^m(U)]$, $[\cdot,\cdot]$ and $[\cdot,\cdot]_x$ denote the Lie bracket operations with respect to $U$ and $x$ respectively, as explicitly defined in \eqref{U}. Using the overlap of $\mathfrak{p}$ in the $\hat{H}^n$ direction and Assumption \ref{assum5.5} (i.e. `solution-dependent manifold spanning condition')--$\Theta_{Z_j^m(\overline U_T)}+\Theta_{[U_T,Y_j^m(\overline U_T)]_x}$ uniformly spans $T\Sigma$, we can further control the {\bf{overlap of $\mathfrak{p}$ in the $T\Sigma$ direction}}. Consequently, it suffices to verify Assumption \ref{assum5.5}. We note that, unlike in reference \cite{NS} where only 
$\Theta_{e_k}$ (with $e_k$ being the real-valued Fourier basis) needed to be considered, the vector fields associated with the tangent bundle $T\Sigma$ here exhibit significantly higher complexity and strongly depend on the stochastic process $U_T$ satisfying equation \eqref{Boussinesq-w1}. Therefore, the spanning property holds only almost everywhere. Furthermore, this necessitates a more delicate treatment in selecting appropriate vector fields when verifying the spanning condition, while carefully balancing the selection of $N$ and the control of the bound on $U_T$.

(2) To construct the required smooth controls under the constraints of degenerate noise and the system's inherent complexity, we develop specialized flows to guide the system's evolution, adapting and extending the `shear flow' and `cellular flow' methods from \cite[Section 7]{JEMS} to the Boussinesq system. The resulting smooth control directly affects only the temperature equation, yet achieves controllability of the velocity equation through the coupling mechanism. Due to the equation structure, these controls inevitably possess spatial dependence—specifically, they take the form of $\sum_{\substack{k \in \hat{\mathcal{Z}}, l \in \{0,1\}}}\sigma_k^l(x)h_k^l(x,t)$. This constitutes a fundamental distinction from the spatially homogeneous $Qh(t)$-type controls employed in prior works \cite{NS, V2, PE}. When establishing the approximation properties essential for verifying approximate controllability, this spatial dependence introduces significant additional complexity (see Proposition \ref{V2-thm5.2}).

Finally, to facilitate focused application of the developed techniques to Lagrangian chaos studies in stochastic incompressible fluid models—such as stochastic incompressible magnetohydrodynamics (MHD) equations driven by degenerate noise—we formalize generalizable structural properties while concretizing some existing conditions.
\begin{itemize}
\item The base, Lagrangian, projective, tangent and Jacobian processes are all almost surely globally well-posed and Markovian, arising as continuous random dynamical systems.  Moreover, the derivative cocycle associated with the Lagrangian process admits an integrability condition. (c.f. \eqref{integrability}).
\item The Lagrangian and projective processes are  asymptotic strong Feller and weak irreducible. For the former specifically, this requires: the essential bounds on the dynamics of the extended system, a `generalized Hörmander condition' and a `solution-dependent manifold spanning condition' (c.f. Section \ref{smoothing estimate} for more details).
\item The projective and Jacobian processes exhibit approximate control condition. Specifically for fluid models with Euler-type nonlinearity (e.g. Navier-Stokes equations, primitive equations and Boussinesq equations etc.), approximate controllability is achieved by constructing smooth controls via shear and cellular flows.
\end{itemize}

The manuscript is organized as follows: In Section \ref{SEC2}, we review fundamental concepts in random dynamical systems theory, establishes the key Furstenberg's criterion, and reduces the proof of the main theorem to establishing the ergodicity of the Lagrangian and projective processes and verifying the approximate controllability of the extended system. Section \ref{SEC3} provides necessary a priori moment estimates for the equations and simplifies the proof of ergodicity for the Lagrangian and projective processes to verifying weak irreducibility and establishing the unit-time asymptotic gradient estimate. In Section \ref{smoothing estimate}, we reduce the unit-time asymptotic gradient estimate to establishing the probabilistic spectral bound on a cone for the Malliavin matrix via a control problem, demonstrating that these bounds hold under Assumptions \ref{assum5.4}–\ref{assum5.5}. In Section \ref{var-assum5.4-5}, we verify that the extended system satisfies Assumptions \ref{assum5.4}–\ref{assum5.5}. Section \ref{Approximate controllability} verifies the nonlinear approximate controllability of the extended system associated with the Boussinesq system \eqref{Boussinesq} and thereby proves the weak irreducibility of the Lagrangian and projective processes. The Appendix supplies proofs for the a priori moment estimates required in Section \ref{SEC3} and selected results used in the paper.

\section{Fundamental mathematical framework and main results}\label{SEC2}
In this section, we primarily introduce the random dynamical systems (RDS) framework, which is the principal abstract framework underpinning our research. Here, we focus specifically on RDS equipped with the independent increments hypothesis and linear cocycles satisfying an integrability condition. We also present the Multiplicative Ergodic Theorem and Furstenberg's criterion within this framework. Detailed proofs for many specific results can be found in the references \cite{Kif86,KS12,Ose68,JEMS}. 
\subsection{The RDS framework and the Multiplicative Ergodic Theorem}
As previously stated, the proof of Theorem \ref{L1} requires establishing two key properties:\\
(A) that the limit defining the Lyapunov exponent $\lambda_+$ exists and is constant almost
surely,\\
(B) that this exponent satisfies $\lambda_+>0$.

We first employ tools from random dynamical systems theory to prove (A). To this end, this section briefly reviews foundational concepts and results in RDS relevant to our work. Specifically, we verify that the base, Lagrangian, projective, and tangent processes are all continuous random dynamical systems, and that they satisfy the required independent increments hypothesis and integrability condition.
\subsubsection{Basic setup for random dynamics system}\label{rds}
Let $(\Omega, \mathcal{F}, \mathbb{P})$ be a probability space and let $(\theta^t)$ be a measure-preserving semiflow on $\Omega$, i.e., $\theta: [0, \infty) \times \Omega \to \Omega$, $(t, w) \mapsto \theta^t w$ is a measurable mapping satisfying
\begin{itemize}
	\item[(i)] $\theta^0 w \equiv w$ for all $w \in \Omega$,
	\item[(ii)] $\theta^t \circ \theta^s = \theta^{t+s}$ for all $s, t \geq 0$, and
	\item[(iii)] $\mathbb{P} \circ (\theta^t)^{-1} = \mathbb{P}$ for all $t \geq 0$.
\end{itemize}


Suppose that $\Omega$ is a Borel subset of a Polish space, and $\mathcal{F}$ is the set of Borel subsets of $\Omega$. Let $(Z, d)$ be a separable, complete metric space. A \emph{random dynamical system} or RDS on $Z$ is an assignment to each $w \in \Omega$ of a mapping $\mathcal{T}_w: [0, \infty) \times Z \to Z$ satisfying the following basic properties:
\begin{itemize}
	\item[(i)] (Measurability) The mapping $\mathcal{T}: [0, \infty) \times \Omega \times Z \to Z$, $(t, w, z) \mapsto \mathcal{T}^t_w z$, is measurable with respect to $\text{Bor}([0, \infty)) \otimes \mathcal{F} \otimes \text{Bor}(Z)$ and $\text{Bor}(Z)$.
	\item[(ii)] (Cocycle property) For all $w \in \Omega$, we have $\mathcal{T}^0_w = \text{Id}_Z$ (the identity mapping on $Z$), and for $s, t \geq 0$, we have $\mathcal{T}^{s+t}_w = \mathcal{T}^t_{\theta^s w} \circ \mathcal{T}^s_w$.
	\item[(iii)] (Continuity) For all elements $w \in \Omega$, the mapping $\mathcal{T}_w: [0, \infty) \times Z \to Z$ belongs to $C_{u,b}([0, \infty) \times Z, Z)$.
\end{itemize}

Here, for metric spaces \( V, W \), space \( C_{u,b}(V, W) \) is the the space of continuous maps $E: V\to W$ such that for every bounded set $O\subseteq V$, both the restriction $E|_{O}$ is uniformly continuous and the image $E(O)$ is a bounded subset of $W$ hold.

\begin{definition} \label{CRDS-def}
We refer to \(\mathcal{T}\) satisfying (i)-(iii) above as a \emph{continuous RDS} on \(Z\).
\end{definition}

In addition to (i)-(iii) above, we will almost always assume that the RDS \(\mathcal{T}\) satisfies the usual \emph{independent increments assumption}.\\
{\textbf{(H1)}} For all \( s, t > 0 \), we have that \(\mathcal{T}_{w}^{t}\) is independent of \(\mathcal{T}_{\theta^t w}^s\). That is, the \(\sigma\)-subalgebra \(\sigma(\mathcal{T}_{w}^{t}) \subset \mathcal{F}\) generated by the \(C_{u,b}(Z, Z)\)-valued random variable \(w \mapsto \mathcal{T}_{w}^{t}\) is independent of the \(\sigma\)-subalgebra \(\sigma(\mathcal{T}_{\theta^t w}^s)\) generated by \(w \mapsto \mathcal{T}_{\theta^t w}^s\).
\begin{remark} \label{L-process-RDS}
Based on the well-posedness results and ergodicity results in \cite{JFA}, and following the approach elaborated in \cite[Section A.1]{JEMS}, we obtain the corresponding mapping $\mathcal{U}:[0,\infty)\times \Omega\times\hat{H}^4\to\hat{H}^4, (t,w,U)\mapsto \mathcal{U}_w^t(U)$ for the base process $(U_t)$ is a continuous random dynamical systems satisfying the independent increment assumption, the corresponding mapping $\mathfrak{P}:[0,\infty)\times \Omega\times\hat{H}^4\times \mathbb{T}^2\to\hat{H}^4\times \mathbb{T}^2$ for the Lagrangian process $(U_t,\mathbf{x}_t)$ also is a continuous random dynamical systems satisfying the independent increment assumption, with analogous results holding for projective process, tangent process and jacobian process.
\end{remark}
We write \( P_t : \hat{H}^4 \times \Sigma \to \mathcal{P}(\hat{H}^4 \times \Sigma) \) to denote the time \( t \) transition kernel, where \( \mathcal{P}(X) \) denotes the space of probability measures on \( X \). We also let \( P_t \) act adjointly on observables \( \varphi : \hat{H}^4 \times \Sigma \to X \) by pulling back:
\[
P_t \varphi(U_0,\mathbf{p}_0) = \int \varphi(U_t,\mathbf{p}_t) \, P_t(U_0,\mathbf{p}_0; \mathrm{d}U_t,\mathrm{d}\mathbf{p}_t).
\]

\subsubsection{Linear cocycles over RDS and the Multiplicative Ergodic Theorem}
In this section, we formalize the concept of linear cocycles over random dynamical systems and rigorously state the multiplicative ergodic theorem (Theorem \ref{MET}), which provides the foundational framework for establishing the proof of \emph{property (A)}.
\begin{definition} \label{cocycle-def}
Let \(\mathcal{T}\) be a continuous RDS as in Section \ref{rds}, referred to below as the base RDS, and let \((\tau^t)\) be its associated skew product. A \(d\)-dimensional linear cocycle \(\mathcal{A}\) over the base RDS \(\mathcal{T}\) is a mapping \(\mathcal{A}:\Omega \to C_{u,b}([0,\infty) \times Z,\, M_{d \times d}(\mathbb{R}))\) with the following properties:
\begin{itemize}
	\item[(i)] The evaluation mapping \(\Omega \times [0,\infty) \times Z \to M_{d \times d}(\mathbb{R})\) sending \((\omega,t,z) \mapsto \mathcal{A}_{w,z}^t\) is \(\mathcal{F} \otimes \text{Bor}([0,\infty)) \otimes \text{Bor}(Z)\)-measurable.
	\item[(ii)] The mapping \(\mathcal{A}\) satisfies the cocycle property: for any \(z \in Z, w \in \Omega\) we have \(\mathcal{A}_{w,z}^0 = \mathrm{Id}_{\mathbb{R}^d}\), the \(d \times d\) identity matrix, and for \(s,t \geq 0\) we have
	\[
	\mathcal{A}_{w,z}^{s+t} = \mathcal{A}_{\tau^t(w,z)}^s \circ \mathcal{A}_{w,z}^t.
	\]
\end{itemize}
\end{definition}

The following is a version of the multiplicative ergodic theorem (MET) in \cite{Ose68}.
\begin{theorem}\label{MET}
Let \(\mathcal{T}\) be a continuous RDS as in Section \ref{rds} satisfying independent increments assumption. Let \(\mu \in \mathcal{P}(Z)\) be an ergodic stationary measure for \(\mathcal{T}\), and let \(\mathcal{A}\) be a linear cocycle over \(\mathcal{T}\) satisfying the following integrability condition: 
\begin{equation} \label{integrability}
\mathbb{E} \int_{Z} \left( \sup_{0 \leq t \leq 1} \log^+ |\mathcal{A}_{w,z}^t| \right) d\mu(z),\ 
\mathbb{E} \int_{Z} \left( \sup_{0 \leq t \leq 1} \log^+ |(\mathcal{A}_{w,z}^t)^{-1}| \right) d\mu(z) < \infty,
\end{equation}
where \(\log^+(x) := \max\{0, \log(x)\}\) for \(x > 0\) and \(\mathbb{E}\) is the expectation with respect to \(\mathbb{P}\). Then there are \(r \in \{1, \ldots, d\}\) deterministic real numbers
\[
\lambda_r < \cdots < \lambda_1
\]
and for \(\mathbb{P} \times \mu\)-a.e. \((w, z)\), a flag of subspaces
\[
\{0\} =: F_{r+1} \subsetneq F_r(w, z) \subsetneq \cdots \subsetneq F_2(w, z) \subsetneq F_1 := \mathbb{R}^d
\]
such that
\[
\lambda_i = \lim_{t \to \infty} \frac{1}{t} \log |\mathcal{A}_{w,z}^t v|, \quad v \in F_i(w, z) \setminus F_{i+1}(w, z).
\]
Moreover, for any \(i \in \{1, \ldots, r\}\), the mapping \((w, z) \mapsto F_i(w, z)\) is measurable and \(\dim F_i(w, z)\) is constant for \(\mathbb{P} \times \mu\)-a.e. \((w, z)\). 
\end{theorem}

In particular, \(\lambda_1\) in the above theorem is precisely the top Lyapunov exponent of interest. By Kingman’s subadditive ergodic theorem (see \cite{Kin73}), \(\lambda_1\) equals \(\lambda_+\) defined by the limit in \eqref{lambda+}, and  
\[
\lambda_{\Sigma} := \sum_{i=1}^{r} m_i \lambda_i = \lim_{t \to \infty} \frac{1}{t} \log \left|\det\left(\mathcal{A}_{w,z}^t\right)\right|.
\]  
Hence, \(\lambda_{\Sigma} = 0\) if \(\det\left(\mathcal{A}_{w,z}^t\right) = 1\). In particular, for the Lagrangian flow considered in this paper, we have \(\lambda_{\Sigma} = 0\).

\begin{remark} \label{tangent-process-cocycle}
	The cocycle \( \mathcal{A}^t_{w, z} := D_z\mathcal{T}^t_w \), \( z \in Z \), \( t \geq 0 \), is often referred to as the \emph{derivative cocycle} for \(\mathcal{T}\). Building on the theoretical framework established above, in this paper we investigate the linear cocycle
	\[
	\mathcal{A}: [0, \infty) \times \Omega \times \hat{H}^4 \times \mathbb{T}^2 \rightarrow M_{2 \times 2}(\mathbb{R})
	\]
	defined by  
	\[
	\mathcal{A}_{w, U, x}^t = D_x \mathbf{x}_{\scriptstyle{w, U}}^t.
	\]  
	For the Boussinesq equations, moment estimates analogous to those established in \cite{JEMS, JEMS-38, KS12} can be similarly established. Consequently, the derivative cocycle associated with the Lagrangian process satisfies the integrability condition \eqref{integrability}. Furthermore, the independent increments condition for the projective RDS asscoiated to the cocycle $\mathcal{A}$ is equivalent to condition (H1) for the \((U_t, x_t, A_t)\) process.
\end{remark}

\begin{remark} \label{L-process-ergodic}
If the Lagrangian process admits an unique stationary measure $\mu^1$, then given that the integrability condition holds, it follows from the MET that the top Lyapunov exponent of this linear cocycle exists and is $\mathbb{P}\times\mu^1$ almost surely constant. \footnote{Indeed, for a random dynamical system $\mathcal{T}$ satisfying the independent increments assumption (H1), if $\mu$ is the stationary measure of the transition semigroup $P_t$, then $\mathbb{P}\times \mu$ is an invariant measure of the skew product $\tau_t$ associated with $\mathcal{T}$. Furthermore, if the semigroup admits a unique invariant measure $\mu$, then $\mathbb{P}\times \mu$ is an ergodic invariant measure of the random dynamical system $\mathcal{T}$.} Therefore, to verify property (A), it suffices to prove that the Lagrangian process admits a unique stationary measure, which will be established in Corollary \ref{erogdicty}.
\end{remark}

\subsection{Positivity of the top Lyapunov exponent}
Now, we focus in this subsection on verifying \emph{property (B)} using Furstenberg's criterion, achieved by ruling out invariant structures for the projective measure that possess certain continuity.
\subsubsection{Furstenberg’s criterion with families of unconditionally continuous measures}
We begin by presenting a standard version of the Furstenberg’s criterion, directly adapted from \cite{Led86}:
\begin{theorem}
If \(\lambda_1 = \lambda_r\), then for each \(z \in Z\) there is a Borel measure \(\nu_z\) on \(P^{1}\) such that
\begin{enumerate}
    \item[(i)] the assignment \(z \mapsto \nu_z\) is measurable, and
    \item[(ii)] for each \(t \in [0, \infty)\) and \((\mathbb{P} \times \mu)\)-almost all \((w, z) \in \Omega \times Z\) (perhaps depending on \(t\)), we have that
    \begin{equation}\label{invariant-measure}
    (\mathcal{A}_{w,z}^t)_*\nu_z = \nu_{\mathcal{T}_{w,z}^t}.
    \end{equation}
\end{enumerate}
\end{theorem}
Although the classical Furstenberg criterion implies that the positivity of the top Lyapunov exponent would follow by simply ruling out any family of projective measures \(\{\nu_z\}_{z \in \operatorname{supp} \mu}\) satisfying the invariant structure \eqref{invariant-measure}, the mere measurability of such a family is insufficient to exclude these invariant structures.
\cite[Section 4]{Bou88} and \cite[Proposition 2.10]{BCZG23} circumvent this by imposing a stronger assumption-namely, that the stationary measure \(\mu\) is mixing under the total variation norm-which enables the selection of a weakly continuous family of measures \(\{\nu_z\}\). However, these results do not extend to Lagrangian trajectories governed by the Boussinesq system, where degenerate noise induces mixing only under the weaker dual-Lipschitz metric. To address this limitation, we employ a refined version of the Furstenberg's criterion proposed in \cite[Proposition 3.4]{NS}. 
\begin{proposition}\emph{\cite[Proposition 3.1]{NS}}\label{NS-prop3.1}
Let \( P(z, dy) \) be a Markov kernel on a Radon space \((Z, \mathcal{M}, \mu)\) where \(\mu\) is stationary under \(P\), that is,
\[
\int P(z, dy) \, \mu(dz) = \mu(dy).
\]
Let \(\mathcal{A}\) be a cocycle on \(Z\), identified with a measurable \(\mathcal{A}: Z \to \mathrm{SL}_2(\mathbb{R})\). We also write \(\mathcal{A}_z\) to denote the induced map \(\mathcal{A}_z: S^{1} \to S^{1}\) given by \(\mathcal{A}_z: \nu \mapsto \frac{\mathcal{A}(\nu)}{|\mathcal{A}(\nu)|}\).

Suppose \(\nu: Z \to \mathcal{P}(S^{1})\) is a measurable family of probability measures on the sphere \(S^{1}\) satisfying the twisted pullback condition
\[
\nu_z = (\mathcal{A}(z)^{-1})_* \int \nu_y \, P(z, dy),
\]
where \(\mathcal{A}_*\nu\) denotes the pushforward measure of \(\nu\) under the map \(\mathcal{A}\).

Let \(\lambda_+\) be the top Lyapunov exponent of the cocycle, then
\[
\lambda_+ :=\lim_{n \to \infty} \frac{1}{n} \mathbb{E} \log \|\mathcal{A}(z_{n-1}) \cdots \mathcal{A}(z_1)\mathcal{A}(z_0)\|,
\]
where \(z_0\) is distributed according to \(\mu\) and the remaining \(z_j\) are distributed according to the trajectory of the Markov process started from \(z_0\).  If $\lambda_+=0$, then the family is almost surely invariant:
$$
\mathcal{A}(z_{0})_*\nu_{z_0}=\nu_{z_1}, \,\,\, P(z_0,dz_1)\mu(dz_0)-a.s..
$$
\end{proposition}

Analogous to \cite[Proposition 3.3]{NS} by using the priori estimates \eqref{Super-Lyapunov}-\eqref{C-alpha} for the Boussinesq equations and the unit-time asymptotic gradient estimate for the Lagrangian and tangent processes (c.f. Proposition \ref{unit-time}), we derive a family of unconditionally continuous projective measures.

\begin{proposition}\label{continuous projective measures}
	There exists a locally Lipschitz family of probability measures
	\[
	\nu : \hat{H}^4 \times \Sigma^{J} \to \mathcal{P}(S^1)
	\]
	such that
	\begin{align}\label{projective measures}
	\nu_{\scriptscriptstyle{U_0,\mathbf{x}_0,A_0}} = (A_0)^{-1}_* \int \nu_{\scriptscriptstyle{U,\mathbf{x},A}} \, P_1(U_0, \mathbf{x}_0, I_d, dU, d\mathbf{x}, dA),
	\end{align}
	where we metrize \(\mathcal{P}(S^1)\) with the \(W^{-1,1}\) norm.
\end{proposition}

Then combining with the cocycle property of the Jacobian process, we obtain:
\begin{theorem}[Furstenberg's criterion]\label{Furstenberg}
Consider the Jacobian process \(\{(U_t, \mathbf{x}_t, A_t)\}_{t \geq 0}\) with initial data \((U_0, \mathbf{x}_0)\) distributed according to the stationary measure of the Lagrangian process and \(A_0 = I_d\). If the top Lyapunov exponent  
	\[
	\lambda_+ := \lim_{t \to \infty} t^{-1} \mathbb{E}\left[\log \|A_t\|\right]
	\]
	is zero, then there is a continuous map \(\nu: \hat{H}^4 \times \Sigma^1 \to W^{-1,1}(S^{1})\) satisfying the invariance formula  
	\begin{align}\label{NS-3.16}
	(A_t)_*\nu_{\scriptscriptstyle{U_t,\mathbf{x}_t}} = \nu_{\scriptscriptstyle{U_0,\mathbf{x}_0}} 
	\end{align}
	\(\mu^1\)-almost surely for \(t \in \mathbb{N}\), where \(\mu^1\) is the unique stationary measure for the Lagrangian process.
\end{theorem}

\subsubsection{Contradicting Furstenberg’s criterion}
This section directly follows the arguments of \cite[Sections 4.3 and 7]{BBPS22a}, and using our Theorem \ref{Furstenberg}
in place of the strong Feller assumption. We begin by citing a highly useful result that establishes two types of continuity structures for the linear cocycle $A$ preserving the family of invariant measures 
$\nu$ in Proposition \ref{continuous projective measures}. This refined result allows for the exclusion of invariant structures directly via the approximate controllability condition (c.f. Definition \ref{def:controllability}), thereby yielding the positivity of the top Lyapunov exponent.

\begin{theorem} \emph{{(Classification of invariant measure families, \cite[Theorem 4.7]{JEMS})}}
Let $(U,\mathbf{x}) \mapsto \nu_{\scriptscriptstyle{U,\mathbf{x}}}$ be the continuous family of probability measures satisfying the invariance formula \eqref{NS-3.16} almost surely. Then one of the following alternatives holds.
\begin{enumerate}
	\item[(a)] There is a continuously-varying inner product $\langle\cdot,\cdot\rangle_{\scriptscriptstyle{U,\mathbf{x}}}$ on $\mathbb{R}^2$ with the property that for any $(U_0,\mathbf{x}_0) \in \hat{H}^4 \times \mathbb{T}^2$ and $A_0 := I_d$, the map $A_t : (\mathbb{R}^2,\langle\cdot,\cdot\rangle_{\scriptscriptstyle{U_0,\mathbf{x}_0}}) \to (\mathbb{R}^2,\langle\cdot,\cdot\rangle_{\scriptscriptstyle{U_t,\mathbf{x}_t}})$ is an isometry almost surely, where $(U_t,\mathbf{x}_t,A_t)$ is distributed according to the transition probability of the Jacobian process.
	\item[(b)] For some $p \in \mathbb{N}$, there are measurably-varying assignments from $z_0 \in \hat{H}^4 \times \mathbb{T}^2$ to one-dimensional linear subspaces $E^i_{z_0} \subsetneq \mathbb{R}^2$, for each $1 \leq i \leq p$, with the property that, if $z_t$ is distributed according to the transition measure $P^J_t(z_0)$, then $A_t E^i_{z_0} = E^{\pi(i)}_{z_t}$ for some permutation $\pi$. Furthermore, the collection $(E^i_z)$ is locally (on small neighborhoods) continuous up to re-labelling.
\end{enumerate}
\end{theorem}
We now state the main result of this section, Theorem \ref{posivie-lambda}, and complete the proof of the paper's central conclusion: the top Lyapunov exponent associated with the Lagrangian flow $\mathbf{x}_t$ is positive.
\begin{definition}\label{def:controllability}
	We say that the cocycle \( \mathcal{A} \) satisfies the \textit{approximate controllability condition} (C) if there exist \( z, z' \in \operatorname{supp} \mu \) such that \( z' \) belongs to the support of the measure \( P_{t_0}(z, \cdot) \) for some \( t_0 > 0 \), and we have each of the following for all $t>0$.
	\begin{enumerate}
		\item[(a)] For any \(\mathbf{x}\in \mathbb{T}^2,\,\varepsilon, M > 0\),
		\[
		\mathbb{P}\left((U_t,\mathbf{x}_t,A_t)\in B_\varepsilon(z') \times B_\varepsilon(\mathbf{x}')\times \{A \in \mathrm{SL}_2(\mathbb{R}) : |A| > M\}|(U_0,\mathbf{x}_0,A_0)=(z,\mathbf{x},\mathrm{Id}) \right)> 0.
		\]
		\item[(b)] For any \( (\mathbf{x},v) \in \mathbb{T}^2\times S^{1} \), open set \( V \subset S^{1} \), and \(\varepsilon > 0\),
		\[
		\mathbb{P}\left((U_t,\mathbf{x}_t,v_t)\in B_\varepsilon(z') \times B_\varepsilon(\mathbf{x}') \times V |(U_0,\mathbf{x}_0,v_0)=(z,\mathbf{x},v)\right)> 0.
		\]
	\end{enumerate}
\end{definition}

\begin{theorem}\label{posivie-lambda}
	Let $\mu$ be an ergodic stationary measures
	for which the approximate controllability condition (C) holds. Then there exists \(\lambda_+ > 0\) such that, for every initial \((U_0, x, v_0) \in \hat{H}^4 \times \Sigma^P\) and \(A_0 := \mathrm{Id}\),
	\begin{equation}\label{lambda>0}
		\lambda_+ = \lim_{t \to \infty} t^{-1} \mathbb{E}\left[\log |A_t|\right] \nonumber \\
		= \lim_{t \to \infty} t^{-1} \mathbb{E} \left[ \int_0^t v_s \cdot \nabla \mathbf{u}_s(\mathbf{x}_s) v_s \, ds \right]. 
	\end{equation}
\end{theorem}
\begin{proof}
	The proof follows an analogous approach to that of \cite[Theorem 3.6]{NS}.
\end{proof}
\begin{proof}[Proof of Theorem \ref{L1}]
In Section \ref{Approximate controllability}, we complete the verification of approximate controllability condition $(C)$, which, combined with Theorem \ref{posivie-lambda}, immediately yields Theorem \ref{L1}.
\end{proof}
\begin{proof}[Proof of Remark \ref{1.2}]
	By utilizing the existence of a unique stationary measure for the projective process (Corollary \ref{erogdicty}) and the random multiplicative Ergodic Theorem ( \cite[Theorem~III.1.2]{PE-50}), we can complete the proof.
\end{proof}
\section{Priori moment estimates and ergodicity}\label{SEC3}
In this section, we present the a priori moment estimates and the super-Lyapunov function for the stochastic Boussinesq equations, which are frequently employed throughout the manuscript. Detailed proofs of these results are deferred to Appendix \ref{priori}. Furthermore, this section establishes ergodicity results for the Markov processes of interest (c.f. Table 1), thereby ensuring the validity of property (A). 


\subsection{Priori moment estimates and super-Lyapunov function for stochastic Boussinesq equations}

\begin{proposition}\label{priori-eatimates}
	Suppose that \(U_t\) is a solution to \eqref{Boussinesq-w1}, $\kappa:=min\left\{\nu_1,\nu_2\right\}$. Then there exists \(C > 0\) such that for all \(\eta \leq C^{-1}\) and \(0 \leq s \leq t\),
	\begin{equation}\label{L2-eatimates}
		\mathbb{E} \exp \left( \eta \sup_{s \leq r \leq t} \left( \| U_r \|_{\hat{H}}^2 + \frac{\kappa}{2} \int_s^r \| U_a \|_{\hat{H}^1}^2 \, da \right) \right) \leq Ce^{C(t-s)} \exp \left( e^{-C^{-1}s} \eta \| U_0 \|_{\hat{H}}^2 \right).
	\end{equation}
	Additionally, there exists \(C(n) > 0\) such that for all \(0 \leq \eta \leq C^{-1}\) and \(0 \leq s \leq t\),
	\begin{align}\label{Hn-eatimates1}
		\mathbb{E} \exp \left( \eta \sup_{s \leq r \leq t} \left( \| U_r \|_{\hat{H}^n}^2 + \frac{\kappa}{2} \int_s^r \| U_a \|_{\hat{H}^{n+1}}^2 \, da \right)^{\frac{1}{n+2}} \right) 
		\leq Ce^{Ct} \exp \left( e^{-C^{-1}s} \eta \| U_0 \|_{\hat{H}^n}^{\frac{2}{n+2}} + C \eta \| U_0 \|_{\hat{H}}^2 \right),
	\end{align}
	and
	\begin{equation}\label{Hn-eatimates2}
		\mathbb{E} \exp \left( \eta \sup_{s \leq r \leq t} \left( \| U_r \|_{\hat{H}^n}^2 + \frac{\kappa}{2} \int_s^r \| U_a \|_{\hat{H}^{n+1}}^2 \, da \right)^{\frac{1}{n+2}} \right) \leq Ce^{Cs^{-1}} e^{Ct} \exp \left( C \eta \| U_0 \|_{\hat{H}}^2 \right).
	\end{equation}
\end{proposition}

\begin{definition}[Super-Lyapunov function for the base process]\label{Vdef}
	For \(U \in \hat{H}^4\), we define
	\begin{equation}\label{Super-Lyapunov}
		V(U) := \iota \left( \| U \|_{\hat{H}}^2 + \delta \| U \|_{\hat{H}^4}^{1/3} \right),
	\end{equation}
	where \(\delta, \iota > 0\) are some fixed small constants, chosen according to Proposition \ref{priori-eatimates} so that Corollary \ref{Super} holds for all \(\eta \in (0, 2)\).
\end{definition}
Following the proof in \cite[Corollaries 2.5-2.6]{NS}, we can similarly establish two useful corollaries toward the proof of Proposition \ref{continuous projective measures}.
\begin{corollary}[Super-Lyapunov property]\label{Super}
	There exists \(\beta \in (0,1)\) and a constant \(C > 0\) such that for all \(\eta \in (0,2)\) and \(t \geq 1\),
	\begin{equation}\label{Super-Lyapunov-property}
		\mathbb{E} \exp(\eta V(U_t)) \leq C \exp(\beta \eta V(U_0)).
	\end{equation}
\end{corollary}

\begin{corollary}[Exponential moments of \( C^\alpha \) norms]
	There exists \(\alpha > 0\) such that for all \(K, \eta, t > 0\) there exists \(C(K, \eta, t) > 0\) such that
	\begin{equation}\label{C-alpha}
		\mathbb{E} \exp \left( K \int_0^t \|\nabla u_s\|_{C^\alpha} \, ds \right) \leq C \exp \left( \eta V(U_0) \right).
	\end{equation}
\end{corollary}
We now present an arbitrary-order version \( V^n \) of the super Lyapunov function \( V \). Noted that \( V^n \) remains essentially identical to \( V \), and considering arbitrary orders does not significantly increase the complexity while slightly enhancing the generality of the framework.
\begin{definition}\label{Vndef}
	Let
	\[
	V^n(U) := \iota_n \left( \| U \|_{\hat{H}}^2 + \| U \|_{\hat{H}^n}^{\frac{2}{n+2}} \right),
	\]
	where \(\iota_n > 0\) is chosen so that we can apply Proposition \ref{priori-eatimates} to
	\[
	e^{\iota_n \eta \| U \|_{\hat{H}}^2} \quad \text{and} \quad e^{\iota_n \eta \| U \|_{\hat{H}^n}^{\frac{2}{n+2}}}
	\]
	for all \(\eta \leq 2\).
\end{definition}

\begin{remark}
	Note that the higher-order version of \( V_n \) controls its lower-order counterparts in the sense that for any \( 1 \leq k \leq n \), there exists \( C(k,n) \) such that $V^k \leq C V^n$.
\end{remark}
\subsection{Ergodicity of the Markov processes}
We have established that to prove (A), namely the existence of the top Lyapunov exponent $\lambda_+$, it is necessary to demonstrate the existence and uniqueness of stationary measures for the base process, Lagrangian process, and projective process. Furthermore, to use Furstenberg's criterion (Theorem \ref{Furstenberg}), the asymptotic gradient estimate for the semigroup over unit-time (Proposition \ref{unit-time}) is also necessary.
	
In this subsection, we state Proposition \ref{unit-time} and Corollary \ref{long-time}, which provide key regularity estimates for the processes under study. In particular, Corollary \ref{long-time} serves as a sufficient condition for the asymptotic strong Feller property introduced in \cite{HM06}, from which the uniqueness of the stationary probability measure for the corresponding process follows directly. Proposition \ref{unit-time} is a special case of Theorem \ref{unit-time-n}, which generalizes \cite[Proposition 2.6]{JFA} to the setting of stochastic Boussinesq equations coupled with vector fields on manifolds.
We note that below we take norms of \(\nabla \varphi (U, \mathbf{p}) \in \hat{H}^4\times T\Sigma\), denoted by \(\| \nabla \varphi (U, \mathbf{p}) \|_{\hat{H}^4}\). By this, we mean that we take the sum of the \(\hat{H}^4\) norm in the \(\hat{H}^4\) component and the norm induced by the Riemannian metric in the \(T\Sigma\) component.
\begin{proposition}[unit-time asymptotic gradient estimate]\label{unit-time}
Consider either the base, Lagrangian, tangent, projective, or Jacobian process. Then for each \(\eta, \gamma \in (0,1)\) there is \(C(\mathbf{p}, \eta, \gamma) > 0\), locally bounded in \(\mathbf{p} \in \Sigma\), such that for each Fréchet differentiable \(\varphi : \hat{H}^4 \times \Sigma \to \mathbb{R}\),
\begin{equation}\label{unit-time-eq}
\| \nabla P_1 \varphi(U, \mathbf{p}) \|_{\hat{H}^4} \leq \exp(\eta V^6(U)) \left( C \sqrt{P_1 |\varphi|^2 (U, \mathbf{p})} + \gamma \sqrt{P_1 \| \nabla \varphi \|^2_{\hat{H}^4}(U, \mathbf{p})} \right).
\end{equation}
\end{proposition}
We will discuss the proof of this proposition in next subsection and complete the proof in Section \ref{smoothing estimate}. Moreover, when \(\Sigma\) is compact, the super-Lyapunov property enables us to directly derive the following time asymptotic gradient estimate from the unit-time asymptotic gradient estimate in Proposition \ref{unit-time}, following a proof strategy analogous to \cite[Lemma 2.10]{NS}.
\begin{corollary}[time asymptotic gradient estimate] \label{long-time}
Consider either the base, Lagrangian, or projective process. For every \(\eta, \gamma \in (0,1)\) there is \(C(\eta, \gamma) > 0\) such that for each Fréchet differentiable observable \(\varphi: \hat{H}^4\times \Sigma \to \mathbb{R}\) and \(t \geq 1\),
\begin{equation}
	\|\nabla P_t \varphi\|_{\eta V} \leq C \|\varphi\|_{\eta V} + \gamma^t \|\nabla \varphi\|_{\eta V},
\end{equation}
where $\| \varphi \|_{\eta V} := \mathop{\rm sup}\limits_{u \in \hat{H}^4} \exp(-\eta V^6(u)) | \varphi(u,\mathbf{p}) |$.
\end{corollary}

To adapt to the asymptotic strong Feller framework in \cite[Corollary 3.17]{HM06}, it is also necessary to prove the following weak irreducibility properties, proved in Section \ref{Approximate controllability} below.
\begin{proposition}[weak irreducibility]\label{irreducibility}
	For Systems \eqref{Boussinesq} and \eqref{Lagrange-process-eq}, we have the following:
	\begin{enumerate}
		\item[(1)] The support of any stationary measure for the Lagrangian process \((\mathbf{u}_t,\theta_t, \mathbf{x}_t)\) on \(H^5(\mathbb{T}^2) \times H^5(\mathbb{T}^2) \times \mathbb{T}^2\) must contain the set \(\{\mathbf{0}\} \times \{0\} \times \mathbb{T}^2\).
		
		\item[(2)] The support of any stationary measure for the projective process \((\mathbf{u}_t, \theta_t, \mathbf{x}_t, v_t)\) on \(H^5(\mathbb{T}^2) \times H^5(\mathbb{T}^2) \times \mathbb{T}^2 \times S^1\) must contain \(\{\mathbf{0}\} \times \{0\} \times \mathbb{T}^2 \times S^1\).
	\end{enumerate}
\end{proposition}

Corollary \ref{long-time} establishes that the Markov semigroup is asymptotic strong Feller. Combining this with the weak irreducibility from Proposition \ref{irreducibility}, we conclude through \cite[Corollary 3.17]{HM06} that there can exist at most one stationary measure. The existence follows from the tightness of the sequence of averaged measures, which is guaranteed by Corollary \ref{Super} and the compactness of \( \Sigma^1  \), \( \Sigma^P\). Subsequently, any subsequence of averaged measures (along a further subsequence) converges to the unique stationary measure, and therefore the original sequence of averaged measures must also converge. This establishes the following ergodicity result:
\begin{corollary}\label{erogdicty}
The base, Lagrangian, and projective processes each have a unique stationary measure, denoted \(\mu\), \(\mu^1\), and \(\mu^P\) respectively. Furthermore, for each initial probability measure \(\nu \in \mathcal{P}({H}^5(\mathbb{T}^2)\times {H}^5(\mathbb{T}^2)\times \Sigma)\), the averaged measures  
\[
	t^{-1} \int_0^t P_s \nu \, ds
\]  
converge weakly to the stationary measure as \( t \to \infty \).
\end{corollary}

\section{Malliavin calculus and probabilistic spectral bound}\label{smoothing estimate}
The objective of this section is to adapt the techniques from \cite{JFA, HM06, HM11a} to our specific setting. Here, we consider an arbitrary manifold \( \Sigma \) and its associated vector field mapping \( \Theta \), as described in Definition \ref{Lagrange-process-def}. Recall that \( \Theta \) and \( \mathbf{u}_t \) induce a extended process \( (\omega_t,\theta_t, \mathbf{p}_t) \), where \( \mathbf{p}_t \in \Sigma \). A central idea in the methodology of \cite{HM06,JFA,HM11a} lies in the differential treatment of such processes. Consequently, our primary focus will be on the linearization of the process \( (\omega_t, \theta_t, \mathbf{p}_t) \).
\begin{definition} \label{J}
We denote the derivative of the process \((\omega_t, \theta_t, \mathbf{p}_t)\)—viewed as a (random) function of its previous value \((\omega_s, \theta_s, \mathbf{p}_s)\)—by the (random) linear operator \(\mathcal{J}_{s,t} \colon \hat{H}^n \times T_{\mathbf{p}_s} \Sigma \to \hat{H}^n \times T_{\mathbf{p}_t} \Sigma\). We note that \(\mathcal{J}_{s,t}\) solves an equation in \(t\), that is
	\[
	\begin{cases}
		\frac{d}{dt} \mathcal{J}_{s,t} = L_t \mathcal{J}_{s,t} \\
		\mathcal{J}_{s,s} = \mathrm{id},
	\end{cases}
	\]
where \( L_t := L(\omega_t, \theta_t, \mathbf{p}_t) \colon \hat{H}^6 \times T_{\mathbf{p}_t}\Sigma \to \hat{H}^4 \times T_{\mathbf{p}_t}\Sigma \) is given by
\[
L_t \begin{pmatrix} \psi \\ \mathbf{q} \end{pmatrix} := 
\begin{pmatrix}
	-A \psi_t - B(\psi_t,U_t)-B(U_t,\psi_t)+G\psi_t \\
	q_t \cdot \nabla \Theta_{U_t}(\mathbf{p}_t) + \Theta_{\psi_t}(\mathbf{p}_t)
\end{pmatrix}.
\]
where \( A \), \( B(U, \tilde{U}) \) and $G$ are as defined in Section \ref{notaion}, \( \psi:=(\psi^1,\psi^2) \), and the specific form of \( L_t \) is given as follows:
\[
L_t \begin{pmatrix} \psi^1 \\ \psi^2 \\ \mathbf{q} \end{pmatrix} := 
\begin{pmatrix}
	\nu_1 \Delta \psi_t^1 - \nabla^\perp \Delta^{-1} \psi_t^1 \cdot \nabla \omega_t - \nabla^\perp \Delta^{-1} \omega_t \cdot \nabla \psi_t^1+g\partial_x \psi_t^2 \\
	\nu_2 \Delta \psi_t^2 - \nabla^\perp \Delta^{-1} \psi_t^1 \cdot \nabla \theta_t - \nabla^\perp \Delta^{-1} \omega_t \cdot \nabla \psi_t^2\\
	q_t \cdot \nabla \Theta_{\omega_t}(\mathbf{p}_t) + \Theta_{\psi_t^1}(\mathbf{p}_t)
\end{pmatrix}.
\]
In the remainder of this paper, for notational simplicity, we shall occasionally identify \( \Theta_{\psi_t^1}(\mathbf{p}_t) \)\textup{(}\( \Theta_{\omega_t}(\mathbf{p}_t) \)\textup{)} with \( \Theta_{\psi_t}(\mathbf{p}_t) \)\textup{(}\( \Theta_{U_t}(\mathbf{p}_t) \)\textup{)} without explicitly emphasizing that the vector fields on the manifold are induced by the velocity field. 

We denote the second derivative of the same process \( (\omega_t,\theta_t, \mathbf{p}_t) \)—again viewed as a function of \( (\omega_s,\theta_s, \mathbf{p}_s) \)—by the (random) bilinear form 
\[
\mathcal{J}_{s,t}^2 \colon (\hat{H}^n \times T_{\mathbf{p}_s}\Sigma) \otimes (\hat{H}^n \times T_{\mathbf{p}_s}\Sigma) \to \hat{H}^n \times T_{\mathbf{p}_t}\Sigma.
\]
The precise value of \( n \) for which these are defined is determined by Assumption \ref{assum5.4} below. In fact, setting $n=4$ suffices.
\end{definition}



We now give the essential bounds on the dynamics we need to prove our hypoellipticity result, compare to \cite[Lemma A.3]{JFA}. We note that below we use the notation \(\|T\|_{\hat{H}^{n} \to \hat{H}^{k}}\) for the operator norm of \(T \colon \hat{H}^{n} \times T\Sigma \to \hat{H}^{k} \times T\Sigma\), with the \(T\Sigma\) component normed by the Riemannian metric.
\begin{assumption}[essential bounds on the dynamics ]\label{assum5.4}
	There exists \( n \in \mathbb{N} \) such that for all \( T, q > 0, \eta \in (0, 1) \), there is \( C(\mathbf{p}_0, \eta, q, T) > 0 \), locally bounded in \( \mathbf{p}_0 \in \Sigma \), such that
	\begin{align}
		\mathbb{E} \sup_{0 \leq t \leq T} \sup_{\substack{\psi \in \hat{H}^{n}}} \|{ \Theta_\psi(\mathbf{p}_t) }\|^q &\leq C\mathbb{E}\|\psi\|_{\hat{H}^n}e^{\eta V^n(U_0)}  \label{assum5.4:line1}\\
		\mathbb{E} \sup_{0 \leq s \leq t \leq T} \|{ \mathcal{J}_{s,t} }\|_{\hat{H}^{n} \to \hat{H}^n}^q &\leq Ce^{\eta V^n(U_0)}  \label{assum5.4:line2}\\
		\mathbb{E} \|{ \mathcal{J}_{T/2, 3T/4} }\|_{\hat{H} \to \hat{H}^{n+1}}^q &\leq Ce^{\eta V^n(U_0)} \label{assum5.4:line3}\\
		\mathbb{E} \sup_{0 \leq s \leq t \leq T} \|{ \mathcal{J}_{s,t}^2 }\|_{\hat{H}^n \otimes \hat{H}^{n} \to \hat{H}^n}^q &\leq Ce^{\eta V^n(U_0)}  \label{assum5.4:line4}\\
		\mathbb{E} \sup_{T/2 \leq t \leq T} \|{ L_t }\|_{\hat{H}^{n+2} \to \hat{H}^n}^q &\leq Ce^{\eta V^n(U_0)} \label{assum5.4:line5}
	\end{align}
\end{assumption}

It is important to note that the gradient estimates for the Markov semigroup we aim to prove involve not only the base process but also the vector fields on the manifold. Therefore, to obtain the probabilistic spectral bound on a cone for the Malliavin matrix, we must additionally impose a spanning condition on the vector fields over the tangent bundle $T\Sigma$.
Specifically, at every point \( \mathbf{p} \in \Sigma \), there should exist some \( N > 0 \) such that
$$\left\{ \Theta_{Z_j^m(\overline U)}(\mathbf{p})+\Theta_{[U,Y_j^m(\overline U)]_x}(\mathbf{p}) : |j| \leqslant N,m\in \{0,1\} \right\} $$
spans \( T_\mathbf{p} \Sigma \). Here, $\overline{U}:=U-\sigma_\theta W$, 
\begin{align}\label{JFA-5.9}
Y_j^m(U):=[F(U),\sigma_j^m]=\nu_2|j|^2\sigma_j^m+B(U,\sigma_j^m)+(-1)^mgj_1\psi_j^{m+1},
\end{align}
and 
\begin{align}\label{JFA-5.11}
Z_j^m(U):&=[F(U),Y_j^m(U)]\notag\\
&=B\bigl(F(U), \sigma_j^m\bigr) + \nu_2^2 \lvert j \rvert^4 \sigma_j^m + (-1)^m (\nu_1 + \nu_2) g j_1 \lvert j \rvert^2 \psi_j^{m+1} + A\bigl(B(U, \sigma_j^m)\bigr) \notag\\
&\quad + (-1)^m g j_1 B\bigl(\psi_j^{m+1}, U\bigr) - B\Bigl(U, -\nu_2 \lvert j \rvert^2 \sigma_j^m + (-1)^{m+1} g j_1 \psi_j^{m+1}\Bigr) \notag\\
&\quad + B\bigl(U, B(U, \sigma_j^m)\bigr) - G B\bigl(U, \sigma_j^m\bigr),
\end{align}
where $F(U):=-AU-B(U,U)+GU$ is the drift of equation \eqref{Boussinesq-w1}. And \([X,Y]\) is the Lie bracket of two vector fields \(X,Y\), defined for each \(U \in \hat{H}^4\) by
\begin{equation}\label{U}
[X(U), Y({U})] = \nabla Y({U}) X(U) - \nabla X(U) Y({U}).
\end{equation}
When not taking the directional derivative with respect to $U$, we distinguish the Lie brackets by adding subscripts, such as $[X,Y]_x:=\nabla Y({x}) X(x) - \nabla X(x) Y({x})$.
During the proof of Proposition \ref{NS-prop5.22} later in the text, we observe that the spanning covering alone is insufficient--it additionally requires this spanning property to satisfy uniform quantitative moment bounds.
\begin{assumption}[solution-dependent manifold spanning condition]\label{assum5.5}
There exists \( n \in \mathbb{N} \) and \( N > 0 \) such that for all \( T, q > 0, \eta \in (0, 1) \), there exists \( C(\mathbf{p}_0, \eta, q, T) > 0 \), locally bounded in \( \mathbf{p}_0 \), such that
\[
	\mathbb{E} \left[ \sup_{\substack{v \in T_{\mathbf{p}}\Sigma, \|v\| = 1}} \left( \max_{|j| \leq N,m\in\{0,1\}} \left|\hat{g}\big(\Theta_{Z_j^m(\overline U)}(\mathbf{p}_T)+\Theta_{[U_t,Y_j^m(\overline U)]_x}(\mathbf{p}_T), v\big)\right| \right)^{-q} \right] \leq Ce^{\eta V^n (U_0)}.
\]
Here, we recall that $\hat{g}$ denotes the Riemannian metric on the manifold \( \Sigma \) as defined in Definition \ref{Lagrange-process-def}. For the case that \( \Sigma \) is compact, we get a much simpler bound, that
\[
	\mathbb{E} \left[\sup_{\substack{v \in T_{\mathbf{p}}\Sigma, \|v\| = 1}} \left( \max_{|j| \leq N,m\in\{0,1\}} \left|\hat{g}\big(\Theta_{Z_j^m(\overline U)}(\mathbf{p}_T)+\Theta_{[U_t,Y_j^m(\overline U)]_x}(\mathbf{p}_T), v\big)\right| \right)^{-1}\right] \leq C.
\]
\end{assumption}
\begin{remark}
It is important to note that, unlike in \emph{\cite{NS}} where only  $\Theta_{e_k}$ needs to be considered (with ${e_k}$ being the real-valued Fourier basis), the vector fields associated with $T\Sigma$ here are strongly dependent on the stochastic process $U$ satisfying equation \eqref{Boussinesq-w1}. Hence, the spanning property holds only almost everywhere, and additional difficulties arise when subsequently verifying the required assumptions.
\end{remark}
The following proposition verifies that the relevant processes satisfy the two assumptions stated above. Its detailed proof, which requires lengthy and involved computations, will be deferred to Subsections \ref{var-assum5.4} and \ref{var-assum5.5}.
\begin{proposition}\label{prop:5.6}
	\textit{The Lagrangian, projective, tangent, and Jacobian processes for Systems \eqref{Boussinesq} and \eqref{Lagrange-process-eq} each satisfy Assumptions \ref{assum5.4} and \ref{assum5.5}.}
\end{proposition}

For the remainder of this section, we fix \( n \) and suppose that Assumptions \ref{assum5.4} and \ref{assum5.5} hold for this \( n \). We then use a bracket to denote the inner product with respect to \( \hat{H}^n \times T_\mathbf{p}\Sigma \), that is
\[
\left\langle{\begin{pmatrix} \varphi \\ \mathbf{p} \end{pmatrix}},{\begin{pmatrix} \psi \\ \mathbf{q} \end{pmatrix}}\right\rangle := \hat{g}(\mathbf{q},\mathbf{p}) + \left\langle{\varphi},{\psi}\right\rangle_{\hat{H}^n}.
\]

All adjoints are taken with respect to this inner product, so if \( J: \hat{H}^n \times T\Sigma \to \hat{H}^n \times T\Sigma \), then \( J^* : \hat{H}^n \times T\Sigma \to \hat{H}^n \times T\Sigma \), with
\[
\left\langle{\mathfrak{q}},{J \mathfrak{p}}\right\rangle = \left\langle{J^* \mathfrak{q}},\mathfrak{p}\right\rangle,
\]
where we will throughout use \( \mathfrak{p} \) to denote an element of \( \hat{H}^n \times T\Sigma \). We also use (unsubscripted) norms to refer to the norm on this space:
$$
\| \mathfrak{p} \| := \| \mathfrak{p} \|_{\hat{H}^n} = \sqrt{\left\langle{\mathfrak{p}},\mathfrak{p}\right\rangle}.
$$
\subsection{Unit-time asymptotic gradient estimate for the Markov semigroup}\label{S-Gradient}
In this section, we utilize the integration-by-parts formula from Malliavin calculus to transform the estimation problem concerning the gradient operator $\nabla P_1\varphi$ in expression \eqref{unit-time-eq} into a control-theoretic framework. This transformation prompted us to undertake a deeper investigation of the Malliavin covariance matrix $\mathcal{M}$—a crucial object linking the existence of an ideal control to the Lie bracket properties of Hölder continuous vector fields (defined on the space $\hat{H}^n$) intrinsically associated with equation \eqref{Boussinesq-w1}.

Building on Proposition \ref{prop:5.6}, Theorem \ref{bounds-Malliavin} establishes the probabilistic spectral bound on a cone for the Malliavin covariance matrix $\mathcal{M}_T$, with its proof deferred to Subsection \ref{Spectral bounds}. Crucially, although Theorem \ref{bounds-Malliavin} superficially resembles corresponding results in \cite{JFA-89, HM06, JFA-3, JFA-56, NS, JFA}, its proof exhibits substantial differences. These arise from the distinctive nonlinear structure inherent in equation \eqref{Boussinesq} and its coupling of vector fields over the manifold $\Sigma$ (c.f. equations \eqref{Lagrange-process-eq}), constituting a primary mathematical innovation of this work. 

Within this subsection, we demonstrate how these spectral bounds integrate with the control system constructed around $\mathcal{M}_T$ to complete the proof of Theorem \ref{unit-time-n}. We note that Proposition \ref{unit-time} is a special case of Theorem \ref{unit-time-n}. Using the compactness of $\Sigma$, this result further establishes that the base, Lagrangian, and projective processes all satisfy the asymptotic strong Feller property. Consequently, Proposition \ref{continuous projective measures} follows, guaranteeing the existence of a family of unconditionally continuous projective measures satisfying \eqref{projective measures}.
\begin{theorem} \label{unit-time-n}
For all \(\eta > 0\), \(\mathbf{p}_0 \in \Sigma\), \(\gamma \in (0,1)\), there exists \(C(\eta, \gamma, \mathbf{p}_0) > 0\), bounded locally uniformly in \(\mathbf{p}_0\), such that for all Fréchet differentiable observables \(\varphi: \hat{H}^n \times \Sigma \to \mathbb{R}\),
\begin{equation}\label{unit-time-n-eq}
	\| \nabla P_1 \varphi (U_0, \mathbf{p}_0) \|_{\hat{H}^4} \leq e^{\eta V^{n+2}(U_0)} \left( C \sqrt{P_1 |\varphi|^2 (U_0, \mathbf{p}_0)} + \gamma \sqrt{P_1 \| \nabla \varphi \|_{\hat{H}^n}^2 (U_0, \mathbf{p}_0)} \right).
\end{equation}
\end{theorem}
	
Motivated by the properties of Malliavin derivative 
, we consider the following objects.
\begin{definition}\label{A}
For \(0 \leq s \leq t\), we define the operator \(\mathcal{A}_{s,t} : L^2\left([s, t], \mathbb{R}^{2|\mathcal{Z}|}\right) \to \hat{H}^n \times T_{\mathbf{p}_t}\Sigma\)
\[
	\mathcal{A}_{s,t}v := \sum_{j \in \mathcal{Z},m\in  \{0,1\}} \int_s^t \alpha_j^m v_r^{j,m} \mathcal{J}_{r,t}\sigma_j^m \, dr.
\]
\end{definition}

Having completed these preparatory steps, we proceed to calculations related to the gradient operator $\nabla P_1\varphi$. By using the Malliavin chain rule and integration by parts formulas 
, we deduce that for any $t\ge 0$ and any admissible (Skorokhod-integrable) test process $v\in L^2\left([s, t], \mathbb{R}^{2|\mathcal{Z}|}\right)$,
\begin{align}
	\nabla P_1 \varphi(U_0,\mathbf{p}_0) \cdot \mathfrak{p} &= \mathbb{E}\left( \nabla \varphi(U_t,\mathbf{p}_t) \cdot \left( \mathcal{A}_{0,1} v + \mathcal{J}_{0,1} \mathfrak{p} - \mathcal{A}_{0,1} v \right) \right) \notag\\
	&= \mathbb{E}\left( \varphi(U_t,\mathbf{p}_t) \int_0^1 v \cdot dW \right) + \mathbb{E}\left( \nabla \varphi(U_t,\mathbf{p}_t) \cdot \left( \mathcal{J}_{0,1} \mathfrak{p} - \mathcal{A}_{0,1} v \right) \right)\notag\\
	&\leq \left( \mathbb{E} {\delta(v)}^2 \right)^{1/2} \sqrt{P_1 |\varphi|^2 (U_0, \mathbf{p}_0)} + \left( \mathbb{E} \|\rho\|_{\hat{H}^n}^2 \right)^{1/2} \sqrt{P_1 \| \nabla \varphi \|_{\hat{H}^n}^2 (U_0, \mathbf{p}_0)},
\end{align}
for any $\mathfrak{p}\in \hat{H}^n\times T_{\mathbf{p}_0}\Sigma$, where $\delta(v):=\int_0^1 v \cdot dW$ denotes the Skorokhod integral of $v$ and $\rho:=\mathcal{J}_{0,1} \mathfrak{p}-\mathcal{A}_{0,1}v$ denotes the error. As such, \eqref{unit-time-n-eq} has been translated to the following control problem: for each \( \gamma, \eta > 0 \) and each element \( \mathfrak{p} \in \hat{H}^n\times T_{\mathbf{p}_0}\Sigma \) with $\|\mathfrak{p}\|=1$, find a (locally Skorokhod integrable) \( v = v(\mathfrak{p}) \in L^2(\Omega; L_{\mathrm{loc}}^2([0, \infty), \mathbb{R}^{2\cdot |\mathcal{Z}|})) \) such that
\begin{equation*}
\sup_{\|\mathfrak{p}\|=1} \mathbb{E} \| \rho( \mathfrak{p}, v) \|_{\hat{H}^n}^2 \leq \gamma\exp(\eta V^{n+2}(U_0))
\end{equation*}
and
\begin{equation*}
\sup_{\|\mathfrak{p}\|=1} \mathbb{E} \left| \int_{0}^{1} v \cdot dW \right|^2 \leq C \exp(\eta V^{n+2}(U_0)),
\end{equation*}
where \( C = C(\gamma, \eta) \) is independent of \( t \). 

In infinite-dimensional systems, due to the highly degenerate nature of the noise, the invertibility of the Malliavin covariance matrix is challenging to establish and generally fails to hold. Following the methodology developed in \cite[Section 3]{JFA}, \cite[Section 5]{JFA-56} etc., we address this issue by implementing Tikhonov regularization on the Malliavin matrix. This approach enables the construction of a control process   $v$ and its corresponding adjoint process $\rho$, ensuring their compliance with the aforementioned control requirements.

To make this more precise we first define several random operators. For any \( s < t \), let  
$\mathcal{A}_{s,t}^* : \hat{H}^n\times T_{\mathbf{p}_t}\Sigma \to L^2([s,t]; \mathbb{R}^{2\cdot|\mathcal{Z}|})$
be the adjoint of \( \mathcal{A}_{s,t} \) defined in Definition \ref{A}. We then define the Malliavin matrix 
\[
\mathcal{M}_{s,t} := \mathcal{A}_{s,t} \mathcal{A}_{s,t}^* :  \hat{H}^n\times T_{\mathbf{p}_t}M\to \hat{H}^n\times T_{\mathbf{p}_t}M,
\]
and for \(t\ge 0\), define the quadratic form \(\hat{H}^n\times T_{\mathbf{p}_t}\Sigma \to [0, \infty)\) given by
\begin{align}\label{recall1}
	\langle \mathfrak{p}, \mathcal{M}_t \mathfrak{p} \rangle := \langle \mathcal{A}_{0,t}^* \mathfrak{p}, \mathcal{A}_{0,t}^* \mathfrak{p} \rangle &= \sum_{j \in \mathcal{Z},m\in  \{0,1\}} {(\alpha_j^m)}^2 \int_0^t \left|\langle \mathfrak{p}, \mathcal{J}_{r,t} \sigma_j^m \rangle\right|^2 \, dr\notag\\ &= \sum_{j \in \mathcal{Z},m\in  \{0,1\}} {(\alpha_j^m)}^2 \int_0^t \left|\langle \sigma_j^m, \mathcal{J}_{r,t}^* \mathfrak{p} \rangle\right|^2 \, dr.
\end{align}
Since \( \mathcal{M}_{s,t} \) is positive definite but generally non-invertible, we regularize it by addition of a multiple of identity. It is noteworthy that compared to the control process $v$ defined in \cite{JFA}, our framework here only requires consideration of semigroup gradient estimates over finite time intervals. This allows for the construction of a significantly simplified control 
$v$, which in turn substantially streamlines the proof procedure while preserving rigor. Let $\mathfrak{p}\in \hat{H}^n\times T_{\mathbf{p}_0}\Sigma$ with $\|\mathfrak{p}\|=1$, then for any $\beta>0$ and $t\in [0,1/2]$, we set 
\begin{align}\label{v11}
v_{t}^{\beta} := \mathcal{A}_{0,1/2}^{*} \left( \mathcal{M}_{1/2} + \beta I \right)^{-1} \mathcal{J}_{0,1/2} \, \mathfrak{p}.
\end{align}
We extend \( v_t^\beta \) by $0$ on \([1/2,1]\). Note then that for any \(\beta > 0\), \(v_t^\beta \in L^2([0,1], \mathbb{R}^{2\cdot|\mathcal{Z}|})\), and set
$\rho^\beta := \mathcal{J}_{0,1}\mathfrak{p} - \mathcal{A}_{0,1}v^\beta.$ We then claim the following two propositions-refinements to the two aforementioned control problem formulations.
\begin{proposition}\label{error}
For all \(\eta, \gamma > 0\) there exists \(\beta_0(\eta, \mathbf{p}_0, \gamma) > 0\), locally lower bounded in \(\mathbf{p}_0\), such that for all \(0 < \beta \leq \beta_0\),  
	\[
	\mathbb{E} \| \rho^\beta \|_{\hat{H}^n}^2 \leq \gamma e^{(\eta V^{n+2}(U_0))}.
	\]
\end{proposition}
\begin{proposition}\label{v}
For all \(\beta > 0\), there exists \(C(\beta, \eta, \mathbf{p}_0) > 0\), locally bounded in \(\mathbf{p}_0\), such that  
\[
\mathbb{E} \delta (v^\beta)^2 \leq C e^{\eta V^n (U_0)}.
\]
\end{proposition}
Building upon the preceding analysis, we note that Theorem \ref{unit-time-n} follows directly from Propositions \ref{error} and \ref{v}. The proof of Proposition \ref{v} closely follows the arguments in \cite[Section 4.6]{HM06} and \cite[Section 5.2.2]{NS}, which can be similarly established using Hypothesis \ref{assum5.4}, the $L^2$-isometry of Skorokhod integration and Malliavin derivatives. It therefore remains only to prove Proposition \ref{error}. We shall establish this estimate by separately addressing the low-mode and high-mode components, beginning with the high-mode analysis.

\begin{lemma} \label{NS5.27}
	For all \(\eta, q, N > 0\) there exists \(C(\eta, q, N) > 0\) such that  
	\[
	\mathbb{E} \left\| \Pi_{\geq N} \mathcal{J}_{1/2, 1} \right\|_{\hat{H}^n \to \hat{H}^n}^q + \mathbb{E} \left\| \mathcal{J}_{1/2, 1} \Pi_{\geq N} \right\|_{\hat{H}^n \to \hat{H}^n}^q \leq C N^{-q} e^{\eta V^n (U_0)},
	\]
where $\Pi_{\geq N}: \hat{H}^n\times T\Sigma \to \hat{H}^n$ be the orthogonal projection onto the span of Fourier modes with wavenumber \( j \) such that \(|j| \geq N\).
\end{lemma}
\begin{proof}
Analogous to \cite[Lemma 5.27]{NS}, this lemma can be proved. Its foundation stems fundamentally from the compactness of the solution operator associated with the operator $A$.
\end{proof}

The following lemma is instrumental in controlling the low-mode part, the proof of which hinges on the spectral bounds of the Malliavin covariance matrix established in Section \ref{Spectral bounds}, i.e., Theorem \ref{bounds-Malliavin}.
\begin{lemma} \label{NS5.28}
For all \( \eta, \gamma, N > 0 \) there exists \( \beta_0 (\eta, \mathbf{p}_0, \gamma, N) > 0 \), locally lower bounded in \( \mathbf{p}_0 \), such that for all \( 0 < \beta \leq \beta_0 \),
\[
	\mathbb{E} \left\| \beta \Pi_N (\mathcal{M}_{1/2} + \beta I)^{-1} \mathcal{J}_{0,1/2} \mathfrak{p} \right\|^4_{\hat{H}^n} \leq \gamma e^{(\eta V^{n+2}(U_0))},
\]
where $\Pi_{N}: \hat{H}^n\times T\Sigma \to \hat{H}^n\times T\Sigma$ be the orthogonal projection such that \(\Pi_N|_{T\Sigma} = \operatorname{id}\) (the identity map), and \(\Pi_N|_{\hat{H}^n}\) is the orthogonal projection onto the span of \((\psi_j^m)_{|j| \leq N,m\in \{0,1\}}\).	
\end{lemma}

\begin{proof}
	In the proof of this lemma, it should also be noted that the stochastic integrability here is worse.  We locally utilize the probabilistic spectral bound on a cone for the Malliavin matrix (c.f. Theorem \ref{bounds-Malliavin}) while simultaneously employing the $O(1)$ polynomial moment bounds to control the probability tail. With the help of Assumption \ref{assum5.4}, the proof proceeds closely following that of \cite[Lemma 5.28 ]{NS}.
\end{proof}

\begin{proof}[Proof of Proposition \ref{error}]
We first note that
\begin{align*}
	\rho^\beta &= \mathcal{J}_{0,1}\mathfrak{p} - \mathcal{A}_{0,1}v^\beta \\
	&= \mathcal{J}_{1/2,1}\left(\mathcal{J}_{0,1/2} - \mathcal{A}_{0,1/2}\mathcal{A}_{0,1/2}^*(\mathcal{M}_{1/2} + \beta)^{-1}\mathcal{J}_{0,1/2}\right)\mathfrak{p} \\
	&= \beta \mathcal{J}_{1/2,1}(\mathcal{M}_{1/2} + \beta I)^{-1}\mathcal{J}_{0,1/2}\mathfrak{p} .
\end{align*}
Then compute that, for \(N > 1\),
\begin{align*}
	\|\rho^\beta\|_{\hat{H}^n} &= \left\| \beta \mathcal{J}_{1/2,1}(\mathcal{M}_{1/2} + \beta I)^{-1}\mathcal{J}_{0,1/2}\mathfrak{p} \right\|_{\hat{H}^n} \\
	&\leq \left\| \beta \mathcal{J}_{1/2,1}\Pi_{\geq N}(\mathcal{M}_{1/2} + \beta I)^{-1}\mathcal{J}_{0,1/2}\mathfrak{p} \right\|_{\hat{H}^n} + \left\| \beta \mathfrak{p}_{1/2,1}\Pi_N(\mathcal{M}_{1/2} + \beta I)^{-1}\mathcal{J}_{0,1/2}\mathfrak{p} \right\|_{\hat{H}^n} \\
	&\leq \left\| \mathcal{J}_{1/2,1}\Pi_{\geq N} \right\|_{\hat{H}^n \to \hat{H}^n} \left\| \mathcal{J}_{0,1/2} \right\|_{\hat{H}^n \to \hat{H}^n}\left\| \beta (\mathcal{M}_{1/2} + \beta I)^{-1} \right\|_{\hat{H}^n \to \hat{H}^n} \\
	&\quad+ \left\| \mathcal{J}_{1/2,1} \right\|_{\hat{H}^n \to \hat{H}^n} \left\| \beta \Pi_N(\mathcal{M}_{1/2} + \beta I)^{-1}\mathcal{J}_{0,1/2}\mathfrak{p} \right\|_{\hat{H}^n}.
\end{align*}
It is observed that, following the proof in \cite[Lemma A.6]{JFA}, we obtain for $0<s<t$ that
$$
\left\| (\mathcal{M}_{s,t} + \beta I)^{-1} \right\|_{\hat{H}^n \to \hat{H}^n}\le \beta^{-1}.
$$
Then combining this with H\"older’s inequality, Lemma \ref{NS5.27} where setting $N=\gamma^{-1}$, Lemma \ref{NS5.28} and the bound \eqref{assum5.4:line2} of Assumption \ref{assum5.4}, we obtain that for \(\beta \leq \beta_0 (\mathbf{p}_0, \gamma, \eta)\), with \(\beta_0\) locally bounded below in \(\mathbf{p}_0\),
\[
\mathbb{E} \left\| \rho^\beta \right\|_{\hat{H}^n}^2 \leq C \gamma e^{(\eta V^{n+2}(U_0))},
\]
where \(C\) does not depend on \(\gamma\). We then conclude after redefining \(\gamma\).

\end{proof}

\subsection{Probabilistic spectral bounds for the Malliavin matrix}\label{Spectral bounds}
The objective of this subsection is to prove Theorem \ref{bounds-Malliavin} under Hypotheses \ref{assum5.4} and \ref{assum5.5}. This theorem establishes probabilistic spectral bounds for the Malliavin matrix $\mathcal{M}_{T}$, demonstrating its quantitative non-degeneracy over finite-dimensional subspaces. This conclusion is pivotal for proving Proposition \ref{v}. In Section \ref{S-Gradient}, the gradient estimate \eqref{unit-time-n-eq} was established via Proposition \ref{v}. Compared to \cite[Theorem 5.8]{NS} and \cite[Theorem 4.1]{JFA}, the stochastic integrability in Theorem \ref{bounds-Malliavin} is significantly weaker, where the arbitrary rate $f(\varepsilon)$ is expressed in terms of $r(\varepsilon)$ and $g(\varepsilon)$. The subsequent proof will explicitly clarify the origin of this discrepancy. 

\begin{theorem}\label{bounds-Malliavin}
Recall that the definition of $\Pi_N$ in Lemma \ref{NS5.28}. Then for any $N\ge 1$, \(T > 0\), \(\eta, \alpha \in (0,1)\), and \(\mathbf{p}_0 \in \Sigma\), there exists a positive constant \( \epsilon^* := \epsilon^*(T, \alpha, \eta, N, \mathbf{p}_0) > 0 \) such that for any \( 0 < \epsilon < \epsilon^* \),  
\[
	\mathbb{P} \left( \inf_{\begin{subarray}{c} 
			\|\mathfrak{p}\|=1, \|\Pi_N \mathfrak{p}\| \geq \alpha 
	\end{subarray}} \langle \mathfrak{p}, \mathcal{M}_T \mathfrak{p} \rangle < \varepsilon \right) \leq f(\varepsilon) e^{(\eta V^{n+2}(U_0))},
\]
where \( f = f(T, \alpha, \eta, N, \mathbf{p}_0) \) is a nonnegative decreasing function such that \(\lim_{\varepsilon \to 0} f(\varepsilon) = 0\) locally uniformly in \(\mathbf{p}_0\).
\end{theorem}

We first introduce some notational conventions that will be frequently used in the sequel, along with several lemmas (Lemmas \ref{JFA6.1}, \ref{JFA-lemma 6.2}) and theorem (Theorem \ref{JFA-Theorem6.4}) that will be frequently employed in the subsequent analysis. For any \( a < b \), \(\beta \in \mathbb{R}\) and \(\alpha \in (0, 1]\), define the semi-norms
\[
\|U\|_{C^{\alpha}([a,b]; \hat{H}^{\beta})} := 
\sup_{\begin{subarray}{c} 
		t_1 \neq t_2 \\ 
		t_1, t_2 \in [a,b]
\end{subarray}} 
\frac{\|U_{t_1} - U_{t_2}\|_{\hat{H}^{\beta}}}{|t_1 - t_2|^{\alpha}}.
\]
If \( a = T/2 \) and \( b = T \), we will write \( \| \cdot \|_{C^\alpha \hat{H}^\beta} \) instead of \( \| \cdot \|_{C^\alpha ([T/2,T],\hat{H}^\beta)} \) and denote
\[
\| U \|_{C^0 \hat{H}^\beta} := \sup_{t \in [T/2,T]} \| U_t \|_{\hat{H}^\beta}.
\]
Similar notations will be employed for the Hölder spaces \( C^\alpha([a,b]), C^{1,\alpha}([a,b]) \), etc. . Recalling the notation of Lie bracket, we have
\[
[X(U), Y(\widetilde{U})] := \nabla Y(\widetilde{U}) X(U) - \nabla X(U) Y(\widetilde{U}),
\]
for all suitably regular \( X, Y : \hat{H}^n \to \hat{H}^n \) and \( U, \widetilde{U} \in \hat{H}^n \). Below we often consider \(\overline{U} = U - \sigma_\theta W\) which satisfies the shifted equation (cf. \eqref{Boussinesq-w1}):
\begin{align}\label{overline-U}
\partial_t \overline{U} = F(U) = F(\overline{U} + \sigma_\theta W), \quad \overline{U}(0) = U_0.
\end{align}
Note that, in contrast to \( U \), \(\overline{U}\) is \( C^{1,\alpha} \) in time for any \( \alpha < 1/2 \). Additionally, we will invoke the following standard result within the framework of evolution operator theory,
\begin{equation}\label{L*}
\frac{d}{dt} \mathcal{J}_{t,T}^*=-L_t^* \mathcal{J}_{t,T}^*.
\end{equation}

We next present two auxiliary lemmas which summarize the process of constructing \([F, E]\)-type brackets through time differentiation, along with the introduction of new vector fields on the manifold $T\Sigma$.
\begin{lemma}\label{JFA6.1}
Suppose \( E : \hat{H}^n\times T_{\mathbf{p}_t}\Sigma \to \hat{H}^n\times T_{\mathbf{p}_t}\Sigma \), ${\mathfrak{q}}\mapsto E({\mathfrak{q}})$ is Fréchet differentiable and decomposes as $E({\mathfrak{q}})=E^1(\overline{U})+E^2_{\overline{U}}(\mathbf{p}_t)$, \(\mathfrak{p} \in \hat{H}^n\times T_{\mathbf{p}_T}\Sigma\), \(U\) solves \eqref{Boussinesq-w1},  \(\overline{U}\) is defined by \eqref{overline-U} and $\mathbf{p}_t$ is defined by \eqref{Lagrange-process-eq}. Then for any \( p \geq 1 \) and any \( \eta \in (0,1) \), we have that
\begin{align}\label{3.11}
	&\mathbb{E} \sup_{t \in [T/2,T]} \left| \partial_t \langle \mathcal{J}^*_{t,T}\mathfrak{p}, E^1(\overline{U})+E^2_{\overline{U}}(\mathbf{p}_t) \rangle \right|^p \notag\\
	&\leq C \|\mathfrak{p}\|^p \exp(\eta V^n(U_0)) \left( \mathbb{E} \sup_{t \in [T/2,T]} \left\| [F(U),E^1(\overline{U}) ]-\Theta_{E^1(\overline{U})}(\mathbf{p}_t)+[\Theta_{\omega_t}(\mathbf{p}_t),E^2_{\overline{U}}(\mathbf{p}_t)]_x \right\|_{\hat{H}^n}^{2p} \right)^{1/2}, 
\end{align}
where \( C = C(\eta, p, T,\mathbf{p}_0) \), locally bounded in \( \mathbf{p}_0 \in \Sigma \). Moreover, for any \(\alpha \in (0, 1]\),
\begin{align}\label{3.12}
	&\mathbb{E} \left( \left\| \partial_t \langle \mathcal{J}^*_{t,T}\mathfrak{p}, E^1(\overline{U})+E^2_{\overline{U}}(\mathbf{p}_t) \rangle \right\|_{C^\alpha}^p \right) \notag\\
	&\leq C \|\mathfrak{p}\|^p \exp(\eta V^{n+2}(U_0)) \notag\\
	&\quad \cdot \Bigg[ \left( \mathbb{E} \sup_{t \in [T/2,T]} \left\| [F(U),E^1(\overline{U}) ]-\Theta_{E^1(\overline{U})}(\mathbf{p}_t)+[\Theta_{\omega_t}(\mathbf{p}_t),E^2_{\overline{U}}(\mathbf{p}_t)]_x \right\|_{\hat{H}^{n}}^{2p} \right)^{1/2} \notag\\
	&\quad\quad+ \left( \mathbb{E} \left\| [F(U),E^1(\overline{U}) ]-\Theta_{E^1(\overline{U})}(\mathbf{p}_t)+[\Theta_{\omega_t}(\mathbf{p}_t),E^2_{\overline{U}}(\mathbf{p}_t)]_x \right\|_{C^\alpha \hat{H}^{n}}^{2p} \right)^{1/2} \Bigg], 
\end{align}
with \( C = C(\eta, p, T, \mathbf{p}_0) \), locally bounded in \( \mathbf{p}_0 \in \Sigma \).
\end{lemma}

\begin{remark}
Notice that if, for the $E^2_{\overline{U}}(\mathbf{p}_t)$ in Lemma \ref{JFA6.1}, there exists an $F_{\overline{U}}$ such that 
$E^2_{\overline{U}}(\mathbf{p}_t)=\Theta_{F_{\overline{U}}}(\mathbf{p}_t)$, then we have
\begin{align}\label{JFA-6.6}
	&\mathbb{E} \sup_{t \in [T/2,T]} \left| \partial_t \langle \mathcal{J}^*_{t,T}\mathfrak{p}, E^1(\overline{U})+\Theta_{F_{\overline{U}}}(\mathbf{p}_t) \rangle \right|^p \notag\\
	&\leq C \|\mathfrak{p}\|^p \exp(\eta V^n(U_0)) \left( \mathbb{E} \sup_{t \in [T/2,T]} \left\| [F(U),E^1(\overline{U}) ]-\Theta_{E^1(\overline{U})}(\mathbf{p}_t)+\Theta_{\big[U_t,F_{\overline{U}}\big]_x}(\mathbf{p}_t) \right\|_{\hat{H}^n}^{2p} \right)^{1/2}, 
\end{align}
where \( C = C(\eta, p, T,\mathbf{p}_0) \), locally bounded in \( \mathbf{p}_0 \in \Sigma \). And, for any \(\alpha \in (0, 1]\),
	\begin{align}\label{JFA-6.7}
		&\mathbb{E} \left( \left\| \partial_t \langle \mathcal{J}^*_{t,T}\mathfrak{p}, E^1(\overline{U})+\Theta_{F_{\overline{U}}}(\mathbf{p}_t) \rangle \right\|_{C^\alpha}^p \right) \notag\\
		&\leq C \|\mathfrak{p}\|^p \exp(\eta V^{n+2}(U_0)) \notag\\
		&\quad \cdot \Bigg[ \left( \mathbb{E} \sup_{t \in [T/2,T]} \left\| [F(U),E^1(\overline{U}) ]-\Theta_{E^1(\overline{U})}(\mathbf{p}_t)+\Theta_{\big[U_t,F_{\overline{U}}\big]_x}(\mathbf{p}_t) \right\|_{\hat{H}^{n}}^{2p} \right)^{1/2} \notag\\
		&\quad\quad+ \left( \mathbb{E} \left\| [F(U),E^1(\overline{U}) ]-\Theta_{E^1(\overline{U})}(\mathbf{p}_t)+\Theta_{\big[U_t,F_{\overline{U}}\big]_x}(\mathbf{p}_t) \right\|_{C^\alpha \hat{H}^{n}}^{2p} \right)^{1/2} \Bigg], 
	\end{align}
	with \( C = C(\eta, p, T, \mathbf{p}_0) \), locally bounded in \( \mathbf{p}_0 \in \Sigma \).
\end{remark}

\begin{proof}
Since \( \mathcal{J}^*_{t,T} \mathfrak{p}\) solves \eqref{L*} and \( \overline{U} \) satisfies \eqref{overline-U}, we have
\begin{align}\label{JFA-6.8}
		&\partial_t \langle \mathcal{J}^*_{t,T}\mathfrak{p}, E^1(\overline{U})+E^2_{\overline{U}}(\mathbf{p}_t)\rangle \notag\\
		&= \langle \partial_t \mathcal{J}^*_{t,T}\mathfrak{p}, E^1(\overline{U})+E^2_{\overline{U}}(\mathbf{p}_t)\rangle + \langle\mathcal{J}^*_{t,T}\mathfrak{p}, \nabla E^1(\overline{U}) \cdot \partial_t \overline{U}+\nabla E^2_{\overline{U}} \cdot \partial_t \mathbf{p}_t \rangle \notag\\
		&= -\langle \mathcal{J}^*_{t,T}\mathfrak{p}, L_t\left(E^1(\overline{U})+E^2_{\overline{U}}(\mathbf{p}_t)\right)\rangle+\langle\mathcal{J}^*_{t,T}\mathfrak{p}, \nabla E^1(\overline{U}) \cdot \partial_t \overline{U}+\nabla E^2_{\overline{U}} \cdot \partial_t \mathbf{p}_t \rangle\notag\\
		&= \langle \mathcal{J}^*_{t,T}\mathfrak{p},\, AE^1(\overline{U})+\nabla B(U)E^1(\overline{U})-GE^1(\overline{U})-\Theta_{E^1(\overline{U})}(p_t)-\nabla\Theta_{U_t}(\mathbf{p}_t)\cdot E^2_{\overline{U}}(\mathbf{p}_t)\rangle\notag\\
		&\quad+\langle\mathcal{J}^*_{t,T}\mathfrak{p},\, \nabla E^1(\overline{U}) \cdot \partial_t \overline{U}+\nabla E^2_{\overline{U}}\cdot \Theta_{U_t}(\mathbf{p}_t)  \rangle\notag\\
		&= -\langle \mathcal{J}^*_{t,T}\mathfrak{p}, \Theta_{E^1(\overline{U})}(\mathbf{p}_t)\rangle-\langle \mathcal{J}^*_{t,T}\mathfrak{p}, \nabla F(U)E^1(\overline{U})\rangle+\langle\mathcal{J}^*_{t,T}\mathfrak{p}, \nabla E^1(\overline{U}) \cdot F(U) \rangle\notag\\
		&\quad-\langle \mathcal{J}^*_{t,T}\mathfrak{p}, \nabla\Theta_{U_t}(\mathbf{p}_t)\cdot E^2_{\overline{U}}(\mathbf{p}_t)\rangle+\langle\mathcal{J}^*_{t,T}\mathfrak{p}, \nabla E^2_{\overline{U}}\cdot \Theta_{U_t}(\mathbf{p}_t) \rangle \notag\\
		&= \langle\mathcal{J}^*_{t,T}\mathfrak{p},\,[F(U),E^1(\overline{U})]+[\Theta_{U_t}(\mathbf{p}_t),E^2_{\overline{U}}(\mathbf{p}_t)]_x-\Theta_{E(\overline{U})}(\mathbf{p}_t)\rangle.
\end{align}
Now, \eqref{3.11} immediately follows from H\"{o}lder's inequality, \eqref{assum5.4:line2} and duality. To prove \eqref{3.12}, we use that for any \( \alpha \in (0, 1) \), \( s, s' \in \mathbb{R} \), and any suitably regular \( A, B \), one has
\begin{align}\label{JFA-6.9}
		\| \langle A, B \rangle \|_{C^\alpha} 
		&:= \sup_{\substack{t \neq s \\ s, t \in [T/2, T]}} \left| \frac{\langle A(t), B(t) \rangle - \langle A(s), B(s) \rangle}{|s - t|^\alpha} \right| \notag\\
		&= \sup_{\substack{t \neq s \\ s, t \in [T/2, T]}} \left| \frac{\langle A(t) - A(s), B(t) \rangle + \langle A(s), B(t) - B(s) \rangle}{|s - t|^\alpha} \right| \notag\\
		&\leq \| A \|_{L^\infty \hat{H}^{n}} \| B \|_{C^\alpha \hat{H}^n} + \| A \|_{C^\alpha \hat{H}^{n}} \| B \|_{L^\infty \hat{H}^{n}}. 
\end{align}
Combining \eqref{JFA-6.9} with \eqref{JFA-6.8}, and using H\"{o}lder's inequality,
\begin{align*}
		&\mathbb{E} \left( \left\| \partial_t \langle \mathcal{J}^*_{t,T} \mathfrak{p}, E(\overline{U}) \rangle \right\|_{C^\alpha}^p \right)\\
		&\leq C \left( \mathbb{E} \left( \sup_{t \in [T/2, T]} \| \mathcal{J}^*_{t,T} \mathfrak{p} \|_{\hat{H}^{n}}^{2p} \right)\right)^{1/2} \left( \mathbb{E} \left( \| [F(U),E^1(\overline{U})]+[\Theta_{U_t}(\mathbf{p}_t),E^2_{\overline{U}}(\mathbf{p}_t)]_x-\Theta_{E(\overline{U})}(\mathbf{p}_t) \|_{C^\alpha \hat{H}^{n}}^{2p} \right)^{1/2} \right) \\
		&\quad + C \left( \mathbb{E} \left( \| \mathcal{J}^*_{t,T} \mathfrak{p} \|_{C^\alpha \hat{H}^{n}}^{2p} \right)\right)^{1/2} \left( \mathbb{E} \left( \sup_{t \in [T/2, T]} \| [F(U),E^1(\overline{U})]+[\Theta_{U_t}(\mathbf{p}_t),E^2_{\overline{U}}(\mathbf{p}_t)]_x-\Theta_{E(\overline{U})}(\mathbf{p}_t)\|_{\hat{H}^{n}}^{2p} \right)^{1/2} \right).
\end{align*}
Now it suffices to estimate $\mathbb{E} \big( \| \mathcal{J}^*_{t,T} \mathfrak{q} \|_{C^\alpha \hat{H}^{n}}^{2p} \big)$. Noted that $\rho^*:=\mathcal{J}^*_{t,T} \mathfrak{q}$ solves
$$
\partial_s \rho^* =A\rho^*+(\nabla B(U(s)))^*\rho^*-(\nabla G)*\rho^*,\, \rho^*(t)=\mathfrak{q}.
$$
Since $\|A\rho^*\|_{\hat{H}^n}\le \|\rho^*\|_{\hat{H}^{n+2}}$ and
\begin{align*}
	&\|{(\nabla B(U))^*\rho^* - (\nabla G(U))^*\rho^*}\|_{\hat{H}^n} \\
	&\leq \sup_{\|{\phi}\|_{\hat{H}^{-n}} \leq 1} \bigl( |\langle{(\nabla B(U))^*\rho^*},{\phi}\rangle| + |\langle{(\nabla G(U))^*\rho^*},{\phi}\rangle| \bigr) \\
	&\leq \sup_{\|\phi\|_{\hat{H}^{-n}} \leq 1} \bigl( |\langle{\rho^*},{B(U, \phi)}\rangle| + |\langle{\rho^*},{B(\phi, U)}\rangle| + |\langle{\rho^*},{\nabla G(U)\phi}\rangle| \bigr) \\
	&\leq \|\rho^*\|_{\hat{H}^{n-1}} \sup_{\|\phi\|_{\hat{H}^{-n}} \leq 1} \bigl( 2\|U\|_{\hat{H}^{-n+1}}\|\phi\|_{\hat{H}^{-n+1}} + |g|\|\phi\|_{\hat{H}^{-n+1}} \bigr) \\
	&\leq C\|\rho^*\|_{\hat{H}^{n-1}} \bigl( \|U\|_{\hat{H}^{-n+1}} + 1 \bigr)\leq C\|\rho^*\|_{\hat{H}^{n}} \bigl( \|U\|_{\hat{H}} + 1 \bigr).
\end{align*}
Then, combining \eqref{JFA-A.5} and \eqref{assum5.4:line2}, we deduce
$$
\mathbb{E} \sup_{t \in [T/2, T]} \|\partial_t \mathcal{J}^*_{t,T} \mathfrak{p}\|_{H^{n}}^p \leq C \exp \bigl( \eta V^{n+2}(U_0) \bigr) \|\mathfrak{p}\|^p.
$$
This completes the proof of the lemma.
\end{proof}
\begin{lemma}\emph{\cite[Lemma 6.2]{JFA}}\label{JFA-lemma 6.2}
Fix \( T > 0 \), \(\alpha \in (0,1]\) and an index set \(\mathcal{I} \). Consider a collection of random functions \( f_\phi \) taking values in \( C^{1,\alpha} ([T/2,T]) \) and indexed by \(\phi \in \mathcal{I} \). Define, for each \(\epsilon > 0\),
\begin{equation}\label{JFA-6.10}
	\Lambda_{\epsilon,\alpha} := \bigcup_{\phi \in \mathcal{I}} \Lambda_{\epsilon,\alpha}^\phi, \quad \text{where } \Lambda_{\epsilon,\alpha}^\phi := 
	\left\{ \sup_{t \in [T/2,T]} |f_\phi(t)| \leq \epsilon \quad \text{and} \sup_{t \in [T/2,T]} |\partial_tf_\phi(t)| > \epsilon^{\frac{\alpha}{2(1+\alpha)}}
	\right\}.
\end{equation}
Then, there is \(\epsilon_0 = \epsilon_0(\alpha, T)\) such that for each \(\epsilon \in (0, \epsilon_0)\)
\begin{equation}\label{JFA-6.11}
	\mathbb{P}(\Lambda_{\epsilon,\alpha}) \leq C \epsilon  \mathbb{E} \left( \sup_{\phi \in \mathcal{I}} \| f_\phi \|_{C^{1,\alpha}([T/2,T])}^{2/\alpha} \right).
\end{equation}
\end{lemma}

\begin{remark}\label{JFA-remark6.3}
Observe that
\[
\Lambda_{\epsilon,\alpha}^c = \bigcap_{\phi \in \mathcal{I}} \left\{ 
\sup_{t \in [T/2,T]} |f_\phi(t)| > \epsilon 
\quad \text{or} \quad 
\sup_{t \in [T/2,T]} |\partial_tf_\phi(t)| \leq \epsilon^{\frac{\alpha}{2(1+\alpha)}}
\right\}.
\]
Thus, on \(\Lambda_{\epsilon,\alpha}^c\),
\begin{equation}\label{JFA-6.13}
\sup_{t \in [T/2,T]} |f_\phi(t)| < \epsilon 
\, \Longrightarrow \,
\sup_{t \in [T/2,T]} |\partial_tf_\phi(t)| \leq \epsilon^{\frac{\alpha}{2(1+\alpha)}} 
\end{equation}
for every \(\phi \in \mathcal{I}\).
\end{remark}

The following is a nonadapted version of Norris's lemma \cite{Nor86}. This result is essential for controlling Malliavin matrices and enables us to distinguish the influences of different Brownian motions. Given any multi-index \(\alpha := (\alpha_1, \ldots, \alpha_d) \in \mathbb{N}^d\) recall the standard notation \(W^\alpha := W_1^{\alpha_1} \cdots W_d^{\alpha_d}\).
\begin{theorem}\emph{\cite[Theorem~6.4]{JFA-56}}\label{JFA-Theorem6.4}
Fix \(M, T > 0\). Consider the collection \(\mathfrak{P}_M\) of \(M\)th degree of `Wiener polynomials' of the form
\[
F = A_0 + \sum_{|\alpha| \leq M} A_\alpha W^\alpha,
\]
where for each multi-index \(\alpha\), with \(|\alpha| \leq M\), \(A_\alpha : \Omega \times [0, T] \to \mathbb{R}\) is an arbitrary stochastic process. Then, for all \(\epsilon \in (0, 1)\) and \(\beta > 0\), there exists a measurable set \(\Omega_{\epsilon,M,\beta}\) with
\[
\mathbb{P}(\Omega_{\epsilon,M,\beta}^c) \leq C\epsilon,
\]
such that on \(\Omega_{\epsilon,M,\beta}\) and for every \(F \in \mathfrak{P}_M\)
\[
\sup_{t \in [0,T]} |F(t)| < \epsilon^\beta \implies 
\begin{cases}
	\text{either} & \mathop{\rm sup}\limits_{|\alpha| \leq M} \mathop{\rm sup}\limits_{t \in [0,T]} |A_\alpha(t)| \leq \epsilon^{\beta 3^{-M}}, \\
	\text{or} & \mathop{\rm sup}\limits_{|\alpha| \leq M} \mathop{\rm sup}\limits_{\substack{s \neq t \\ s,t \in [0,T]}} \frac{|A_\alpha(t) - A_\alpha(s)|}{|t-s|} \geq \epsilon^{-\beta 3^{-(M+1)}}.
\end{cases}
\]
\end{theorem}

We plan to establish the probabilistic spectral bound for the Malliavin matrix by considering its $\hat{H}^n$-component and $T\Sigma$-component separately. To this end, we need to obtain Lie bracket vector fields satisfying the generalized Hörmander condition (see \cite[Definition 1.2]{JFA}). Following an inductive argument similar to those in \cite{HM06,JFA,HM11a}, we derive the corresponding implications on the $\hat{H}^n$-component and the $T\Sigma$-component. Below is the base case of the inductive argument.
\begin{proposition}\label{JFA-Lemma6.5}
For every \( 0 < \epsilon < \epsilon_0(T) \) and every \(\eta \in (0,1)\) there exist a set \(\Omega_{\epsilon,\Sigma}\) and \(C = C(\eta, T, \mathbf{p}_0)>0\) that is locally bounded in \( \mathbf{p}_0 \in \Sigma \), such that
\[
\mathbb{P}(\Omega_{\epsilon,\Sigma}^c) \leq C \exp(\eta V^n(U_0))\epsilon
\]
and on the set \(\Omega_{\epsilon,\Sigma}\)
\[
\langle\mathcal{M}_{T} \mathfrak{p}, \mathfrak{p}\rangle \leq \epsilon \|\mathfrak{p}\|^2 \implies \sup_{t \in [T/2,T]} |\langle\mathcal{J}_{t,T}^* \mathfrak{p}, \sigma_j^m \rangle| \leq \epsilon^{1/8} \|\mathfrak{p}\|,
\]
for each \( j \in \mathcal{Z} \), \( m \in \{0,1\} \), and every \( \mathfrak{p} \in \hat{H}^n\times T_{\mathbf{p}_T}\Sigma \). 
\end{proposition}

\begin{proof}
For any \(\mathfrak{p} \in \hat{H}^n\times T_{\mathbf{p}_T}\Sigma\) with \(\|\mathfrak{p}\| = 1\), recall that \eqref{recall1}, then define
	\[
	f_\mathfrak{p}(t) := \sum_{\substack{j \in \mathcal{Z} \\ m \in \{0,1\}}} (\alpha_j^m)^2 \int_0^t \langle \sigma_j^m, \mathcal{J}_{r,T}^* \mathfrak{p} \rangle^2  dr \leq \sum_{\substack{j \in \mathcal{Z} \\ m \in \{0,1\}}} (\alpha_j^m)^2 \int_0^T \langle \sigma_j^m, \mathcal{J}_{r,T}^* \mathfrak{p} \rangle^2  dr = \langle \mathcal{M}_{T} \mathfrak{p}, \mathfrak{p} \rangle.
	\]
Note that
\begin{equation*}
	\partial_tf_\mathfrak{p}(t) = \sum_{\substack{j \in \mathcal{Z} \\ m \in \{0,1\}}} (\alpha_j^m)^2 \langle \sigma_j^m, \mathcal{J}_{r,T}^* \mathfrak{p} \rangle^2, 
\end{equation*}
\begin{equation*}
	\partial_{tt}f_\mathfrak{p}(t) = 2 \sum_{\substack{j \in \mathcal{Z} \\ m \in \{0,1\}}} (\alpha_j^m)^2 \langle \sigma_j^m, \mathcal{J}_{r,T}^* \mathfrak{p} \rangle \langle \sigma_j^m, \partial_t \mathcal{J}_{r,T}^* \mathfrak{p} \rangle=2 \sum_{\substack{j \in \mathcal{Z} \\ m \in \{0,1\}}} (\alpha_j^m)^2 \langle \sigma_j^m, \mathcal{J}_{r,T}^* \mathfrak{p} \rangle \langle L_t\sigma_j^m, \mathcal{J}_{r,T}^* \mathfrak{p} \rangle.
\end{equation*}
Let \(\Omega_{\epsilon, \Sigma} := \Lambda_{\epsilon, 1}^c\), where \(\Lambda_{\epsilon, \alpha}\) is as in \eqref{JFA-6.10} with \(\mathcal{I} := \{\mathfrak{p} \in \hat{H}^n\times T_{\mathbf{p}_T}\Sigma : \|\phi\| = 1\}\), i.e.,
\begin{equation*}
	\Lambda_{\epsilon,1} = \bigcup_{\mathfrak{p} \in \mathcal{I}}  
	\left\{ \sup_{t \in [T/2,T]} |f_\mathfrak{p}(t)| \leq \epsilon \quad \text{and} \sup_{t \in [T/2,T]} |\partial_tf_\mathfrak{p}(t)| > \epsilon^{\frac{1}{4}}
	\right\}.
\end{equation*}
Then by Lemma \ref{JFA-lemma 6.2} with \(\alpha = 1\), one has
	\[
	\mathbb{P}(\Omega_{\epsilon, \Sigma}^c) \leq C \epsilon \sum_{\substack{j \in \mathcal{Z} \\ m \in \{0,1\}}} (\alpha_j^m)^4 \mathbb{E} \left( \sup_{\substack{t \in [T/2, T] \\ \|\mathfrak{p}\| = 1}} |\langle \sigma_j^m, \mathcal{J}_{r,T}^* \mathfrak{p} \rangle \langle L_t\sigma_j^m, \mathcal{J}_{r,T}^* \mathfrak{p} \rangle|^2 \right) .
	\]
Note that 
\[ \left| \langle L_t \sigma_j^m, \mathcal{J}_{r,T}^*  \mathfrak{p} \rangle \right| \leqslant \|L_t \sigma_j^m\| \|\mathcal{J}_{r,T}^* \mathfrak{p}\| \leqslant \frac{1}{2} \left( C \|L_t\|^2_{\hat{H}^{n+2}\to \hat{H}^n} + \|\mathcal{J}_{r,T}\|^2_{\hat{H}^{n}\to \hat{H}^n} \right)\|\mathfrak{p}\|. \]
Then by Assumption \ref{assum5.4}, we obtain
\[
\mathbb{P}(\Omega_{\epsilon, \Sigma}^c) \leq C\exp(\eta V^n(U_0)) \epsilon
\]
for any \(\epsilon < \epsilon_0 = \epsilon_0(T)\), where \(C = C(\eta, T,\mathbf{p}_0)\) that is locally bounded in \( \mathbf{p}_0 \in \Sigma \). Finally, on \(\Omega_{\epsilon, \Sigma}\) we have, cf. \eqref{JFA-6.13}, that
	\[
	\langle \mathcal{M}_{T} \mathfrak{p}, \mathfrak{p} \rangle \leq \epsilon \|\mathfrak{p}\|^2 \Rightarrow \sup_{t \in [T/2, T]} |\alpha_j^m||\langle \mathcal{J}_{t,T}^* \mathfrak{p}, \sigma_j^m \rangle| \leq \epsilon^{1/8} \|\mathfrak{p}\|,
	\]
for each \(j \in \mathcal{Z}\), \(m \in \{0,1\}\) and any \(\mathfrak{p} \in \hat{H}^n\times T_{\mathbf{p}_t}\Sigma\). Noted that \(\alpha_j^m \neq 0\), then the assertion of the lemma follows for \(\epsilon \leq \epsilon_0(T)\).
\end{proof}

Our objective is not only to span the velocity and temperature fields, as in \cite[Sections 5-6]{JFA}, but also to obtain information about the manifold $T\Sigma$. The following propositions constitute the inductive step of this subsection's main argument. Next, we turn to implications of the form $\sigma\to [\sigma,F]=Y$.

\begin{proposition}\label{JFA-Lemma6.6}
Fix any \( j \in \mathbb{Z}_+^2 \). For each \( 0 < \epsilon < \epsilon_0(T) \) and \( \eta \in (0,1) \), there exist a set \( \Omega_{\epsilon,j,Y} \) and a constant \( C = C(\eta,T,\mathbf{p}_0)>0\) that is locally bounded in \( \mathbf{p}_0 \in \Sigma \) such that
\[
	\mathbb{P}(\Omega_{\epsilon,j,Y}^c) \leq C|j|^{16} \exp(\eta V^n(U_0))\epsilon,
\]
and such that on the set \( \Omega_{\epsilon,j,Y} \), for each \( m \in \{0,1\} \), it holds that
\begin{equation}\label{JFA-6.14}
	\sup_{t \in [T/2,T]} |\langle \mathcal{J}_{t,T}^* \mathfrak{p}, \sigma_j^m \rangle| \leq \epsilon \|\mathfrak{p}\| \implies \sup_{t \in [T/2,T]} |\langle \mathcal{J}_{t,T}^* \mathfrak{p}, Y_j^m(U) \rangle| \leq \epsilon^{1/4} \|\mathfrak{p}\|. 
\end{equation}	
\end{proposition}

\begin{proof}
By expanding \( U = \overline{U} + \sigma W \), and by exploiting the fact that
$$
B(U,\tilde{U})=0,\, \text{if}\, U=(0,\theta), 
$$
we observe that
\begin{equation}\label{JFA-6.15}
Y_j^m(U) = Y_j^m(\overline{U}).
\end{equation}
Then for fixed \( m \in \{0,1\} \) and any \(\mathfrak{p} \in \mathcal{I} := \{\mathfrak{p} \in \hat{H}^n\times T_{\mathbf{p}_T}\Sigma : \|\mathfrak{p}\| = 1\}\), define \( f_\mathfrak{p}(t) := \langle \mathcal{J}_{t,T}^* \mathfrak{p}, \sigma_j^m \rangle \) and using \eqref{JFA-6.8} and \eqref{JFA-6.8}, we derive that 
\[
\partial_t f_\phi(t) = \langle \mathcal{J}_{t,T}^* \mathfrak{p}, [F(U), \sigma_j^m]-\Theta_{\sigma_j^m}(\mathbf{p}_t) \rangle = \langle \mathcal{J}_{t,T}^* \mathfrak{p}, Y_j^m(U)-\Theta_{\sigma_j^m}(\mathbf{p}_t) \rangle = \langle \mathcal{J}_{t,T}^* \mathfrak{p}, Y_j^m(\overline{U})-\Theta_{\sigma_j^m}(\mathbf{p}_t) \rangle.
\]
Observe that the vector fields on the manifold are induced by the velocity field, and since the noise under consideration acts exclusively on the temperature equation, it follows that essentially
\[
\partial_t f_\phi(t) = \langle \mathcal{J}_{t,T}^* \mathfrak{p}, Y_j^m(\overline{U}) \rangle.
\]
Let \(\Omega_{\epsilon,j,Y} := \Lambda_{\epsilon,1}^c\) with \(\alpha=1\). Once again, with \eqref{JFA-6.13}, we see that \eqref{JFA-6.14} holds on \(\Omega_{\epsilon,j,Y}\). On the other hand, by \eqref{JFA-6.11}, \eqref{JFA-6.6}, \eqref{JFA-6.15} and \eqref{JFA-5.11}, we have
\begin{align*}
\mathbb{P}(\Omega_{\epsilon,j,Y}^c) &\leq C \epsilon \mathbb{E} \left( \sup_{\phi \in \mathcal{I}} \sup_{t \in [T/2,T]} \left|\partial_t \langle \mathcal{J}_{t,T}^* \mathfrak{p}, Y_j^m(\overline{U}) \rangle\right|^2 \right)\\
&\leq C \epsilon \exp \left( \frac{\eta}{2}V^n(U_0) \right) \left(\mathbb{E} \sup_{t \in [T/2,T]} \| Z_j^m(U)-\Theta_{Y_j^m(\overline{U})} (\mathbf{p}_t)\|_{\hat{H}^n}^4\right)^{1/2}.
\end{align*}
Noted that 
\begin{equation}\label{JFA-6.16}
\sup_{t\in [T/2,T]}\|Z_j^m(U)\|_{\hat{H}^s}\le C|j|^{4+s}(1+\sup_{t\in [T/2,T]}\|U\|^2_{\hat{H}^{s+2}}),
\end{equation}
and
\begin{equation}\label{Y-estimate}
\sup_{t\in [T/2,T]}\|Y_j^m(U)\|_{\hat{H}^s}\le C|j|^{2+s}(1+\sup_{t\in [T/2,T]}\|U\|^2_{\hat{H}^{s+1}}),
\end{equation}
which follows from \eqref{JFA-5.9} and \eqref{JFA-5.11} by counting derivatives and applying the Hölder and Poincaré inequalities. Thus by \eqref{JFA-A.5} and Assumption \ref{assum5.4}, we have 
\begin{align*}
	\mathbb{P}(\Omega_{\epsilon,j,Y}^c) \leq C \epsilon |j|^{16} \exp \left( \frac{\eta}{2} V^n(U_0) \right) \left(\mathbb{E} \left( 1 + \sup_{t \in [T/2,T]} \| U \|_{\hat{H}^6}^8 \right) \right)^{1/2}\leq C \epsilon |j|^{16} \exp (\eta V^n(U_0))
\end{align*}
for any \(\epsilon < \epsilon^*(T)\), where \(C = C(\eta, T,\mathbf{p}_0)\) that is locally bounded in \( \mathbf{p}_0 \in \Sigma \). 
\end{proof}

We next establish implications corresponding the chain of brackets $Y\to Z\to [Z,\sigma]$.

\begin{proposition}\label{JFA-Lemma6.8}
Fix \( j \in \mathbb{Z}_+^2 \). For each \( 0 < \epsilon < \epsilon_0(T) \), and \( \eta \in (0,1) \) there exist a set \( \Omega_{\epsilon,j} \) and \(C = C(\eta, T,\mathbf{p}_0)\) that is locally bounded in \( \mathbf{p}_0 \in \Sigma \) such that
\begin{equation}\label{JFA-6.17}
	\mathbb{P}(\Omega_{\epsilon,j}^c) \leq C|j|^{90 \times 10} \exp(\eta V^{n+2}(U_0))\epsilon, 
\end{equation}
and on the set \( \Omega_{\epsilon,j} \), for each \( m \in \{0,1\} \), it holds that
\begin{align*}
	\sup_{t \in [T/2,T]} |\langle \mathcal{J}^*_{t,T}\mathfrak{p}, Y_j^m(U) \rangle| \leq \epsilon \|\mathfrak{p}\|
	\Longrightarrow 
	\begin{cases}
		\mathop{\rm sup}\limits_{t \in [T/2,T]} |\langle \mathcal{J}^*_{t,T}\mathfrak{p}, Z_j^m(\overline{U})-\Theta_{Y_j^m(\overline{U})} \rangle| \leq \epsilon^{1/30} \|\mathfrak{p}\|, \\
		\mathop{\rm sup}\limits_{k \in \mathcal{Z},l \in \{0,1\}} \mathop{\rm sup}\limits_{t \in [T/2,T]}|\langle \mathcal{J}^*_{t,T}\mathfrak{p}, [Z_j^m(U), \sigma_k^l] \rangle| \leq \epsilon^{1/30} \|\mathfrak{p}\|.
	\end{cases}
\end{align*}	
\end{proposition}

\begin{proof}

For fixed \(m \in \{0,1\}\) and \(\mathfrak{p} \in \hat{H}^n\times T_{\mathbf{p}_T}\Sigma\) let \(f_\mathfrak{p}(t) := \langle \mathcal{J}^*_{t,T}\mathfrak{p}, Y_j^m(U) \rangle = \langle\mathcal{J}^*_{t,T}\mathfrak{p}, Y_j^m(\overline{U}) \rangle\) (cf.\eqref{JFA-6.15}) so that \(\partial_t f_\mathfrak{p}(t) = \langle \mathcal{J}^*_{t,T}\mathfrak{p}, [F(U), Y_j^m(\overline{U})]-\Theta_{Y_j^m(\overline{U})}(\mathbf{p}_t) \rangle = \langle \mathcal{J}^*_{t,T}\mathfrak{p}, Z_j^m(U)-\Theta_{Y_j^m({U})}(\mathbf{p}_t) \rangle\) (see \eqref{JFA-6.8}, \eqref{JFA-5.11}). Let \(\Omega_{\epsilon,j}^1 = \Lambda_{\epsilon, 1/4}^c(\alpha=1/4)\), where \(\Lambda_{\epsilon,\alpha}\) is as in \eqref{JFA-6.10} over with \(\mathcal{I} := \{\mathfrak{p} \in \hat{H}^n\times T_{\mathbf{p}_T}\Sigma : \|\mathfrak{p}\| = 1\}\). Then, on \(\Omega_{\epsilon,j}^1\) one has, in view of \eqref{JFA-6.13},
\begin{equation}\label{JFA-6.18}
\sup_{t \in [T/2,T]} |\langle\mathcal{J}^*_{t,T}\mathfrak{p}, Y_j^m(U) \rangle| \leq \epsilon \|\mathfrak{p}\| \Longrightarrow \sup_{t \in [T/2,T]} |\langle \mathcal{J}^*_{t,T}\mathfrak{p}, Z_j^m(U)-\Theta_{Y_j^m({U})}(\mathbf{p}_t) \rangle| \leq \epsilon^{1/10} \|\mathfrak{p}\|.
\end{equation}
By Lemma \ref{JFA-lemma 6.2} with \(\alpha = 1/4\) and \eqref{JFA-6.7}, \eqref{JFA-A.5}-\eqref{JFA-A.6}, and
$$
\pi_1Y_j^m(U)=(-1)^mgj_1\pi_1\psi_j^{m+1},
$$
where $\pi_1U:=w$ denotes the first vorticity component of $U$, the precise definition can be found in Lemma \ref{NS-lemmaA.17}, we have 
\begin{align*}
\mathbb{P}((\Omega_{\epsilon,j}^1)^c) &\leq C \epsilon \mathbb{E} \left( \sup_{\mathfrak{p} \in \mathcal{I}} \| \partial_t f \|_{C^{1/4}}^{8} \right)\\
&\leq C \epsilon \exp \left( \frac{\eta}{2} V^{n+2}(U_0) \right) \Bigg[ \left( \mathbb{E} \sup_{t \in [T/2,T]} \| Z_j^m(U)-\Theta_{Y_j^m(U)}(\mathbf{p}_t) \|_{\hat{H}^n}^{16} \right)^{1/2} \\
&\quad+ \left( \mathbb{E} \| Z_j^m(U)-\Theta_{Y_j^m(U)}(\mathbf{p}_t) \|_{C^{1/4}\hat{H}^n} ^{16} \right)^{1/2} \Bigg]\\
&\leq C \epsilon |j|^{64} \exp \left( \frac{\eta}{2} V^{n+2}(U_0) \right) \Bigg[ \left( \mathbb{E} (1 + \sup_{t \in [T/2,T]} \| U \|_{\hat{H}^{n+2}}^{32} ) \right)^{1/2} \\
&\quad+ \mathbb{E} \left( \| U \|_{C^{1/4}\hat{H}^{n+2}} ^{16} (1 + \sup_{t \in [T/2,T]} \| U \|_{\hat{H}^{n+2}}^{16} ) \right)^{1/2} \Bigg]\\
&\leq C \epsilon |j|^{64} \exp (\eta V^{n+2}(U_0)),
\end{align*}
where \(C = C(\eta, T,\mathbf{p}_0)\) that is locally bounded in \( \mathbf{p}_0 \in \Sigma \), and we used the bilinearity of \(Z\) with estimates like those leading to \eqref{JFA-6.16}. Next, by expanding \(U = \overline{U} + \sigma W\) we find
\begin{align}\label{JFA-6.20}
Z_j^m(U) = Z_j^m(\overline{U}) - \sum_{\substack{k \in \mathcal{Z}\\ l \in \{0,1\}}} \alpha_k^l [Z_j^m(U), \sigma_k^l] W^{k,l}.
\end{align}
Immediately available,
\begin{align}\label{JFA1-6.20}
	Z_j^m(U)-\Theta_{Y_j^m({U})}(\mathbf{p}_t) = Z_j^m(\overline{U})-\Theta_{Y_j^m(\overline{U})}(\mathbf{p}_t) - \sum_{\substack{k \in \mathcal{Z}\\ l \in \{0,1\}}} \alpha_k^l [Z_j^m(U), \sigma_k^l] W^{k,l}.
\end{align}
Given that
\begin{align}\label{JFA-5.12}
\left[ Z_j^m (U), \sigma_k^{m'} \right] = g \left( (-1)^{m + 1} j_1 B(\psi_j^{m + 1}, \sigma_k^{m'}) + (-1)^{m'} k_1 B(\psi_k^{m' + 1}, \sigma_j^m) \right),
\end{align} 
all of the second order terms in \eqref{JFA-6.20} of the form \( [ [Z(U), \sigma_k^l], \sigma_{k'}^{l'} ] W^{k,l} W^{k',l'} \) are zero and \([Z_j^m(U), \sigma_k^l] = [Z_j^m(\overline{U}), \sigma_k^l]\).

To estimate each of the terms in \eqref{JFA1-6.20}, base on the Theorem \ref{JFA-Theorem6.4}, for \( s \in \{0, 1\} \), \(\mathfrak{p} \in \hat{H}^n\times T_{\mathbf{p}_T}\Sigma\), we introduce 
\[
\mathcal{C}_s(\phi) := \max_{k \in \mathcal{Z}, l \in \{0, 1\}} \left\{ \left\| \langle \mathcal{J}^*_{t,T}\mathfrak{p}, Z_j^m(\overline{U})-\Theta_{Y_j^m(\overline{U})}(\mathbf{p}_t) \rangle \right\|_{C^s}, |\alpha_{k}^l| \left\| \langle \mathcal{J}^*_{t,T}\mathfrak{p}, [Z_j^m(U), \sigma_k^l] \rangle \right\|_{C^s} \right\}.
\]
By Theorem \ref{JFA-Theorem6.4}, let $M=1$, then there exists a set \(\Omega_\epsilon^\#\) such that \(\mathbb{P}((\Omega_\epsilon^\#)^c) < C\epsilon\), and on \(\Omega_\epsilon^\#\) we have
\begin{align}\label{JFA1-6.21}
\sup_{t \in [T/2, T]} \left| \langle \mathcal{J}^*_{t,T}\mathfrak{p}, Z_j^m(U)-\Theta_{Y_j^m({U})}(\mathbf{p}_t) \rangle \right| \leq \epsilon^{1/10} \Longrightarrow
\begin{cases}
	\text{either } \mathcal{C}_0(\mathfrak{p}) \leq \epsilon^{1/30}, \\
	\text{or } \mathcal{C}_1(\mathfrak{p}) \geq \epsilon^{-1/90}.
\end{cases} 
\end{align}
Recalling that \(\mathcal{I} = \{\mathfrak{p} \in \hat{H}^n\times T_{\mathbf{p}_T}\Sigma : \|\mathfrak{p}\| = 1\}\), let
\[
\Omega_{\epsilon,j}^2 := \bigcap_{\mathfrak{p} \in \mathcal{I}} \{ \mathcal{C}_1(\mathfrak{p}) < \epsilon^{-1/90} \} \cap \Omega_{\epsilon,j}^\#.
\]
By \eqref{JFA-6.18}, on the set \(\Omega_{\epsilon,j} := \Omega_{\epsilon,j}^1 \cap \Omega_{\epsilon,j}^2\) we obtain the desired conclusion for each \(\epsilon < \epsilon_0(T)\). Thus it remains to estimate  \(\Omega_{\epsilon,j}^c\). By the Markov inequality we have
\begin{align}\label{JFA-6.22}
\mathbb{P}(\Omega_{\epsilon,j}^c) &\leq \mathbb{P}((\Omega_{\epsilon,j}^1)^c) + \mathbb{P}((\Omega_{\epsilon,j}^\sharp)^c) + \mathbb{P} \left( \sup_{\mathfrak{p} \in \mathcal{I}} \mathcal{C}_1(\mathfrak{p}) \geq \epsilon^{-1/90} \right)
\notag\\
&\leq C |j|^{64} \exp(\eta V^{n+2}(U_0)) \epsilon + C \epsilon \mathbb{E} \left( \sup_{\mathfrak{p} \in \mathcal{I}} (\mathcal{C}_1(\mathfrak{p}))^{90} \right).
\end{align}
First, we note that by utilizing \eqref{JFA-6.6} and combining it with Assumption \ref{assum5.4}, \eqref{JFA-A.5}, and \eqref{JFA-6.16}-\eqref{Y-estimate}, we obtain
\begin{align}\label{JFA-6.23}
&\mathbb{E} \| \langle \mathcal{J}^*_{t,T}\mathfrak{p}, Z_j^m(\overline{U})-\Theta_{Y^j_m(\overline{U})}(\mathbf{p}_t) \rangle \|_{C^1([T/2,T]; \mathbb{R})}^{90}
\notag\\
&\leq C \exp(\frac{\eta}{2} V^n(U_0)) \left( \mathbb{E} \sup_{t \in [T/2,T]} \left\| \left[F(U),Z_j^m(\overline{U}) \right]+\Theta_{Z_j^m(\overline{U})}(\mathbf{p}_t)-\Theta_{[U_t,Y^j_m(\overline{U})]_x}(\mathbf{p}_t) \right\|_{\hat{H}^n}^{180} \right)^{1/2}
\notag\\
&\leq C \exp(\frac{\eta}{2} V^n(U_0)) |j|^{90 \times 10} \left( \mathbb{E} (1 + \| U \|_{\hat{H}^8}^{3 \times 180} ) \right)^{1/2}
\notag\\
&\leq C \exp(\eta V^n(U_0)) |j|^{90 \times 10},
\end{align}
where \(C = C(\eta, T,\mathbf{p}_0)\) that is locally bounded in \( \mathbf{p}_0 \in \Sigma \). Then, due to   \eqref{JFA-5.12} and similar applications of \eqref{JFA-6.6} and \eqref{JFA-A.5}, the estimate
\[
\mathbb{E} \| \langle \mathcal{J}^*_{t,T}\mathfrak{p}, [Z_j^m(U), \sigma_k^l] \rangle \|_{C^1([T/2,T]; \mathbb{R})}^{90} \leq C \exp(\eta V^n(U_0)) |j|^{90 \times 2}
\]
follows. Combining with \eqref{JFA-6.22}-\eqref{JFA-6.23}, we obtain \eqref{JFA-6.17}. This completes the proof of the proposition.
\end{proof}

The last fundamental theorem of iteration corresponds to the brackets of the form $Y\to Z\to [Z,Y]$. For fixed \( j \in \mathbb{Z}_+^2 \), define \( \mathcal{Z}_j \) as the union of \( j \) with the set of points in \( \mathbb{Z}_+^2 \) adjacent to \( j \), that is,
\[
\mathcal{Z}_j := \{ k \in \mathbb{Z}_+^2 : k = j \pm m \text{ for some } m \in \{0\} \cup \mathcal{Z} \}.
\]
\begin{proposition}\label{JFA-Lemma6.9}
Fix \( j \in \mathbb{Z}_+^2 \). For each \( 0 < \epsilon < \epsilon_0(T) |j|^{-2} \) and \( \eta \in (0,1) \) there exist \( C = C(\eta, T) \) and a measurable set \( \Omega_{\epsilon, j} \) with
\begin{align}\label{JFA-6.25}
\mathbb{P}((\Omega_{\epsilon, j})^c) \leq C |j|^{14 \times 5400} \exp(\eta V^{n+2}(U_0)) \epsilon, 
\end{align}
such that on the set \( \Omega_{\epsilon, j} \) it holds that, for every \( \mathfrak{p} \in \hat{H}^n\times T_{\mathbf{p}_T}\Sigma \),
\begin{align}\label{JFA-6.26}
	&\sum\limits_{\substack{i \in \mathcal{Z}_j \\ m \in \{0,1\}}} 
	\sup_{t \in [T/2,T]} 
	|\langle \mathcal{J}^*_{t,T}\mathfrak{p}, Y_i^m(U) \rangle| 
	\leq \epsilon \|\mathfrak{p}\| \notag\\
	&\Longrightarrow
	\begin{cases}
	\sum\limits_{\substack{k \in \mathcal{Z} \\ m, l \in \{0,1\}}} 
	\sup\limits_{t \in [T/2,T]} 
	|\langle \mathcal{J}^*_{t,T}\mathfrak{p}, [Z_j^m(U), Y_k^l(U)] \rangle| 
	\leq \epsilon^{1/3600} \|\mathfrak{p}\|,\\
	\sup\limits_{t \in [T/2,T]}\left|\left\langle \mathcal{J}^*_{t,T}\mathfrak{p},[F(\overline U),Z_j^m(\overline U)]-\Big(\Theta_{Z_j^m(\overline U)}(\mathbf{p}_t)+\Theta_{[U_t,Y_j^m(\overline U)]_x}(\mathbf{p}_t)\Big)\right\rangle\right| \leq \epsilon^{1/1800}\|\mathfrak{p}\|.
	\end{cases}
\end{align}
\end{proposition}

\begin{proof}
We recalled that for $k\in \mathcal{Z},\,{m'}\in \{0,1\}$,
\begin{align*}
	[[F(U),Y_j^m(U)],Y_k^{m'}(U)] 
	&= [Z_j^m(U),Y_k^{m'}(U)] \\
	&= [[Z_j^m(U),F(U)],\sigma_k^{m'}] - [[Z_j^m(U),\sigma_k^{m'}],F(U)].
\end{align*}
Therefore, it suffices to find a set \(\Omega_{\epsilon,j}\) satisfying \eqref{JFA-6.25} such that on \(\Omega_{\epsilon,j}\),
\begin{align}\label{3.35.1}
\sum\limits_{\substack{i \in \mathcal{Z}_j \\ m \in \{0,1\}}} 
\sup_{t \in [T/2,T]} 
|\langle \mathcal{J}^*_{t,T}\mathfrak{p}, Y_i^m(U) \rangle| 
\leq \epsilon 
\end{align}
implies
\begin{equation}\label{JFA-6.27}
	\sum_{\substack{k \in \mathcal{Z} \\ m,l \in \{0,1\}}} 
	\sup_{t \in [T/2,T]} 
	|\langle \mathcal{J}^*_{t,T}\mathfrak{p}, [[Z_j^m(U), \sigma_k^l], F(U)] \rangle| 
	\leq \epsilon^{1/2}, 
\end{equation}
and
\begin{equation}\label{JFA-6.28}
	\sum_{\substack{k \in \mathcal{Z} \\ m,l \in \{0,1\}}} 
	\sup_{t \in [T/2,T]} 
	|\langle \mathcal{J}^*_{t,T}\mathfrak{p}, [[Z_j^m(U), F(U)], \sigma_k^l] \rangle| 
	\leq \epsilon^{1/1800}. 
\end{equation}	
We begin by proving the \eqref{JFA-6.27}. Noting that $j \pm k \in \mathcal{Z}_j$, then we obtain 
\begin{align}\label{JFA-6.29}
&|\langle \mathcal{J}^*_{t,T}\mathfrak{p}, [[Z_j^m(U), \sigma_k^l], F(U)] \rangle| \notag\\
&= |g(j^{\perp} \cdot k)| \cdot
\Big|\Big\langle \mathcal{J}^*_{t,T}\mathfrak{p}, \Big[\Big(\frac{j_1}{|j|^2}+\frac{k_1}{|k|^2}\Big) \sigma_j^{m+l+1} + (-1)^{l+1} \Big(\frac{j_1}{|j|^2}-\frac{k_1}{|k|^2}\Big) \sigma_j^{m+l+1}, F(U)\Big] \Big\rangle\Big|\notag\\ 
&\leq C |j| \left( |\langle \mathcal{J}^*_{t,T}\mathfrak{p}, Y_{j+k}^{m+l+1}(U) \rangle| 
+ |\langle \mathcal{J}^*_{t,T}\mathfrak{p}, Y_{j-k}^{m+l+1}(U) \rangle| \right) \notag\\
&\leq C |j| \epsilon \| \mathfrak{p} \| .
\end{align}
Then \eqref{JFA-6.27} follows for any \(\epsilon < (C |j|)^{-2}\).

Now we only need to prove that \eqref{JFA-6.28} can be derived from \eqref{3.35.1} on a suitable set. First, note that by Proposition \ref{JFA-Lemma6.8}, there exists a measurable set $\Omega^1_{\epsilon, j}$ such that 
$$
\mathbb{P}\big(({\Omega^1_{\epsilon,j}})^c\big) \leq C|j|^{90 \times 10} \exp(\eta V^{n+2}(U_0)) \epsilon,
$$
and on the set $\Omega^1_{\epsilon, j}$, for ecah $m\in \{0,1\}$ and each $\mathfrak{p} \in \mathcal{I}=\{\mathfrak{p}\in \hat{H}^n\times T_{\mathbf{p}_T}\Sigma:\|\mathfrak{p}\|=1\}$
\begin{align}\label{JFA-6.30}
	\sup_{t \in [T/2,T]} \left| \left\langle \mathcal{J}_{t,T}^{*}\mathfrak{p}, Y_j^{m}(U) \right\rangle \right| \leq \epsilon \| \mathfrak{p} \|\Rightarrow \sup_{t \in [T/2,T]} \left| \left\langle \mathcal{J}_{t,T}^{*}\mathfrak{p}, Z_j^{m}(\overline{U}) - \Theta_{Y_j^{m}(\overline{U})}(\mathbf{p}_t)\right\rangle \right| \leq \epsilon^{1/30} \| \mathfrak{p} \|.
\end{align}
Setting $f_{\mathfrak{p}}(t):=\langle \mathcal{J}_{t,T}^{*}\mathfrak{p}, Z_j^{m}(\overline{U}) - \Theta_{Y_j^{m}(\overline{U})}(\mathbf{p}_t)\rangle$, then $$\partial_tf_{\mathfrak{p}}(t)=\left\langle \mathcal{J}_{t,T}^{*}\mathfrak{p}, [F(U),Z_j^{m}(\overline{U})] - \left(\Theta_{Z_j^{m}(\overline{U})}(\mathbf{p}_t)+\Theta_{[U_t,Y_j^{m}(\overline{U})]_x}(\mathbf{p}_t)\right)\right\rangle.$$ 
Let \(\Omega_{\varepsilon,j}^2 := \Lambda_{\varepsilon^{1/30}, 1/4}^c\) with $\alpha=1/4$ as in Lemma \ref{JFA-lemma 6.2}. Thus on \(\Omega_{\varepsilon,j}^3 := \Omega_{\varepsilon,j}^1 \cap \Omega_{\varepsilon,j}^2\) we have for each \(m \in \{0,1\}\), and \(\mathfrak{p} \in \mathcal{I}\)
\begin{align}\label{JFA-6.31}
&\sup_{t \in [T/2,T]} |\langle \mathcal{J}_{t,T}^{*}\mathfrak{p}, Y_j^m(U) \rangle| \leq \varepsilon \notag\\
&\Rightarrow \sup_{t \in [T/2,T]} \left|\left\langle \mathcal{J}_{t,T}^{*}\mathfrak{p}, [F(U),Z_j^{m}(\overline{U})] - \left(\Theta_{Z_j^{m}(\overline{U})}(\mathbf{p}_t)+\Theta_{[U_t,Y_j^{m}(\overline{U})]_x}(\mathbf{p}_t)\right)\right\rangle\right| \leq \varepsilon^{1/300}.
\end{align}
Similarly as in the proof of Proposition \ref{JFA-Lemma6.8}, employing the same estimates we can obtain 
\[
\mathbb{P}((\Omega_{\varepsilon,j}^2)^c) \leq C \varepsilon \mathbb{E} \left( \sup_{\mathfrak{p} \in I} \|g'_\mathfrak{p}(t)\|_{C^{1/4}([T/2,T])}^{8 \times 30} \right) \leq C \varepsilon |j|^{240 \times 10} \exp(\eta V^{n+2}(U_0)).
\]
Subsequently, we aim to establish a result for $Z^m_j(U)$ similar to \eqref{JFA-6.31}, to facilitate the subsequent noise separation, yielding the form $\langle \mathcal{J}^*_{t,T}\mathfrak{p}, [[Z_j^m(U), F(U)], \sigma_k^l] \rangle$, for which we use the expansion \eqref{JFA-6.20}. Proceeding analogously to equation \eqref{JFA-6.29}, we can obtain 
\begin{align}\label{JFA-6.32}
\sup_{t \in [T/2,T]} |\langle \mathcal{J}_{t,T}^{*}\mathfrak{p}, [[Z_j^m(U), \sigma_k^l], F(U)]\rangle|\,|W^{k,l}(t)| \leq C|j|\epsilon\|\mathfrak{p}\| \sup_{t \in [T/2,T]} |W^{k,l}(t)|. 
\end{align}
Noted that \(\mathbb{E}\|W^{k,l}\|_{L^\infty} < \infty\), then by Markov inequality one has \(\mathbb{P}((\Omega_{\epsilon,k,l}^4)^c) \leq C \epsilon^{1/2}\), where \(\Omega_{\epsilon,k,l}^4 := \{\sup_{t \in [T/2,T]} |W^{k,l}(t)| \leq \epsilon^{-1/2}\}\). Combining equations \eqref{JFA-6.20}, \eqref{JFA-6.31}, and \eqref{JFA-6.32}, we know that on the set \(\Omega_{\epsilon,j}^5 := \Omega_{\epsilon,j}^3 \cap \Omega_{\epsilon,j}^4\), for any \(\epsilon < \epsilon_0(T)|j|^{-2}\) it holds that
\begin{align*}
&\sum_{\substack{i \in \mathcal{Z}_j \\ m \in \{0,1\}}} \sup_{t \in [T/2,T]} |\langle \mathcal{J}_{t,T}^{*}\mathfrak{p}, Y_i^m(U) \rangle| \leq \epsilon\\
&\Rightarrow \sum_{m \in \{0,1\}} \sup_{t \in [T/2,T]} \left|\left\langle \mathcal{J}_{t,T}^{*}\mathfrak{p}, [Z_j^m(U), F(U)]- \left(\Theta_{Z_j^{m}(\overline{U})}(\mathbf{p}_t)+\Theta_{[U_t,Y_j^{m}(\overline{U})]_x}(\mathbf{p}_t)\right) \right\rangle\right| \leq \epsilon^{1/600}.
\end{align*}
Proceeding analogously to the proof of Proposition \ref{JFA-Lemma6.8}, to separate the noise component, we expand \([Z_j^m(U), F(U)]\) with respect to \(U = \overline{U} + \sigma W\) and obtain
\begin{align*}
[Z_j^m(U), F(U)]&=\Bigg[Z_j^m(\overline{U})-\sum_{\substack{k \in \mathcal{Z}\\ l \in \{0,1\}}}\alpha_k^l[Z_j^m(\overline{U}),\sigma_k^l]W^{k,l},\,F(\overline{U})-\sum_{\substack{k \in \mathcal{Z}\\ l \in \{0,1\}}}\alpha_k^l[F(\overline{U}),\sigma_k^l]W^{k,l}\Bigg]\notag\\
&=[Z_j^m(\overline{U}),F(\overline{U})]-\sum_{\substack{k \in \mathcal{Z}\\ l \in \{0,1\}}}\alpha_k^l[[Z_j^m(\overline{U}),F(\overline{U})],\sigma_k^l]W^{k,l},
\end{align*}
In this sequence, the last equality is identical to that in equation \eqref{JFA-6.20}, where all higher-order terms vanish. It follows immediately that
\begin{align*}
	&[Z_j^m(U), F(U)]-\left(\Theta_{Z_j^{m}(\overline{U})}(\mathbf{p}_t)+\Theta_{[U_t,Y_j^{m}(\overline{U})]_x}(\mathbf{p}_t)\right)\\
	&=\left([Z_j^m(\overline{U}),F(\overline{U})]-\left(\Theta_{Z_j^{m}(\overline{U})}(\mathbf{p}_t)+\Theta_{[U_t,Y_j^{m}(\overline{U})]_x}(\mathbf{p}_t)\right)\right)-\sum_{\substack{k \in \mathcal{Z}\\ l \in \{0,1\}}}\alpha_k^l[[Z_j^m(\overline{U}),F(\overline{U})],\sigma_k^l]W^{k,l}.
\end{align*}
Then again we use Theorem \ref{JFA-Theorem6.4} to establish
\begin{align*}
&\sup_{t \in [T/2,T]} |\langle \mathcal{J}_{t,T}^{*}\mathfrak{p}, [Z_j^m(U), F(U)] \rangle| \leq \epsilon^{1/600}\|\mathfrak{p}\|\\
&\Longrightarrow
\begin{cases}
\sum\limits_{\substack{k \in \mathcal{Z}\\ l \in \{0,1\}}} \sup\limits_{t \in [T/2,T]} |\langle \mathcal{J}_{t,T}^{*}\mathfrak{p}, [[Z_j^m(\overline{U}), F(\overline{U})], \sigma_k^l] \rangle| \leq \epsilon^{1/1800}\|\mathfrak{p}\|,\\
\sup\limits_{t \in [T/2,T]}\left|\left\langle \mathcal{J}^*_{t,T}\mathfrak{p},[F(\overline U),Z_j^m(\overline U)]-\Big(\Theta_{Z_j^m(\overline U)}(\mathbf{p}_t)+\Theta_{[U_t,Y_j^m(\overline U)]_x}(\mathbf{p}_t)\Big)\right\rangle\right| \leq \epsilon^{1/1800}\|\mathfrak{p}\|
\end{cases}
\end{align*}
on a set \(\Omega_{\epsilon}^{6}\), where \(\Omega_{\epsilon}^{6}\) satisfies
\begin{align}\label{JFA-6.33}
\mathbb{P}\left((\Omega_{\epsilon}^{6})^{c}\right) \leq C|j|^{14 \times 5400} \exp(\eta V^{n}(U_0))\epsilon. 
\end{align}
Let \(\Omega_{\epsilon,j} := \Omega_{\epsilon,j}^{5} \cap \Omega_{\epsilon,j}^{6}\), and note that
\[
[[Z_{j}^{m}(\overline{U}), F(\overline{U})], \sigma_{k}^{l}] = [[Z_{j}^{m}(U), F(U)], \sigma_{k}^{l}].
\]
Thus, we complete the proof of the proposition.
\end{proof}
Subsequently, by inductively iterating the base steps above (Propositions \ref{JFA-Lemma6.5}–\ref{JFA-Lemma6.9}), we arrive at the following key proposition, which is crucial for establishing the probabilistic spectral bound of the Malliavin matrix.
\begin{proposition}\label{NS-prop5.18}
For any \( N \geq 2 \), there exist \( \gamma_N >0\) and \( \epsilon_1= \epsilon_1（(N,T)>0\) such that for every \( \epsilon \in (0, \epsilon_1) \), \( \eta \in (0,1) \) and \( \mathfrak{p} \in \hat{H}^n\times T_{\mathbf{p}_T}\Sigma \), there exists a set \( \Omega_{\epsilon,N}^* \) with
\[
\mathbb{P}\bigl((\Omega_{\epsilon,N}^*)^{c}\bigr) \leq CN^{q} \exp(\eta V^{n+2}(U_0)) \epsilon^{\gamma_{N}},
\]
where \( q>1,C = C(\eta, T, \mathbf{p}_0) >0\), locally bounded in $\mathbf{p}_0$, such that on the set \( \Omega_{\epsilon,N}^* \) one has
\begin{align*}
\langle \mathcal{M}_T\mathfrak{p}, \mathfrak{p} \rangle \leq \epsilon \|\mathfrak{p}\|^2 \Longrightarrow \max_{|j|\le N, m\in \{0,1\}}\sup_{t\in [T/2,T]}|\langle \mathcal{J}^*_{t,T}\mathfrak{p}, \sigma_j^m \rangle| \leq \epsilon^{\gamma_N} \|\mathfrak{p}\|,
\end{align*}
\begin{align*}
\langle \mathcal{M}_T\mathfrak{p}, \mathfrak{p} \rangle \leq \epsilon \|\mathfrak{p}\|^2 \Longrightarrow \max_{|j|\le N, m\in \{0,1\}}|\langle  \psi_j^m + J_{j,m}^N(U_T),\mathfrak{p} \rangle| \leq \epsilon^{\gamma_N} \|\mathfrak{p}\|,
\end{align*}
and 
\begin{align*}
	&\langle \mathcal{M}_T\mathfrak{p}, \mathfrak{p} \rangle \leq \epsilon \|\mathfrak{p}\|^2\\ &\Longrightarrow
	\max_{|j|\le N, m\in \{0,1\}}\left|\left\langle [F(\overline U_T),Z_j^m(\overline U_T)]-\Big(\Theta_{Z_j^m(\overline U_T)}(\mathbf{p}_T)+\Theta_{[U_T,Y_j^m(\overline U_T)]_x}(\mathbf{p}_T)\Big),\mathfrak{p}\right\rangle\right| \leq \epsilon^{\gamma_N} \|\mathfrak{p}\|,
\end{align*}
where $J_{j,m}^{{N}}(U):=Q_{{N}}J_{j,m}(U)$, and
\begin{equation}\label{Jm}
	J_{j,m}(U) = 
	\begin{cases} 
		(-1)^m \dfrac{\nu_2 \lvert j \rvert^2}{g j_1} \sigma_j^{m+1} + (-1)^m \dfrac{1}{g j_1} B(U, \sigma_j^{m+1}) & \text{if } j_1 \neq 0, \\[6pt]
		\dfrac{1 + \lvert j \rvert^2}{g^2 \lvert j \rvert^3} \bigl(-H_{j+e_1,e_1}^{0,0}(U) - H_{j+e_1,e_1}^{1,1}(U)\bigr) & \text{if } j_1 = 0,\ m = 0, \\[6pt]
		\dfrac{1 + \lvert j \rvert^2}{g^2 \lvert j \rvert^3} \bigl(-H_{j+e_1,e_1}^{0,1}(U) + H_{j+e_1,e_1}^{1,0}(U)\bigr) & \text{if } j_1 = 0,\ m = 1,
	\end{cases}
\end{equation}
here $e_1:=(1,0)$, \( U \mapsto H_{j,k}^{m,m'}(U) \) is affine and it is concentrated entirely in the \( \theta \) component.
\end{proposition}

\begin{remark} \label{E1}
Indeed, we can provide an explicit form for $\epsilon_1（(N,T)$ (see \cite[Proposition 4.4]{JFA}), namely 
$$\epsilon_1（(N,T):=\min\left\{1,{\left(\frac{C(T)}{N}\right)}^{q_1^N}\right\},$$ 
where $q_1>1$. Moreover, the explicit form of $\gamma_{N}$ is given by $\gamma_{N}:=q_2^{N}$, where $0<q_2<1$.
\end{remark}
\begin{proof}[Proof of Proposition \ref{NS-prop5.18}]
By utilizing Propositions \ref{JFA-Lemma6.5}–\ref{JFA-Lemma6.9} and Assumption \ref{assum5.4}, we complete the proof of the proposition by following the iterative inductive scheme outlined in \cite[Section 6.3]{JFA}.
\end{proof}
The following lemma provides the moment estimates for $J_{j,m}^{\tilde{N}}(U)$ required in the subsequent Proposition \ref{NS-prop5.21}.
\begin{lemma}\label{JFA-lemma5.6}
	For every integers \( N \), \(\tilde{N}\) with \(\tilde{N} \geq N > 0\), and every integer \( \sigma\geq s \geq 1 \) and \( U \in H^{\sigma+1} \)
	\begin{equation}\label{JFA-5.27}
		\|J_{j,m}^{\tilde{N}}(U)\|_{\hat{H}^s} \leq C \frac{N^{\sigma+3}}{\tilde{N}^{(\sigma-s)/2}} (1 + \|U\|_{\hat{H}^{\sigma+1}}) \quad (|j| \leq N, \, m \in \{0, 1\}),
	\end{equation}
	where \( C = C(s) \) is independent of \( N \), \(\tilde{N}\) and \( U \).
\end{lemma}

\begin{proof}
	Since the \(\tilde{N}\)th eigenvalue \(\lambda_{\tilde{N}} \sim \tilde{N}\), cf. \cite{long-3}, one has by the generalized Poincaré inequality that
	\begin{equation}
		\|J_{j,m}^{\tilde{N}}(U)\|_{\hat{H}^s} \leq C \frac{1}{\lambda_{\tilde{N}}^{(\sigma-s)/2}} \|J_{j,m}(U)\|_{\hat{H}^\sigma} \leq C \frac{1}{\tilde{N}^{(\sigma-s)/2}} \|J_{j,m}^{\tilde{N}}(U)\|_{\hat{H}^\sigma}.
	\end{equation}
	Noting that \( J_{j,m} \) is affine in \( U \), we obtain
	\begin{equation}
		\|J_{j,m}^{\tilde{N}}(U)\|_{\hat{H}^\sigma} \leq C (1 + \|U\|_{\hat{H}^{\sigma+1}}),
	\end{equation}
	where \( C = C(N,s) \). The exact dependence of the right hand side on \( N \) can be inferred from the fact that each derivative of \( J_{j,m}(U) \) can produce at most one factor of \( |j| \leq N \). 
\end{proof}

In fact, if we ignore the manifold direction, the first two implication steps alone would suffice to prove Theorem \ref{bounds-Malliavin}. It is important to note that only the final implication step involves the manifold direction. Our objective is to isolate the manifold component arising from this implication, thereby controlling the overlap of vectors with nontrivial components in the tangent space $T\Sigma$.

Note that the first two implications control only finitely many Fourier modes in each iteration, whereas $\overline{U}_T$ is typically supported on infinitely many Fourier modes.
Furthermore, unlike in \cite{NS}, the vector fields generated by Lie brackets are $\{\sigma_j^m\}$ and $\{\psi_j^m+ J_{j,m}^{\tilde{N}}(U_T)\}$, which now depend on the stochastic process $U_T$. In the proof, we employ a soft argument, using the first two implications for infinitely many $N$ to control the overlap of $\mathfrak{p}$ with all Fourier modes. This is achieved by constructing $h(\epsilon)$ and $\tilde{N}(\epsilon)$ to control $\|U_T\|_{\hat{H}^4}^2$. However, particular care must now be taken to ensure that the constructed $\tilde{N}(\epsilon)$ is well-defined within the context of Proposition \ref{NS-prop5.18} and to avoid circular reasoning. This requirement also results in worse integrability compared to the situation in \cite{NS}.
\begin{definition} 
We denote by $\pi_{\hat{H}^n} \colon \hat{H}^n \times T\Sigma \to \hat{H}^n$ the projection onto the $\hat{H}^n$ coordinate and $\pi_{\scriptscriptstyle{T\Sigma}} \colon \hat{H}^n \times T\Sigma \to T\Sigma$ the projection onto the $T\Sigma$ coordinate.
\end{definition} 

In the proposition below, we simultaneously control infinitely many Fourier modes of the $\hat{H}^n$ component of $\mathfrak{p}$, proving that with high probability its corresponding Sobolev norm does not become too large. 
Furthermore, we note that the rate $r(\epsilon)$ can be made arbitrarily slow. Consequently, the resulting bounds exhibit weak dependence on $\epsilon$.
\begin{proposition}\label{NS-prop5.21}
	For any $N\le \tilde{N}$, $T>0$, $\eta\in (0,1)$, and $\mathbf{p}_0\in \Sigma$, there exists $C(T,\eta,N,\mathbf{p}_0)>0$, locally bounded in $\mathbf{p}_0$, and a positive constant $\epsilon^*:=\epsilon^*(T,\eta,N,\mathbf{p}_0)$, such that for any $0<\epsilon<\epsilon^*$, 
\begin{align*}
	\mathbb{P}\left(\exists \mathfrak{p} \in \hat{H}^n\times T_{\mathbf{p}_T}\Sigma,\|\mathfrak{p}\|=1\wedge \langle \mathcal{M}_T\mathfrak{p}, \mathfrak{p} \rangle \leq \epsilon \wedge \|\pi_{\hat{H}^n}\mathfrak{p}\|_{\hat{H}^{n-2}}\ge g(\epsilon)\right)\le Cr(\epsilon)\exp(V^{n+2}(U_0)),
\end{align*}
where $g:(0,1)\to [0,1]$ satisfies $\lim_{\epsilon\to 0}g(\epsilon)=0$ locally uniformly in $\mathbf{p}_0$, and $r(\epsilon)$ is a nonnegative decreasing function vanishing as $\epsilon\to 0$, also locally uniformly in $\mathbf{p}_0$.
\end{proposition}

\begin{proof}
Fix $\eta\in (0,1)$, for any $N\le \tilde{N}$, define
$$ 
\langle\mathcal{Q}_{N,\tilde{N}}(U)\pi_{\hat{H}^n}\mathfrak{p},\pi_{\hat{H}^n}\mathfrak{p}\rangle
:=\sum_{\substack{\tilde{b}\in \mathcal{B}_{N,\tilde{N}}(U)}}|\langle\pi_{\hat{H}^n}\mathfrak{p},\tilde{b}(U)\rangle_{\hat{H}^{n-2}}|^2,
$$
where $\mathcal{B}_{N,\tilde{N}}(U):=\{\sigma_j^m,\psi_j^m+J_{j,m}^{\tilde{N}}(U):m\in\{0,1\},\,j\in \mathbb{Z}^2_+,\,|j|\le N\}$ and $J_{j,m}^{\tilde{N}}(U):=Q_{\tilde{N}}J_{j,m}(U)$.
By the first two implications of Proposition \ref{NS-prop5.18}, we have, on the set \( \Omega_{\epsilon,\tilde{N}}^* \),
$$
\langle \mathcal{M}_T\mathfrak{p}, \mathfrak{p} \rangle \leq \epsilon \Longrightarrow \langle\mathcal{Q}_{N,\tilde{N}}(U_T)\pi_{\hat{H}^n}\mathfrak{p},\pi_{\hat{H}^n}\mathfrak{p}\rangle\le \epsilon^{\gamma_{\tilde{N}}}
$$
with
\begin{align}\label{JFA-4.7}
\mathbb{P}\bigl((\Omega_{\epsilon,\tilde{N}}^*)^{c}\bigr) \leq C\tilde{N}^{q} \exp(\eta V^{n+2}(U_0)) \epsilon^{\gamma_{\tilde{N}}}.
\end{align}
Simultaneously, we observe that
\begin{align*}
&\|\pi_{\hat{H}^n}\mathfrak{p}\|^2_{\hat{H}^{n-2}}\\
&\le N^{-4}+\langle\mathcal{Q}_{N,\tilde{N}}(U)\pi_{\hat{H}^n}\mathfrak{p},\pi_{\hat{H}^n}\mathfrak{p}\rangle-\sum_{\substack{|j|\le N \\ m \in \{0,1\}}}2\left(\big\langle\psi_j^m,\pi_{\hat{H}^n}\mathfrak{p}\big\rangle_{\hat{H}^{n-2}}\big\langle J_{j,m}^{\tilde{N}},\pi_{\hat{H}^n}\mathfrak{p}\big\rangle_{\hat{H}^{n-2}}+\big|\big\langle J_{j,m}^{\tilde{N}},\pi_{\hat{H}^n}\mathfrak{p}\big\rangle_{\hat{H}^{n-2}}\big|^2\right)\\
&:=N^{-4}+\langle\mathcal{Q}_{N,\tilde{N}}(U)\pi_{\hat{H}^n}\mathfrak{p},\mathfrak{p}\rangle-Q.
\end{align*}
Noted that $Q\ge -\frac{1}{2}\|\pi_{\hat{H}^n}\Pi_{N}\mathfrak{p}\|^2_{\hat{H}^{n-2}}-\sum\limits_{\substack{|j|\le N,m \in \{0,1\}}}\|J_{j,m}^{\tilde{N}}(U_T)\|^2_{\hat{H}^{n-2}}$. Thus, on the set \( \Omega_{\epsilon,\tilde{N}}^* \) one has 
\begin{align*}
\frac{1}{2}\|\pi_{\hat{H}^n}\mathfrak{p}\|^2_{\hat{H}^{n-2}}&\le N^{-4}+\epsilon^{\gamma_{\tilde{N}}}+\sum\limits_{\substack{|j|\le N,m \in \{0,1\}}}\|J_{j,m}^{\tilde{N}}(U_T)\|^2_{\hat{H}^{n-2}} \\
&\le N^{-4}+\epsilon^{\gamma_{\tilde{N}}}+C\frac{N^{16}}{{\tilde{N}}}(1+\|U_T\|_{\hat{H}^4}^2),
\end{align*} 
where the last line employs Lemma \ref{JFA-lemma5.6} with $s=n,\,\sigma=n+1$. Note that $U_T$, being a solution to the equation \eqref{Boussinesq-w1}, is unlikely to admit bounds controlled by a deterministic constant. Thus, we modify the set $\Omega_{\epsilon,\tilde{N}}^*$ by intersecting it with a set $\Omega_{\epsilon,\tilde{N}}$ that quantifies bounds on $U_T$ via a function $h$, which becomes unbounded as $\epsilon\to 0$. It is important to note that caution is warranted during the construction of $h$ and $\tilde{N}$ to avoid circular reasoning. For each \(\epsilon \in (0, 1/e)\), let
\begin{equation} \label{JFA-4.10}
	h(\epsilon) := \log(\log(\log(\epsilon^{-1}))), \quad 
	\tilde{N}(\epsilon) := \left\lceil \frac{C N^{16} h(\epsilon)}{{h(\epsilon)}^{-1/2}-N^{-4}} \right\rceil,
\end{equation}
where \(\lceil x\rceil\) denotes the smallest integer larger or equal to \(x\). Then
\begin{equation}\label{JFA-4.11}
	C\frac{N^{16}}{{\tilde{N}(\epsilon)}} h(\epsilon) \leq {h(\epsilon)}^{-1/2}-N^{-4}.
\end{equation}
Observe that, there exists \(\epsilon^* > 0\) such that \(\epsilon < \epsilon_1(\tilde{N}(\epsilon))\) whenever \(\epsilon < \epsilon^*\), where \(\epsilon_1\) is defined as in the Remark \ref{E1}. Indeed, one can begin by supposing \(\epsilon^*< 1/e\) and observe that
\begin{align*}
	\limsup_{\epsilon \to 0^+} \epsilon \left( \frac{C(T)}{\tilde{N}(\epsilon)} \right)^{-q_1^{\tilde{N}(\epsilon)}} &\leq \limsup_{\epsilon \to 0^+} \epsilon h(\epsilon)^{3q_1^{\tilde{N}(\epsilon)}} = \limsup_{\epsilon \to 0^+} \epsilon \exp\left( 3q_1^{\tilde{N}(\epsilon)} \log(h(\epsilon)) \right)\\
	&\leq \limsup_{\epsilon \to 0^+} \epsilon \exp\left( \exp(h(\epsilon)^3) \right)\\
	&= \limsup_{h \to \infty} \exp(-\exp(\exp(h))) \exp\left( \exp(h^3) \right) = 0.
\end{align*} 
Thus for any $\epsilon<\epsilon^*$, $\Omega_{\epsilon,\tilde{N}(\epsilon)}^*$ (see Proposition \ref{NS-prop5.18}) is well defined. Then we set
\begin{align*}
	\Omega_\epsilon := {\Omega}_{\epsilon,U,h} \cap \Omega_{\epsilon,\tilde{N}(\epsilon)}^{*}, 
\end{align*}
where ${\Omega}_{\epsilon,U,h}:= \{1 + \|U_T\|_{\hat{H}^4}^2 \leq h(\epsilon)\}$. Then by \eqref{JFA-4.7}, \eqref{JFA-A.5} and the Markov inequality, we have 
\begin{align}\label{JFA-4.12}
	\mathbb{P}(\Omega_\epsilon^c) &\leq \mathbb{P}({\Omega}_{\epsilon,U,h}^c) + \mathbb{P}((\Omega_{\epsilon,\tilde{N}(\epsilon)}^{*})^c) \nonumber \leq C \left( \frac{1}{h(\epsilon)} + (\tilde{N}(\epsilon))^{q} \epsilon^{\gamma_{\tilde{N}(\epsilon)}} \right) \exp(\eta V^{n+2}(U_0)) \nonumber \\
	&=: r(\epsilon) \exp(\eta V^{n+2}(U_0)) 
\end{align}
whenever \(\epsilon < \epsilon^*\). since \(h(\epsilon) \to \infty\) as \(\epsilon \to 0^+\), the later quantity \(r(\epsilon)\) decays to zero as \(\epsilon \to 0^+\). Indeed, we have that
\begin{align}\label{JFA-4.13}
	\limsup_{\epsilon \to 0^+} \tilde{N}(\epsilon)^{q} \epsilon^{\gamma_{\tilde{N}}} &=\limsup_{\epsilon \to 0^+} \tilde{N}(\epsilon)^{q} \epsilon^{{q_2}^{\tilde{N}}}\leq \limsup_{\epsilon \to 0^+} h(\epsilon)^{3q} e^{\exp(-h(\epsilon)^3)}\notag\\
	&\leq \exp\left( \limsup_{\epsilon \to 0^+} \left(3q \log h(\epsilon) + \exp(-h(\epsilon)^3) \log \epsilon\right) \right) \notag\\
	&= \exp\left( \limsup_{h \to \infty} \left(3q \log(h) - \exp(-h^3) \exp(\exp(h))\right) \right) = 0. 
\end{align}

Finally, choosing $g(\epsilon)>2\left(\epsilon^{\gamma_{\tilde{N}}}+{h(\epsilon)}^{-1/2}\right)$ with $\lim_{\epsilon\to 0}g(\epsilon)=0$ locally uniformly in $\mathbf{p}_0$, the proposition follows.
\end{proof}
\begin{remark}
In fact, from the above proof we can give an explicit expression for $r$, namely $r(\epsilon)=C \left( \frac{1}{h(\epsilon)} + (\tilde{N}(\epsilon))^{q} \epsilon^{\gamma_{\tilde{N}(\epsilon)}} \right)$.
\end{remark}

Next, we use the above proposition-which controls all Fourier modes of the $\hat{H}^n$-component of $\mathfrak{p}$-together with the final implication from Proposition \ref{NS-prop5.18} and the Assumption \ref{assum5.5} to control the overlaps of $\mathfrak{p}$ with the $T\Sigma$ directions. It is also noteworthy that the decay rate $f(\epsilon)$ can similarly be made arbitrarily slow. In fact, $f(\epsilon)$ can be made slower than $r(\epsilon)$, exhibiting even weaker dependence on $\epsilon$.

\begin{proposition}\label{NS-prop5.22}
For any $\epsilon \in (0,1)$, $T>0$, $\eta\in (0,1)$, and $\mathbf{p}_0\in \Sigma$, there exists $C(T,\eta,N,\mathbf{p}_0)>0$, locally bounded in $\mathbf{p}_0$, and a positive constant $\epsilon^*:=\epsilon^*(T,\eta,N,\mathbf{p}_0)$, such that for any $0<\epsilon<\epsilon^*$, 
\[
\mathbb{P}
\bigg( 
\exists \mathfrak{p} \in \hat{H}^n \times T_{\mathbf{p}_T}\Sigma,\ 
\|\mathfrak{p}\| = 1\wedge\ 
\langle \mathcal{M}_T\mathfrak{p}, \mathfrak{p} \rangle \leq \epsilon\wedge\ 
\| \pi_{\scriptscriptstyle{T\Sigma}} \mathfrak{p} \| > g(\epsilon)
\bigg)
< f(\epsilon)\exp(\eta V^{n+2}(U_0)),
\]
where $g:(0,1)\to [0,1]$ satisfies $\lim_{\epsilon\to 0}g(\epsilon)=0$ locally uniformly in $\mathbf{p}_0$, and $f(\epsilon)$ is a nonnegative decreasing function vanishing as $\epsilon\to 0$, also locally uniformly in $\mathbf{p}_0$.
\end{proposition}

\begin{proof}
Fix \(\eta \in (0,1)\). Let \(N \ge 2\) be such that Assumption \ref{assum5.5} applies. By Proposition 5.18, there exists \(\gamma_{\tilde{N}} > 0\) such that
\[
\mathbb{P} \left( \exists \mathfrak{p},\ \|\mathfrak{p}\| = 1 \wedge \langle \mathcal{M}_T \mathfrak{p}, \mathfrak{p} \rangle \leq \varepsilon \wedge \max_{|j|\le N, m\in \{0,1\}}\left|\left\langle  [F(\overline U_T),Z_j^m(\overline U_T)]-\widehat{\Theta}_{U_T}(\mathbf{p}_T) ,\mathfrak{p}\right\rangle\right| \geq \varepsilon^\gamma \right) \leq C \epsilon^{\gamma_{\tilde{N}}} e^{\eta V^{n+2}(U_0)},
\]
where \(\widehat{\Theta}_{U_T}(\mathbf{p}_T) := \Theta_{Z_j^m(\overline U_T)}(\mathbf{p}_T)+\Theta_{[U_T,Y_j^m(\overline U_T)]_x}(\mathbf{p}_T)\) denotes the new direction on the tangent bundle manifold $T\Sigma$.
Thus for any $g(\epsilon)>2\epsilon^{\gamma_{\tilde{N}}}$, we have that
\begin{align}\label{NS-5.23}
	&\mathbb{P} \left( \exists \mathfrak{p},\ \|\mathfrak{p}\| = 1 \wedge \langle \mathcal{M}_T \mathfrak{p}, \mathfrak{p} \rangle \leq \epsilon \wedge \max_{|j|\le N, m\in \{0,1\}} |\langle \widehat{\Theta}_{U_T}(\mathbf{p}_T), \mathfrak{p} \rangle| \geq g(\epsilon) \right) \notag\\
	&\leq C \epsilon^{\gamma_{\tilde{N}}} e^{\eta V^{n+2}(U_0)} + \mathbb{P} \left( \exists \mathfrak{p},\ \|\mathfrak{p}\| = 1 \land \langle \mathcal{M}_T \mathfrak{p}, \mathfrak{p} \rangle \leq \epsilon \wedge \max_{|j|\le N, m\in \{0,1\}} |\langle [F(\overline U_T),Z_j^m(\overline U_T)], \mathfrak{p} \rangle| \geq g(\epsilon)/2 \right) \notag\\
	&= C \epsilon^{\gamma_{\tilde{N}}} e^{\eta V^{n+2}(U_0)} + \mathbb{P} \left( \exists \mathfrak{p},\ \|\mathfrak{p}\| = 1 \land \langle \mathcal{M}_T \mathfrak{p}, \mathfrak{p} \rangle \leq \epsilon \wedge \max_{|j|\le N, m\in \{0,1\}} 2|\langle [F(\overline U_T),Z_j^m(\overline U_T)], \mathfrak{p} \rangle| \geq g(\epsilon) \right) \notag\\
	&\leq C \epsilon^{\gamma_{\tilde{N}}} e^{\eta V^{n+2}(U_0)} + \mathbb{P} \left( \exists \mathfrak{p},\ \|\mathfrak{p}\| = 1 \land \langle \mathcal{M}_T \mathfrak{p}, \mathfrak{p} \rangle \leq \epsilon \wedge 2\max_{|j|\le N, m\in \{0,1\}}\|[F(\overline U_T),Z_j^m(\overline U_T)]\|_{\hat{H}^n} \|\mathfrak{p}\|_{\hat{H}^{n-2}} \geq g(\epsilon) \right),\notag\\
	&\leq C \epsilon^{\gamma_{\tilde{N}}} e^{\eta V^{n+2}(U_0)} + \mathbb{P} \left( \exists \mathfrak{p},\ \|\mathfrak{p}\| = 1 \land \langle \mathcal{M}_T \mathfrak{p}, \mathfrak{p} \rangle \leq \epsilon \wedge C(1+\|\overline{U}_T\|^3_{\hat{H}^6}) \|\mathfrak{p}\|_{\hat{H}^{n-2}} \geq g(\epsilon) \right).
\end{align}
For the second item in the last row, we have
\begin{align}\label{NS-5.24}
	&\mathbb{P} \left( \exists \mathfrak{p},\ \|\mathfrak{p}\| = 1 \land \langle \mathcal{M}_T \mathfrak{p}, \mathfrak{p} \rangle \leq \epsilon \wedge C(1+\|\overline{U}_T\|^3_{\hat{H}^6}) \|\mathfrak{p}\|_{\hat{H}^{n-2}} \geq g(\epsilon) \right) \notag\\
	&\leq \mathbb{P} \left( \exists \mathfrak{p},\ \|\mathfrak{p}\| = 1 \land \langle \mathcal{M}_T \mathfrak{p}, \mathfrak{p} \rangle \leq \epsilon \land \|\mathfrak{p}\|_{\hat{H}^{n-2}} \geq g(\epsilon)^{3/2} \right) + \mathbb{P} \left( C(1+\|\overline{U}_T\|^3_{\hat{H}^6}) > g(\epsilon)^{-1/2} \right) \notag\\
	&\leq \mathbb{P} \left( \exists \mathfrak{p},\ \|\mathfrak{p}\| = 1 \land \langle \mathcal{M}_T \mathfrak{p}, \mathfrak{p} \rangle \leq \epsilon \land \|\pi_{\hat{H}^n}\mathfrak{p}\|_{\hat{H}^{n-2}} \geq g(\epsilon)^{3/2} \right) + \mathbb{P} \left( C(1+\|\overline{U}_T\|^3_{\hat{H}^6}) > g(\epsilon)^{-1/2} \right) \notag\\
	&\leq C r(\epsilon) e^{\eta V^{n+2}(U_0)} + C g(\epsilon) e^{\eta V^{n+2}(U_0)} \notag\\
	&= C (r(\epsilon)+g(\epsilon)) e^{\eta V^{n+2}(U_0)}:=C f(\epsilon) e^{\eta V^{n+2}(U_0)},
\end{align}
where we choose \(g(\varepsilon)^{3/2}\) as in Proposition \ref{NS-prop5.21} and for the other term we use Chebyshev and Proposition \ref{priori-eatimates}. Thus combining \eqref{NS-5.23} and \eqref{NS-5.24}, we have
\[
\mathbb{P} \left( \exists \mathfrak{p},\ \|\mathfrak{p}\| = 1 \wedge \langle \mathcal{M}_T \mathfrak{p}, \mathfrak{p} \rangle \leq \epsilon \wedge \max_{|j|\le N, m\in \{0,1\}} |\langle \widehat{\Theta}_{U_T}(\mathbf{p}_T), \mathfrak{p} \rangle| \geq g(\epsilon) \right) \leq C f(\epsilon) e^{\eta V^{n+2}(U_0)}.
\]
Then, using Assumption \ref{assum5.5}, we compute
\begin{align*}
	&\mathbb{P}
	\bigg( 
	\exists \mathfrak{p} \in \hat{H}^n \times T_{\mathbf{p}_T}\Sigma,\ 
	\|\mathfrak{p}\| = 1,\ 
	\langle \mathcal{M}_T\mathfrak{p}, \mathfrak{p} \rangle \leq \epsilon,\ 
	\| \pi_{\scriptscriptstyle{T\Sigma}} \mathfrak{p} \| > {g(\epsilon)}^{3/2}
	\bigg) \\
	&\leq \mathbb{P} \left( \exists \mathfrak{p},\ \|\mathfrak{p}\| = 1 \land \langle \mathcal{M}_T\mathfrak{p}, \mathfrak{p} \rangle \leq \epsilon \land \max_{|j|\le N, m\in \{0,1\}} |\langle \widehat{\Theta}_{U_T}(\mathbf{p}_T), \mathfrak{p} \rangle| \geq g(\epsilon) \right) \\
	&\quad + \mathbb{P} \left( \exists \mathfrak{p},\ \|\mathfrak{p}\| = 1 \land \langle \mathcal{M}_T\mathfrak{p}, \mathfrak{p} \rangle \leq \epsilon \land \sup_{\substack{v \in T_{{\mathbf{p}}_T} \Sigma \\ \|v\|=1}} \left( \max_{|j|\le N, m\in \{0,1\}}  |\langle \widehat{\Theta}_{U_T}(\mathbf{p}_T), v \rangle| \right)^{-1} \geq g(\epsilon)^{-1/2} \right) \\
	&\leq C f(\epsilon) e^{\eta V^{n+2}(U_0)},
	\end{align*}
	thus we conclude, after redefining \(g(\epsilon)\) and \(f(\epsilon)\).
\end{proof}

The proof of Theorem \ref{bounds-Malliavin} now follows immediately from Propositions \ref{NS-prop5.21} and \ref{NS-prop5.22}.

\begin{proof}[Proof of Theorem \ref{bounds-Malliavin}]
We note that
\begin{align*}
	\mathbb{P} \left( \inf_{\|\mathfrak{p}\|=1,\,\mathfrak{p}\in \mathcal{S}_{\alpha,N}} \langle \mathcal{M}_T\mathfrak{p}, \mathfrak{p} \rangle < \epsilon \right) 
	&\leq \mathbb{P} \left( \exists \mathfrak{p},\ \|\mathfrak{p}\| = 1 \land \langle \mathcal{M}_T\mathfrak{p}, \mathfrak{p} \rangle \leq \epsilon \land \|\pi_{\scriptscriptstyle{T\Sigma}} \mathfrak{p}\| \geq \alpha/2 \right) \\
	&\quad +\mathbb{P} \left( \exists \mathfrak{p},\ \|\mathfrak{p}\| = 1 \land \langle \mathcal{M}_T\mathfrak{p}, \mathfrak{p} \rangle \leq \varepsilon \land \|\pi_{\hat{H}^n} \mathfrak{p}\|_{\hat{H}^{n-2}} \geq C^{-1} \alpha \right) \\
	&\leq \mathbb{P} \left( \exists \mathfrak{p},\ \|\mathfrak{p}\| = 1 \land \langle \mathcal{M}_T\mathfrak{p}, \mathfrak{p} \rangle \leq \epsilon \land \|\pi_{\scriptscriptstyle{T\Sigma}} \mathfrak{p}\| \geq g(\epsilon) \right) \\
	&\quad + \mathbb{P} \left( \exists \mathfrak{p},\ \|\mathfrak{p}\| = 1 \land \langle \mathcal{M}_T\mathfrak{p}, \mathfrak{p} \rangle \leq \epsilon \land \|\pi_{\hat{H}^n} \mathfrak{p}\|_{\hat{H}^{n-2}} \geq g(\epsilon) \right) \\
	&\leq Cf(\epsilon) e^{\eta V^{n+2} (U_0)} + C r(\epsilon) e^{\eta V^{n+2} (U_0)} \\
	&\leq C f(\epsilon) e^{\eta V^{n+2} (U_0)},
\end{align*}
where the second inequality holds for $\epsilon$ small enough and for the third inequality we use Proposition \ref{NS-prop5.21} and Proposition \ref{NS-prop5.22}. By redefining $f(\epsilon)$ we complete the proof.
\end{proof}
\section{Proof of Proposition \ref{prop:5.6} }\label{var-assum5.4-5}
In this section, we complete the proof of Proposition \ref{prop:5.6} by verifying that the Lagrangian, projective, tangent, and Jacobi processes satisfy Assumptions \ref{assum5.4} and \ref{assum5.5}. This proposition is frequently used throughout Section \ref{smoothing estimate} to establish the probabilistic spectral bound on a cone for the Malliavin covariance matrix. The detailed proofs will be provided in Subsections \ref{var-assum5.4} and \ref{var-assum5.5}.
\subsection{Linearization and second derivative of the extended system}\label{var-assum5.4}
In this subsection, we primarily focus on establishing estimates for the linearization and second derivatives of the extended system. This is mainly to verify Assumption \ref{assum5.4}. We note that it suffices to provide bounds for the linearization and second derivatives of the total extended system $(U_t,x_t,\tau_t,v_t,A_t)\in \hat{H}^4\times \mathbb{T}^2\times \mathbb{R}^2\times S^1\times \mathrm{SL}_2(\mathbb{R})$, since this ensures that the Lagrangian processes, projection processes, tangent processes, and matrix processes satisfy Assumption \ref{assum5.4}. It should be emphasized that this overall system is not hypoelliptic. Consequently, when applying hypoelliptic theory, we only consider a subset of these coordinates.

Indeed, it suffices to consider the simplified system $(U_t,x_t,\tau_t, A_t)$, without explicit consideration of the normalized tangent vector. This follows because estimates for the projective process are directly controlled by the tangent process (c.f. \cite{NS}).

Although our conclusion closely resembles that in \cite{NS}, the proof differs from \cite{NS} due to the presence of the buoyancy term $g\theta$. Specifically, in our proof, we must apply different weighting schemes to the velocity equation and the temperature equation, and also require appropriate use of inequality estimates during the analysis.

Fix initial data \((U_0, x_0, \tau_0, A_0)\) and time interval \(0 < s < t < T\). Given a unit random direction vector \((\psi_s^1,\psi_s^2, y_s, \zeta_s, B_s) \in H^4(\mathbb{T}^2) \times H^4(\mathbb{T}^2) \times \mathbb{R}^2 \times \mathbb{R}^2 \times T_{A_s}\mathrm{SL}_2(\mathbb{R})\) with \(\|(\psi_s^1,\psi_s^2, y_s, \zeta_s, B_s)\|=1\). Using the canonical embedding \(\mathrm{SL}_2(\mathbb{R}) \hookrightarrow \mathbb{R}^{2\times 2}\)-which induces the tangent space embedding \(T_{A_s}\mathrm{SL}_2(\mathbb{R}) \hookrightarrow \mathbb{R}^{2\times 2}\)-we construct the equations. Define the directional derivative \((\partial U_t, \partial x_t, \partial \tau_t, \partial A_t)\) of \((U_t, x_t, \tau_t, A_t)\) along \((\psi_s^1,\psi_s^2, y_s, \zeta_s, B_s)\), treating the system states as random functions of initial parameters \((U_s, x_s, \tau_s, A_s)\). And let $(\phi_t^1,\phi_t^2,z_t,\xi_t,C_t)$ denotes the second directional derivative of the same function in the same
direction. Here, we only need to control the diagonal part. Then, we have that
\begin{equation}\label{A1}
	\left(\psi_{t}^1,\psi_{t}^2,y_{t},\zeta_{t},B_{t}\right)=\mathcal{J}_{s,t}(\psi_{s}^1,\psi_{s}^2,y_{s},\zeta_{s},B_{s})
\end{equation}
and
\begin{equation}\label{A2}
	(\phi_{t}^1,\phi_{t}^2,z_{t},\xi_{t},C_{t})=\mathcal{J}^{2}_{s,t}\left((\phi_{s}^1,\phi_{s}^2,y_{s},\xi_{s},B_{s}),(\phi_{s}^1,\phi_{s}^2,y_{s},\xi_{s},B_{s})\right).
\end{equation}
Naturally, for linearization of the extended system, direct computation yields that
\begin{equation}\label{NS-A.7}
	\begin{cases}
		\dot{\psi}_{t}^1 = \nu_1\Delta\psi_{t}^1 - \nabla^{\perp}\Delta^{-1}\psi_{t}^1\cdot\nabla\omega_{t} - \nabla^{\perp}\Delta^{-1}\omega_{t}\cdot\nabla\psi_{t}^1+g\partial_1\psi_{t}^2 \\
		\dot{\psi}_{t}^2 = \nu_2\Delta\psi_{t}^2 - \nabla^{\perp}\Delta^{-1}\psi_{t}^2\cdot\nabla\theta_{t} - \nabla^{\perp}\Delta^{-1}\omega_{t}\cdot\nabla\psi_{t}^2 \\
		\dot{y}_{t} = y_{t}\cdot\nabla u_{t}(x_{t}) + \nabla^{\perp}\Delta^{-1}\psi_t^1(x_{t}) \\
		\dot{\zeta}_{t} = \zeta_{t}\cdot\nabla u_{t}(x_{t}) + \tau_{t}\otimes y_{t}:\nabla^{2}u_{t}(x_{t}) + \tau_{t}\cdot\nabla\nabla^{\perp}\Delta^{-1}\psi_{t}^1(x_{t}) \\
		\dot{B}_{t} = B_{t}\nabla u_{t}(x_{t}) + A_{t}y_{t}\cdot\nabla^{2}u_{t}(x_{t}) + A_{t}\nabla\nabla^{\perp}\Delta^{-1}\psi_{t}^1(x_{t}),
	\end{cases}
\end{equation}
where the initial data for these equations are given by \((\psi_{s}^1,\psi_{s}^2,y_{s},\zeta_{s},B_{s})\). For the second derivative, we also have
\begin{equation}\label{NS-A.8}
	\begin{cases}
		\dot{\phi}_{t}^1 = \nu_1\Delta\phi_{t}^1 - \nabla^{\perp}\Delta^{-1}\phi_{t}^1\cdot\nabla\omega_{t} - \nabla^{\perp}\Delta^{-1}\omega_{t}\cdot\nabla\phi_{t}^1 - 2\nabla^{\perp}\Delta^{-1}\psi_{t}^1\cdot\nabla\psi_{t}^1+g\partial_1\phi_{t}^2 \\
		\dot{\phi}_{t}^2 = \nu_2\Delta\phi_{t}^2 - \nabla^{\perp}\Delta^{-1}\phi_{t}^2\cdot\nabla\theta_{t} - \nabla^{\perp}\Delta^{-1}\omega_{t}\cdot\nabla\phi_{t}^2 - \nabla^{\perp}\Delta^{-1}\psi_{t}^2\cdot\nabla\psi_{t}^1-\nabla^{\perp}\Delta^{-1}\psi_{t}^1\cdot\nabla\psi_{t}^1\\
		\dot{z}_{t} = z_{t}\cdot\nabla u_{t}(x_{t}) + y_{t}\otimes y_{t}:\nabla^{2}u_{t}(x_{t}) + 2y_{t}\cdot\nabla\nabla^{\perp}\Delta^{-1}\psi_{t}^1(x_{t}) + \nabla^{\perp}\Delta^{-1}\phi_{t}^1(x_{t}) \\
		\dot{\xi}_{t} = \xi_{t}\cdot\nabla u_{t}(x_{t}) + \tau_{t}\nabla\nabla^{\perp}\Delta^{-1}\phi_{t}^1(x_{t}) + \tau_{t}\otimes z_{t}:\nabla^{2}u_{t}(x_{t}) + \tau_{t}\otimes y_{t}\otimes y_{t}:\nabla^{3}u_{t}(x_{t}) \\
		\qquad + 2\tau_{t}\otimes y_{t}:\nabla^{2}\nabla^{\perp}\Delta^{-1}\psi_{t}^1(x_{t}) + 2\zeta_{t}\otimes y_{t}:\nabla^{2}u_{t}(x_{t}) + 2p_{t}\cdot\nabla\nabla^{\perp}\Delta^{-1}\psi_{t}^1(x_{t}) \\
		\dot{C}_{t} = C_{t}\nabla u_{t}(x_{t}) + A_{t}\nabla\nabla^{\perp}\Delta^{-1}\phi_{t}^1(x_{t}) + A_{t}z_{t}\cdot\nabla^{2}u_{t}(x_{t}) + A_{t}y_{t}\otimes y_{t}:\nabla^{3}u_{t}(x_{t}) \\
		\qquad + 2B_{t}\nabla\nabla^{\perp}\Delta^{-1}\psi_{t}^1(x_{t}) + 2A_{t}y_{t}\cdot\nabla^{2}\nabla^{\perp}\Delta^{-1}\psi_{t}^1(x_{t}) + 2B_{t}y_{t}\cdot\nabla^{2}u_{t}(x_{t}),
	\end{cases}
\end{equation}
where these equations are given initial data 0.

Recalling \eqref{A1} and \eqref{A2}, to obtain bounds for the linearization and second derivative of the original nonlinear system, we need to estimate the bounds of $\psi_t$ and $\phi_t$ in $\hat{H}^4$-spaces. Before estimating the former, we first require bounds for $\psi_t$ in $\hat{H}$-spaces.
\begin{lemma}\label{NS-lemmaA.7}
	There exists \( C > 0 \) such that for \( t \geq s \),
	\[
	\|\psi_t\|_{\hat{H}} \leq \|\psi_s\|_{\hat{H}} \exp \left( C \int_s^t \|U_r\|_{\hat{H}^1}^{4/3}  dr \right).
	\]
\end{lemma}

\begin{proof}
	Employing the weighted estimation approach adopted for the temperature and vorticity equations in Appendix \ref{priori}, we can analogously derive
	\begin{align*}
		\frac{d}{dt} \frac{1}{2} \|\psi_t\|_{\hat{H}}^2 &\leq -\frac{1}{2} \|\psi_t\|_{D(A^{1/2})} + C_{\varkappa} \|\omega_t\|_{H^1} \|\psi_t^1\|_{W^{-1,4}} \|\psi_t^1\|_{L^4}+C\|\theta_t\|_{{H}^1}\|\psi_t^2\|_{W^{-1,4}}\|\psi_t^2\|_{L^4}\\
		&\leq -\frac{1}{2} \|\psi_t\|_{D(A^{1/2})}+ C \|U_t\|_{\hat{H}^1} \|\psi_t\|_{W^{-1,4}}\|\psi_t\|_{L^4}\\
		&\leq -\frac{1}{2} \|\psi_t\|_{D(A^{1/2})}+C\|U_t\|_{\hat{H}^1}\|\psi_t\|_{\hat{H}}^{3/2}\|\psi_t\|_{\hat{H}^1}^{1/2}\\
		&\leq C\|\|_{\hat{H}^1}^{4/3}\|\psi_t\|_{\hat{H}}^2,
	\end{align*}
	where $\|\cdot\|_{D(A^{1/2})}$ is defined as in \eqref{DA}. Then, applying the Gr\"{o}nwal inequality yields the conclusion.
\end{proof}

Now, we consider bounds for $\psi_t$ in $\hat{H}^4$-spaces. Taking derivatives of the equation for $(\psi_t^1,\psi_t^2)$, one has 
$$
\frac{d}{dt}\nabla^n\psi_t^1=\nu_1\Delta\nabla^n\psi_t^1-\sum_{j=0}^{n}\binom{n}{j}\left(\nabla^{n-j}\nabla^{\perp}\Delta^{-1}\omega_t\cdot\nabla\nabla^j\psi_t^1+\nabla^{\perp}\Delta^{-1}\nabla^j\psi_t^1\cdot\nabla\nabla^{n-j}\omega_t\right)+g\partial_1\nabla^n\psi_t^2
$$
and
$$
\frac{d}{dt}\nabla^n\psi_t^2=\nu_2\Delta\nabla^n\psi_t^2-\sum_{j=0}^{n}\binom{n}{j}\left(\nabla^{n-j}\nabla^{\perp}\Delta^{-1}\omega_t\cdot\nabla\nabla^j\psi_t^2+\nabla^{\perp}\Delta^{-1}\nabla^j\psi_t^2\cdot\nabla\nabla^{n-j}\theta_t\right).
$$
Correspondingly, to bound the second terms on the right-hand side of both equations above, we establish the following lemma.
\begin{lemma}\label{NS-lemmaA.10}
	For \( n \geq 1 \), there exists \( C(n) > 0 \) such that
	\begin{align}\label{NS-A.10}
		&\left| \int \nabla^n \psi_t^1 : \sum_{j=0}^{n}\binom{n}{j}\left(\nabla^{n-j}\nabla^{\perp}\Delta^{-1}\omega_t\cdot\nabla\nabla^j\psi_t^1+\nabla^{\perp}\Delta^{-1}\nabla^j\psi_t^1\cdot\nabla\nabla^{n-j}\omega_t\right)  dx \right|\notag\\
		&\leq \frac{\nu_1}{8} \| \psi_t^1 \|_{H^{n+1}}^2 + C \| \omega_t \|_{H^n}^{2n+2} \| \psi_t^1 \|_{L^2}^2
	\end{align}
	and 
	\begin{align}\label{A.11}
		&\left| \int \nabla^n  \psi_t^2 : \sum_{j=0}^{n}\binom{n}{j}\left(\nabla^{n-j}\nabla^{\perp}\Delta^{-1}\omega_t\cdot\nabla\nabla^j\psi_t^2+\nabla^{\perp}\Delta^{-1}\nabla^j\psi_t^2\cdot\nabla\nabla^{n-j}\theta_t\right)  dx \right|\notag\\
		&\leq \frac{\nu_2}{8} \| \psi_t^2 \|_{H^{n+1}}^2 + C \| \omega_t \|_{H^n}^{2n+2} \| \psi_t^2 \|_{L^2}^2.
	\end{align}
\end{lemma}
\begin{proof}
	The proof of equation \eqref{NS-A.10} can be found in Lemma A.8 of \cite{NS}. To establish \eqref{A.11}, however, we require certain modifications to the proof of \cite[Lemma A.8]{NS}. Firstly, we have that
	\begin{align*}
		&\left| \int \nabla^n  \psi_t^2 : \sum_{j=0}^{n}\binom{n}{j}\left(\nabla^{n-j}\nabla^{\perp}\Delta^{-1}\omega_t\cdot\nabla\nabla^j\psi_t^2+\nabla^{\perp}\Delta^{-1}\nabla^j\psi_t^2\cdot\nabla\nabla^{n-j}\theta_t\right)  dx \right|\\
		&\leq \left| \int \nabla^n \psi_t^2 : \nabla^\perp \Delta^{-1} \psi_t^2 \cdot \nabla \nabla^n \theta_t  dx \right| + \left| \int \nabla^n \psi_t^2 : \nabla^\perp \Delta^{-1} \omega_t \cdot \nabla \nabla^n \psi_t^2  dx \right| \\
		&\quad+ C \| \omega_t \|_{H^n} \| \psi_t^2 \|_{W^{n,4}}^2+C\|\omega_t\|_{W^{n,4}}\|\psi_t^2\|_{H^n}\|\theta_t\|_{W^{n,4}}\\
		&\leq \left| \int \nabla^{n+1} \psi_t^2 : \nabla^\perp \Delta^{-1} \psi_t^2 \nabla^n \omega_t  dx \right| + \left| \int \nabla^n \psi_t^2 : \nabla^\perp \Delta^{-1} \omega_t \cdot \nabla \nabla^n \psi_t^2  dx \right| \\
		&\quad+ C \| \omega_t \|_{H^n} \| \psi_t^2 \|_{W^{n,4}}^2+C\|\omega_t\|_{W^{n,4}}\|\psi_t^2\|_{H^n}\|\theta_t\|_{W^{n,4}}\\
		&\leq \| \psi_t^2 \|_{H^{n+1}} \| \nabla^\perp \Delta^{-1} \psi_t^2 \|_{L^\infty} \| \theta_t \|_{H^n} + \| \psi_t^2 \|_{H^n} \| \nabla^\perp \Delta^{-1} \omega_t \|_{L^\infty} \| \psi_t^2\|_{H^{n+1}} \\
		&\quad+  C \| \omega_t \|_{H^n} \| \psi_t^2 \|_{W^{n,4}}^2+C\|\omega_t\|_{W^{n,4}}\|\psi_t^2\|_{H^n}\|\theta_t\|_{W^{n,4}}\\
		&\leq \| \psi_t^2 \|_{H^{n+1}} \| \psi_t^2 \|_{H^1} \| \theta_t \|_{H^n} + \| \psi_t^2 \|_{H^n} \| \omega_t \|_{H^1} \| \psi_t^2 \|_{H^{n+1}} \\
		&\quad+ C \| \omega_t \|_{H^n} \| \psi_t^2 \|_{W^{n,4}}^2+C\|\omega_t\|_{W^{n,4}}\|\psi_t^2\|_{H^n}\|\theta_t\|_{W^{n,4}}.
	\end{align*}
	Then note
	\[
	\| \psi_t^2 \|_{W^{n,4}}^2 \leq \| \psi_t^2 \|_{H^{n+1/2}}^2 \leq \| \psi_t^2 \|_{H^n} \| \psi_t^2 \|_{H^{n+1}} \leq \| \psi_t^2 \|_{L^2}^{\frac{1}{n+1}} \| \psi_t^2 \|_{H^{n+1}}^{\frac{2n+1}{n+1}}.
	\]
	Thus putting it together
	\begin{align*}
		&\left| \int \nabla^n  \psi_t^2 : \sum_{j=0}^{n}\binom{n}{j}\left(\nabla^{n-j}\nabla^{\perp}\Delta^{-1}\omega_t\cdot\nabla\nabla^j\psi_t^2+\nabla^{\perp}\Delta^{-1}\nabla^j\psi_t^2\cdot\nabla\nabla^{n-j}\theta_t\right)  dx \right|\\
		&\leq \| \psi_t^2 \|_{H^{n+1}} \| \psi_t^2 \|_{H^1} \| \theta_t \|_{H^n} + \| \psi_t^2 \|_{H^n} \| \omega_t \|_{H^1} \| \psi_t^2 \|_{H^{n+1}} \\
		&\quad+ C \| U_t \|_{\hat{H}^n} \| \psi_t^2 \|_{L^2}^{\frac{1}{n+1}}\| \psi_t^2 \|_{H^{n+1}}^{\frac{2n+1}{n+1}}+C\|\psi_t^2\|_{H^n}\|U_t\|_{\hat{H}}^{\frac{1}{n+1}}\| U_t \|_{\hat{H}^{n+1}}^{\frac{2n+1}{n+1}}\\
		&\le C \|U_t\|_{\hat{H}}^{\frac{1}{n+1}}\| U_t \|_{\hat{H}^{n+1}}^{\frac{2n+1}{n+1}}\| \psi_t^2 \|_{L^2}^{\frac{1}{n+1}}\| \psi_t^2 \|_{H^{n+1}}^{\frac{2n+1}{n+1}}\\
		&\le \frac{\nu_2}{8} \| \psi_t^2 \|_{H^{n+1}}^2 + C \| \omega_t \|_{H^n}^{2n+2} \| \psi_t^2 \|_{L^2}^2.
	\end{align*}
\end{proof}

We immediately obtain the following lemma.
\begin{lemma}\label{NS-lemmaA.11}
	For all \( n \in \mathbb{N} \), there exists \( C(n) > 0 \) such that
	\begin{equation}\label{NS-A.11}
		\|\psi_t\|_{\hat{H}^n}^2 + \frac{\kappa}{2} \int_s^t \|\psi_r\|_{\hat{H}^{n+1}}  dr \leq \|\psi_s\|_{\hat{H}^n} + C(t - s) \sup_{s \le r \le t} \left( \|U_r\|_{\hat{H}^{n+1}}^{2n+2} \|\psi_r\|_{\hat{H}}^{2} \right)
	\end{equation}
\end{lemma}
Consequently, for the original nonlinear system, synthesizing with the a priori estimates from Proposition \ref{priori-eatimates}, we obtain
	$$
	\mathbb{E} \sup_{0 \leq s \leq t \leq T} \|{ \mathcal{J}_{s,t} }\|_{\hat{H}^{n} \to \hat{H}^n}^q \leq Ce^{\eta V^n(U_0)} 
	$$
holds for base process.

Then, as done in \cite[Appendix A.2]{NS}, we can get control on \( A_t \), \( \tau_t \) and the first derivatives $y_t,\zeta_t,B_t$.
\begin{lemma}There exists \( C > 0 \) such that
	\begin{align*}
		|A_t| &\leq \exp \left( C \int_s^t \|\omega_r\|_{H^{5/4}}  dr \right) |A_s|,\\
		|\tau_t| &\leq \exp \left( C \int_s^t \|\omega_r\|_{H^{5/4}}  dr \right) |\tau_s|,\\
		|y_t| &\leqslant C e^{t - s} \exp\left( C \int_s^t \|\omega_r\|_{H^{5/4}} \,\mathrm{d}r \right) \left( |y_s| + \sup_{s \leqslant r \leqslant t} \|\psi_r^1\|_{H^1} \right),\\
		|\zeta_t| &\leqslant C e^{t - s} \exp\left( C \int_s^t \|\omega_r\|_{H^{5/4}} \,\mathrm{d}r \right) \left( |\zeta_s| + \sup_{s \leqslant r \leqslant t} \left( \|\psi_r^1\|_{H^2}^{3/2} + \|\omega_r\|_{H^3}^3 + |y_r|^3 \right) \right),\\
		|B_t| &\leqslant C e^{t - s} \exp\left( C \int_s^t \|\omega_r\|_{H^{5/4}} \,\mathrm{d}r \right) \left( |B_s| + \sup_{s \leqslant r \leqslant t} \left( |A_r|^3 + |y_r|^3 + \|\omega_r\|_{H^3}^3 + \|\psi_r^1\|_{H^2}^{3/2} \right) \right).
	\end{align*}
\end{lemma}

Next, we consider estimating the second-order derivative of the extended system. Similarly, we first bound $\phi_t:=(\phi_t^1,\phi_t^2)$ in the space $\hat{H}$.
\begin{lemma}\label{A.10}
	There exists \( C > 0 \) such that
	\[
	\|\phi_t\|_{\hat{H}} \leq C(t - s) \exp\left(C \int_s^t \|U_r\|_{\hat{H}^1}^{4/3}  dr\right) \sup_{s \leq r \leq t} \|\psi_r\|_{\hat{H}^1}^2.
	\]
\end{lemma}
\begin{proof}
	Using the equation for \(\phi_t\), we compute as in Lemma \ref{NS-lemmaA.7},
	\begin{align*}
		\frac{d}{dt} \frac{1}{2} \|\phi_t\|_{\hat{H}}^2 &\leq-\frac{1}{2}\|\phi_t\|_{D(A^{1/2})}^2 +C_{\varkappa}\|\omega_t\|_{H^1}\|\phi_t^1\|_{W^{-1,4}}\|\phi_t^1\|_{L^4}+C_{\varkappa}\|\phi_t^1\|_{L^4}\|\psi_t^1\|_{W^{-1,4}}\|\psi_t^1\|_{H^1}\\
		&\quad+C\|\theta_t\|_{H^1}\|\phi_t^2\|_{W^{-1,4}}\|\phi_t^2\|_{L^4}+C\|\phi_t^2\|_{L^4}\|\psi_t^2\|_{W^{-1,4}}\|\psi_t^2\|_{H^1}+C\|\phi_t^2\|_{L^4}\|\psi_t^1\|_{W^{-1,4}}|\psi_t^2\|_{H^1}\\
		&\le -\frac{1}{2}\|\phi_t\|_{D(A^{1/2})}^2+C\Big[\|\omega_t\|_{H^1}\|\phi_t^1\|_{L^2}^{3/2}\|\phi_t^1\|_{H^1}^{1/2}+\|\psi_t^1\|_{H^1}\|\psi_t^1\|_{L^2}\|\phi_t^1\|_{H^1}^{1/2}\|\phi_t^1\|_{L^2}^{1/2}\\
		&\quad\quad+\|\theta_t\|_{H^1}\|\phi_t^2\|_{L^2}^{3/2}\|\phi_t^2\|_{H^1}^{1/2}+\|\psi_t^2\|_{H^1}\|\psi_t^2\|_{L^2}\|\phi_t^2\|_{H^1}^{1/2}\|\phi_t^2\|_{L^2}^{1/2}+\|\psi_t^2\|_{H^1}\|\psi_t^1\|_{L^2}\|\phi_t^2\|_{H^1}^{1/2}\|\phi_t^2\|_{L^2}^{1/2}\Big]\\
		&\le C\|U_t\|_{\hat{H}^1}^{4/3}\|\phi_t\|_{\hat{H}}^2+C\|\psi_t^1\|_{H^1}^2\|\psi_t^1\|_{L^2}^2+\tilde{C}_1\|\phi_t^1\|_{L^2}^2\\
		&\quad+C\|\psi_t^2\|_{H^1}^2\|\psi_t^2\|_{L^2}^2+\tilde{C}_2\|\phi_t^2\|_{L^2}^2+C\|\psi_t^2\|_{H^1}^2\|\psi_t^1\|_{L^2}^2+\tilde{C}_3\|\phi_t^2\|_{L^2}^2\\
		&\le C(\|U_t\|_{\hat{H}^1}^{4/3}+1)\|\phi_t\|_{\hat{H}}^2+C\|\psi_t\|_{\hat{H}^1}^4,
	\end{align*}
	and so conclude by Gr\"onwall inequality, using that \(\phi_s = 0\).
\end{proof}
\begin{lemma}\label{NS-lemmaA.13}
	For all $n > 1$, there exists $C(n) > 0$ such that
	\begin{align*}
		\|\phi_t\|_{\hat{H}^n}\le C(t-s)\left(\sup\limits_{s\le r\le t}\|U_r\|_{\hat{H}^n}^{4n+4}+\|\phi_r\|_{\hat{H}}^2+\|\psi_r\|_{\hat{H}^n}^2\right).
	\end{align*}
\end{lemma}
\begin{proof}
	The proof proceeds similarly to that of Lemmas \ref{NS-lemmaA.10} and \ref{NS-lemmaA.11}.
\end{proof}
\begin{lemma}\emph{\cite[Lemma A.14]{NS}}\label{NS-lemmaA.14}
	There exists \( C > 0 \) such that
	\begin{align*}
		|z_t| &\leq Ce^{t-s} \exp\left(C \int_{s}^{t} \|\omega_r\|_{H^{5/4}}  dr\right) \times \left( \sup_{s \leq r \leq t} |y_r|^4 + \|\omega_r\|_{H^{3}}^2 + \|\phi_r\|_{\hat{H}^{2}}^{4/3} + \|\psi_r\|_{\hat{H}^{1}} \right) \\
		|\xi_t| &\leq Ce^{t-s} \exp\left(C \int_{s}^{t} \|\omega_r\|_{H^{5/4}}  dr\right) \\
		&\quad\times \left( \sup_{s \leq r \leq t} \|\phi_r\|_{\hat{H}^{2}}^2 + |z_r|^3 + |y_r|^6 + \|\omega_r\|_{H^{4}}^3 + \|\psi_r\|_{\hat{H}^{3}}^3 + |\zeta_r|^3 + |\tau_r|^3 + 1 \right) \\
		|C_t| &\leq Ce^{t-s} \exp\left(C \int_{s}^{t} \|\omega_r\|_{H^{5/4}}  dr\right) \\
		&\quad\times \left( \sup_{s \leq r \leq t} \|\phi_r\|_{\hat{H}^{2}}^2 + |A_r|^3 + |B_r|^3 + |z_r|^3 + |y_r|^6 + \|\omega_r\|_{H^{4}}^3 + \|\psi_r\|_{\hat{H}^{3}}^2 + 1 \right)
	\end{align*}
\end{lemma}

Thus, we also can get control on  the second derivatives $z_t,\xi_t,C_t$ by Lemma \ref{A.10}-\ref{NS-lemmaA.14}. Finally, by utilizing these bounds and Proposition \ref{priori-eatimates}, and following the proof in \cite[Appendix A.2]{NS}, we can verify that the Lagrangian, projective, tangent, and Jacobian processes all satisfy Assumption \ref{assum5.4}.
\subsection{Span on manifolds for the projective and Jacobian processes}\label{var-assum5.5}
In this subsection, we primarily focus on verifying Assumption \ref{assum5.5} and integrating this with the results from the preceding subsection to complete the proof of Proposition \ref{prop:5.6}. The proof strategy presented here largely follows the approach in \cite[Appendix A.3]{NS}. However, whereas the setting in \cite{NS} only required consideration of the simple operators $\Theta_{e_k}$, a key difference in our manuscript is that the newly introduced vector fields on the manifold exhibit significantly higher complexity and depend on the stochastic vector field $U_T$. This necessitates a more refined treatment when verifying the spanning condition and selecting suitable vector fields, requiring careful balancing of two critical aspects: (1) choosing an appropriate $N>0$ by using properties of the trigonometric basis, and (2) selecting an appropriate $\eta\in (0,1)$ for estimates, taking into account the bounds on $U_T$.

The following lemma establishes the nondegeneracy of the two-point process $(\mathbf{u}_t,\theta_t,\mathbf{x}_t,\mathbf{y}_t)$ (an additional result), while the non-degeneracy for the Lagrangian process readily follows as a direct consequence.
\begin{lemma}[Nondegeneracy for two-point process]\label{NS-lemmaA.17}
	There exist constants \( C, N > 0 \) such that for all \( (x, y) \in \mathbb{T}^2 \times \mathbb{T}^2 \) and all \( (a, b) \in \mathbb{R}^2 \times \mathbb{R}^2 \),
	\[
	|a|+|b| \leq C |x-y|^{-1}\mathbb{E}\Big[\max_{|j|\le N,m\in\{0,1\}}\big|\hat{g}\big(\Theta_{Z_j^m(\overline U_T)+[U_T,Y_j^m(\overline U_T)]_x}^2(x,y),(a,b)\big)\big|\Big].
	\]
\end{lemma}
\begin{proof}
	Without loss of generality, we suppose that \( |x_1 - y_1| \geq |x_2 - y_2| \), \( |a| \leq |b| \), \( x = 0 \).  
	First, we present the specific form of the velocity field under consideration. Let $\pi_1:\hat{H}^n\to H^n(\mathbb{T}^2)$ denote the projection onto the vorticity coordinate and $\pi_2:\hat{H}^n\to H^n(\mathbb{T}^2)$ denote the projection onto the temperature coordinate, i.e., for any $U_t\in \hat{H}^n$, $\pi_1U_t=\omega_t$, $\pi_2U_t=\theta_t$. To facilitate subsequent presentation, we first provide the explicit mathematical form of the vector field components pertaining to $\overline{U}$ as in \eqref{overline-U}, along with their governing equations. Let $\overline{\omega}_t:=\pi_1\overline{U}_t$ and $\overline{\theta}_t:=\pi_2\overline{U}_t={\theta}_t-\tilde\sigma_\theta W$, where 
	$$\tilde\sigma_\theta W:=\sum_{k\in \mathcal{Z},l\in\{0,1\}}\alpha_k^l\tilde\sigma_k^l(x)W^{k,l}(t):=\sum_{k\in \mathcal{Z},l\in\{0,1\}}\alpha_k^l\pi_2\sigma_k^l(x)W^{k,l}(t).$$ 
	Then they satisfy
	\begin{equation} \label{Boussinesq-overline-w}
		\begin{aligned}
			\left\{
			\begin{array}{lr}
				d\overline{\omega}+(\mathbf{u}\cdot \nabla\omega-\nu_1\Delta\omega)dt=g\partial_x\overline{\theta} dt,\quad\overline{\omega}_0=\omega_0,\\
				d\overline{\theta}+(\mathbf{u}\cdot \nabla\theta-\nu_2\Delta\theta)dt=0,\quad\overline{\theta}_0=\theta_0.
			\end{array}\right.
		\end{aligned}
	\end{equation}
	Thus, we obtain that $\omega_t=\overline{\omega}_t+g\partial_x(\tilde\sigma_\theta W)$. Recalling the form of $Z_j^m(\overline U_T)$ in \eqref{JFA-5.11}, we have
	$$
	\pi_1Z_j^m(\overline U_T)=(-1)^m\big[(\nu_1+\nu_2)gj_1|j|^2\widetilde{\psi}_j^{m+1}+gj_1(K *\widetilde{\psi}_j^{m+1})\cdot \nabla\overline\omega_t+(K *\overline\omega_t)\cdot gj_1\nabla\widetilde{\psi}_j^{m+1}\big],
	$$
	where $\widetilde{\psi}_j^{m+1}:=\pi_1{\psi}_j^{m+1}$. The corresponding velocity field is then given by
	\begin{align}\label{Z}
		\nabla^{\perp}\Delta^{-1}\pi_1Z_j^m(\overline U_T)=-(\nu_1+\nu_2)gj_1j^{\perp}\widetilde{\psi}_j^{m}-gj_1\frac{j^{\perp}}{|j|^2}\nabla\overline{u}_T\cdot\widetilde{\psi}_j^{m}+(-1)^mgj_1\frac{j^{\perp}\cdot j^{\top}}{|j|^2}\widetilde{\psi}_j^{m+1}\cdot \overline{u}_T.
	\end{align}
	We now consider the vector field associated with $[U_T,Y_j^m(\overline U_T)]_x$. Recalling the form of $Y_j^m$ given in \eqref{JFA-5.9}, the definition of the Lie bracket yields the corresponding vector field representation for $[U_T,Y_j^m(\overline U_T)]_x$ as follows:
	\begin{align}\label{Y}
		(-1)^mgj_1\frac{j^{\perp}\cdot j^{\top}}{|j|^2}\widetilde{\psi}_j^{m+1}\cdot u_T-\nabla u_T\cdot gj_1\frac{j^{\perp}}{|j|^2}\widetilde{\psi}_j^{m},
	\end{align}
	where $u_T$ and $\nabla u_T$ also satisfy
	\begin{align*}
		u_T&=\overline{u}_T+g\sum_{k\in \mathcal{Z},l\in\{0,1\}}k_1\frac{k^{\perp}}{|k|^2}\alpha_k^l\tilde\sigma_k^l(x)W^{k,l},\\
		\nabla u_T&=\nabla\overline{u}_T+g\sum_{k\in \mathcal{Z},l\in\{0,1\}}(-1)^{l+1}k_1\frac{k^{\perp}\cdot k^{\top}}{|k|^2}\alpha_k^l\tilde\sigma_k^{l+1}(x)W^{k,l}.
	\end{align*}
	Combining \eqref{Z} and \eqref{Y}, we derive the vector field 
	$\mathcal{V}_j^m$ associated with $Z_j^m(\overline U_T)+[U_T,Y_j^m(\overline U_T)]_x$.
	\begin{align}\label{V}
		\mathcal{V}_j^m:=(-1)^mgj_1\frac{j^{\perp}\cdot j^{\top}}{|j|^2}\widetilde{\psi}_j^{m+1}\cdot (u_T+\overline{u}_T)-gj_1(\nabla u_T+\nabla\overline{u}_T)\frac{j^{\perp}}{|j|^2}\widetilde{\psi}_j^{m}-(\nu_1+\nu_2)gj_1j^{\perp}\widetilde{\psi}_j^{m}.
	\end{align}
	Let $\mathcal{V}_N\subseteq H^5(\mathbb{T}^2)$ be defined by
	$$
	\mathcal{V}_N:=\mathop{\rm span}\{\mathcal{V}_j^m:|j|\le N, m\in\{0,1\}\}
	$$ 
	for some $N\ge 2$ to be determined. Here, we assume that 
	$$
	\mathbb{E}\Big[\max_{|j|\le N,m\in\{0,1\}}\big|\hat{g}\big(\Theta_{Z_j^m(\overline U_T)}^2(0,y)+\Theta^2_{[U_T,Y_j^m(\overline U_T)]_x}(0,y),(a,b)\big)\big|\Big]\le 1
	$$
	such that for arbitrary $\tilde{u}\in \mathcal{V}_N$,
	it satisfies 
	\begin{equation}\label{NS-A.32}
		\mathbb{E}\Big[\big|\hat{g}\big(\Theta_{\tilde{u}}^2(0,y),(a,b)\big)\big|\Big]\le C\mathbb{E}\|\tilde{u}\|_{H^5(\mathbb{T}^2)}.
	\end{equation}
	To conclude, it suffices then to show 
	\[
	|b| < C |y|^{-1}.
	\]
	Then, we prove this for different cases of $y$. For $|y|\ge \frac{3\pi}{4}$, it is clear one can choose $N_1$ sufficiently large so that
	\begin{equation*}
		|b| \leq C\mathbb{E}\Big[\max_{|j|\le N_1,m\in\{0,1\}}\big|\hat{g}\big(\Theta_{\mathcal{V}_j^m}^2(x,y),(a,b)\big)\big|\Big]	\le C|y|^{-1}\mathbb{E}\Big[\max_{|j|\le N_1,m\in\{0,1\}}\big|\hat{g}\big(\Theta_{\mathcal{V}_j^m}^2(x,y),(a,b)\big)\big|\Big]\le C|y|^{-1}
	\end{equation*}
	by \eqref{JFA-A.5}. Thus, we now restrict consideration to the case $|y|\le \frac{3\pi}{4}$. Setting $j=(1,0)$, $m=1$, and by using \eqref{NS-A.32}, we obtain
	\begin{equation}\label{b2}
		g\mathbb{E}\Big|\nabla_2v_T^1(y)(b_1+b_2)\sin(y_1)+b_2\cos(y_1)v_T^1(y)+(\nu_1+\nu_2)b_2\sin(y_1)-a_2v_T^1(0)\Big|\le C\mathbb{E}\|\mathcal{V}_{(1,0)}^1\|_{H^5(\mathbb{T}^2)},
	\end{equation}
	where $v_T=(v_T^1,v_T^1):=u_T+\overline{u}_T$. Here we review the specific form of \eqref{JFA-A.5} as follows:
	$$
	\mathbb{E}\Big(\sup\limits_{t\in [T/2,T]}\|U_t\|^p_{\hat{H}^s}\Big)\le C\exp(\eta\|U_0\|_{\hat{H}}^2).
	$$
	Based on this, suitable $\eta_1(a,b)$ can be found such that for any $\eta\in(0,\eta_1(a,b))$, the following holds:
	$$
	|b_2\sin(y_1)|\le Ce^{\eta\|U_0\|_{\hat{H}}^2}(1+|y|)\le C.
	$$
	Thus $\frac{3\pi}{4}\ge |y_1|\ge \frac{|y|}{3}$ implies $|\sin(y_1)|\ge C^{-1}|y|$, which when combined with \eqref{b2}  immediately yields 
	$$
	|b_2| < C |y|^{-1}.
	$$
	If $|y_2|\ge \sqrt{1-(\frac{5}{9R_1})}|y|$, where $R_1$ is a constant to be determined later. We set $j=(R_1,R_1)$, $m=1$, then we also have that
	\begin{align*}
		&\frac{g}{2}\mathbb{E}\Big|(v_T^1+v_T^2)\cos(R_1y_1+R_1y_2)(b_1-b_2)-\frac{1}{R_1}\sin(R_1y_1+R_1y_2)\big((\nabla_2v_T^1-\nabla_1v_T^1)b_1+(\nabla_2v_T^2-\nabla_1v_T^2)b_2\big)\notag\\
		&+2R_1^2(\nu_1+\nu_2)\sin(R_1y_1+R_1y_2)(b_1-b_2)+(v_T^1+v_T^2)(a_1-a_2)\Big|\le C.
	\end{align*}
	Similarly, there exists a suitable $\eta_2(a,b,R_1)$ such that for any $\eta\in(0,\eta_2(a,b,R_1))$, the following holds:
	\begin{align}\label{b1}
		|\sin(R_1y_1+R_1y_2)(b_1-b_2)|\le C.
	\end{align}
	Then there must exist $R_1$ such that
	\begin{align*}
		|\cos(R_1y_1)\sin(R_1y_2)b_1|\le C.
	\end{align*}
	Noted that $|R_1y_1|\le \frac{5\pi}{12}$, then $|\cos(R_1y_1)|\ge \cos(\frac{5\pi}{12})>0$. Thus 
	\begin{align*}
		|\sin(R_1y_2)b_1|\le C.
	\end{align*}
	By analogous extension, we obtain that
	$$
	|b_1| < C |y|^{-1}.
	$$
	We now examine the final case where $|y_2|\le \sqrt{1-(\frac{5}{9R_1})}|y|$, in which we similarly establish \eqref{b1}. Then since we’ve already seen that $|b_2| < C |y|^{-1}$, we either have $|b_1| < C |y|^{-1}$ or 
	$$
	|b_1| < C|\sin(R_1y_1+R_1y_2)|^{-1}\le C |R_1y|^{-1}\le C |y|^{-1},
	$$
	where we use that on the set $\Big\{y:|y|\le \frac{3\pi}{4},|y_2|\le \sqrt{1-(\frac{5}{9R_1})}|y|\Big\}$, we have that $|\sin(R_1y_1+R_1y_2)|^{-1}\ge C^{-1}|R_1y|$ for some $C>0$. Finally, setting $N=\max\{N_1,R_1\}$ and $\eta^*:=\min\{\eta_1(a,b),\eta_2(a,b,R_1)\}$, for every 
	$\eta\in (0,\eta^*)$ the following holds:
	$$
	|b|\le C|y|^{-1}.
	$$
	This establishes the proof of the lemma.
\end{proof}
\begin{lemma}[Nondegeneracy for tangent process]\label{NS-lemmaA.18}
	There exist constants \( C > 0 \) such that for all \( (x, \tau) \in \mathbb{T}^2 \times \mathbb{R}^2 \) and all \( (y, \zeta) \in \mathbb{R}^2 \times \mathbb{R}^2 \), 
	\[
	|(y,\zeta)|\leq C (1+|\tau|^{-1})\mathbb{E}\Big[\max_{|j|\le 2,m\in\{0,1\}}\big|\hat{g}\big(\Theta_{Z_j^m(\overline U_T)+[U_T,Y_j^m(\overline U_T)]_x}^T(x,\tau),(y,\zeta)\big)\big|\Big].
	\]
\end{lemma}
\begin{proof} 
	By exploiting translation invariance, we may assume without loss of generality that $x=0$. Let $\mathcal{V}_2\subseteq H^5(\mathbb{T}^2)$ be defined by
	$$
	\mathcal{V}_2:=\mathop{\rm span}\{\mathcal{V}_j^m:|j|\le 2. m\in\{0,1\}\}
	$$ 
	Assume that 
	$$\mathbb{E}\Big[\max_{|j|\le 2,m\in\{0,1\}}\big|\hat{g}\big(\Theta_{Z_j^m(\overline U_T)+[U_T,Y_j^m(\overline U_T)]_x}^T(0,\tau),(y,\zeta)\big)\big|\Big]\le 1,$$
	such that for any ${\tilde{u}}\in \mathcal{V}_2$, we have that 
	\begin{equation}\label{NS-A.33}
		\mathbb{E}\Big[\big|\hat{g}\big(\Theta_{\tilde{u}}^T(0,\tau),(y,\zeta)\big)\big|\Big]\le C\mathbb{E}\|\tilde{u}\|_{H^5(\mathbb{T}^2)}.
	\end{equation}
	Fix $j=(1,-1),m=0$, and by using \eqref{NS-A.33}, we obtain
	\begin{align*}
		\mathbb{E}\Big[\big|\hat{g}\big(\Theta_{\mathcal{V}_{(1,-1)}^0}^T(0,\tau),(y,\zeta)\big)\big|\Big]
		&=\mathbb{E}\Big[\big|\mathcal{V}_{(1,-1)}^0(0)\cdot y+\tau\cdot \nabla\mathcal{V}_{(1,-1)}^0(0)\cdot\zeta\big|\Big]\\
		&=g\mathbb{E}\Big[\big|(\nu_1+\nu_2)(y_1+y_2)+(\tau_1+\tau_2)(\zeta_1-\zeta_2)(v_T^1(0)+v_T^2(0))\\
		&\quad\quad\quad
		-\zeta_1\nabla_1(\nabla_1+\nabla_2)v_T(0)\cdot \tau-\zeta_2\nabla_2(\nabla_1+\nabla_2)v_T(0)\cdot \tau\big|\Big]\\
		&\leq C\mathbb{E}\|\mathcal{V}_{(1,-1)}^0\|_{H^5(\mathbb{T}^2)}.
	\end{align*}
	Combining with \eqref{JFA-A.5}, we see that there exists $\eta_1(\tau,\zeta)>0$ such that for any $\eta\in (0,\eta_1(\tau,\zeta))$, it holds that
	$$
	|y|\le C\le C(|\tau|^{-1}+1).
	$$
	We now turn to term $\zeta$. 
	Setting $j=(1,1)$ and $m=1$, analogous computation yields 
	\begin{align*}
		\mathbb{E}\Big[\big|\hat{g}\big(\Theta_{\mathcal{V}_{(1,1)}^1}^T(0,\tau),(y,\zeta)\big)\big|\Big]
		&= -\frac{g}{2}\mathbb{E}\Big[\Big|-(v_T^1(0)+v_T^2(0))(y_1+y_2)+(\zeta_1+\zeta_2)\big((\nabla_2-\nabla_1)v_T(0)\cdot\tau\\
		&\quad\quad\quad+(\nu_1+\nu_2)(\tau_2-\tau_1)\big)+\zeta\cdot\nabla(v_T^1(0)+v_T^2(0))\Big|\Big]\\
		&\leq C.
	\end{align*}
	Therefore, similarly, there exists $\eta_2(\zeta)>0$ such that for any $\eta\in (0,\eta_2(\zeta))$, the following holds, 
	\begin{align}\label{zeta1}
		|\zeta_1+\zeta_2|\le C|\tau|^{-1}\le C(|\tau|^{-1}+1).
	\end{align}
	Then, setting $j=(1,-1)$ and $m=1$, subsequent computation yields 
	\begin{align*}
		\mathbb{E}\Big[\big|\hat{g}\big(\Theta_{\mathcal{V}_{(1,-1)}^1}^T(0,\tau),(y,\zeta)\big)\big|\Big]
		&= \frac{g}{2}\mathbb{E}\Big[\Big|(v_T^1(0)+v_T^2(0))(y_1+y_2)+(\zeta_1+\zeta_2)(\tau_1+\tau_2)(\nabla_1+\nabla_2)(v_T^1(0)+v_T^2(0))\\
		&\quad\quad\quad+(\zeta_1-\zeta_2)\Big((\tau_1+\tau_2)(\nu_1+\nu_2)+(\nabla_1+\nabla_2)(\tau\cdot v_T(0))\Big)\Big|\Big]\\
		&\leq C.
	\end{align*}
	Thus, we have 
	\begin{align}\label{zeta2}
		|\zeta_1-\zeta_2|\le C|\tau|^{-1}\le C(|\tau|^{-1}+1).
	\end{align}
	Synthesizing \eqref{zeta1} and \eqref{zeta2}, we derive that 
	$$
	|\zeta_1|\le C(|\tau|^{-1}+1) \,\,\,\, \text{and} \,\,\,\, |\zeta_2|\le C(|\tau|^{-1}+1).
	$$
	Therefore, incorporating the estimates from all preceding cases, there exists $\eta^*:=\min\{\eta_1,\eta_2\}$ such that for any $\eta\in (0,\eta^*)$, it holds that 
	$$
	|(y,\zeta)|\le C(|\tau|^{-1}+1).
	$$
	Combining these bounds, we have now completed the proof.
\end{proof}

We now address the nondegeneracy of the Jacobi process. 
We now proceed to discuss the nondegeneracy of the Jacobian process. Based on the following observation, it suffices to consider the case where $A={\mathrm{I}}$ for the Jacobian process:
\begin{lemma}\label{NS-lemmaA.20}
	Suppose that for some \( C, N > 0 \) with \( N \geq 1 \), we have for all \( x \in \mathbb{T}^2 \) and \((y, B) \in \mathbb{R}^2 \times T_{\mathrm{I}}\mathrm{SL}(2, \mathbb{R}) \) that  
	\[
	|(y, B)| \leq C \mathbb{E}\Big[\max_{|j| \leq N,m\in\{0,1\}} \Big|\hat{g}\big( \Theta^J_{Z_j^m(\overline U_T)+[U_T,Y_j^m(\overline U_T)]_x}(x, I), (y, B) \big)\Big|\Big].
	\]  
	Then for all \((x, A) \in \mathbb{T}^2 \times \mathrm{SL}(2, \mathbb{R}) \) and \((y, B) \in \mathbb{R}^2 \times T_{\mathrm{A}}\mathrm{SL}(2, \mathbb{R}) \) we have that  
	\[
	|(y, B)| \leq C |A| \mathbb{E}\Big[\max_{|j| \leq N,m\in\{0,1\}} \Big|\hat{g}\big( \Theta^J_{Z_j^m(\overline U_T)+[U_T,Y_j^m(\overline U_T)]_x}(x, A), (y, B) \big)\Big|\Big].
	\] 
\end{lemma}
\begin{proof}
The above lemma may be proved analogously following the approach in \cite[Appendix A.3.3]{NS}.
\end{proof}
We now complete the proof of nondegeneracy for the Jacobian process by combining the aforementioned lemma with a judicious selection of vector fields.
\begin{lemma}[Nondegeneracy for Jacobian process]\label{NS-lemmaA.21}
	For all \((x, A) \in \mathbb{T}^2 \times \mathrm{SL}_2(\mathbb{R})\) and \((y, B) \in \mathbb{R}^2 \times T_{\mathrm{A}}\mathrm{SL}_2(\mathbb{R})\) , we have
	\[
	|(y,B)| \leq C |A|\mathbb{E}\Big[\max_{|j| \leq N,m\in\{0,1\}} \Big|\hat{g}\big( \Theta^J_{Z_j^m(\overline U_T)+[U_T,Y_j^m(\overline U_T)]_x}(x, A), (y, B) \big)\Big|\Big].
	\]
\end{lemma}
\begin{proof}
	It suffices to establish the case $A={\mathrm{I}}$. Without loss of generality, we similarly assume $x=0$. 
	To establish the desired conclusion, it suffices to prove that 
	$$
	|y|+|B|\le C.
	$$
	Employing methods closely parallel to those in Lemmas \ref{NS-lemmaA.17} and \ref{NS-lemmaA.18}, we deduce the existence of $N$ and $\eta_1$ such that for every $\eta\in(0,\eta_1)$, the relation $|y|\le C$ holds. We therefore restrict our attention to estimating the $B$. Observing that $B\in T_I\mathrm{SL}_2(\mathbb{R})$, we thus have $\mathrm{Tr}(B)=0$, and $B$ takes the following form:
	\[
	B = \begin{pmatrix} a & b \\ c & -a \end{pmatrix}.
	\]
	Let $\mathcal{V}_N\subseteq H^5(\mathbb{T}^2)$ be defined by
	$$
	\mathcal{V}_N:=\mathop{\rm span}\{\mathcal{V}_j^m:|j|\le N. m\in\{0,1\}\},
	$$ 
	where $N$ is the parameter selected during the aforementioned estimation of $y$. Suppose that 
	$$
	\mathbb{E}\Big[\max_{|j| \leq N,m\in\{0,1\}} \Big|\hat{g}\big( \Theta^J_{Z_j^m(\overline U_T)+[U_T,Y_j^m(\overline U_T)]_x}(x, A), (y, B) \big)\Big|\Big]\le 1,
	$$
	such that for any $\tilde{u}\in \mathcal{V}_N$, we have that 
	\begin{equation}\label{B}
		\mathbb{E}\Big[\max_{|j| \leq N,m\in\{0,1\}} \Big|\hat{g}\big( \Theta^J_{\tilde{u}}(x, A), (y, B) \big)\Big|\Big]\le C\mathbb{E}\|\tilde{u}\|_{H^5(\mathbb{T}^2)}.
	\end{equation}
	Proceeding, we estimate various bounds for matrix $B$ by selecting appropriate vector fields. Setting $j=(1,0)$ and $m=1$, we compute that
	\begin{align*}
		\mathbb{E}\Big[\big|\hat{g}\big(\Theta_{\mathcal{V}_{(1,0)}^0}^J(0,I),(y,B)\big)\big|\Big]
		&=\mathbb{E}\Big[\big|\mathcal{V}_{(1,0)}^0(0)\cdot y+\mathrm{Tr}( (\nabla\mathcal{V}_{(1,0)}^0(0))^\top\cdot B)\big|\Big]\\
		&=g\mathbb{E}\Big[\big|y_2v_T^1(0)+c\nabla_2v_T^2(0)+(\nu_1+\nu_2)c-a\nabla_{22}v_T^2(0)
		\big|\Big]\\
		&\leq C\mathbb{E}\|\mathcal{V}_{(1,0)}^0\|_{H^5(\mathbb{T}^2)}.
	\end{align*}
	Similarly, application of \eqref{JFA-A.5} yields the existence of $\eta_2(<\eta_1)$ such that for any $\eta\in(0,\eta_2)$,
	$$
	|c|\le C.
	$$
	Next, we assign $j=(1,0)$, $m=0$ and $j=(1,1)$, $m=1$ respectively, and we obtain that 
	\begin{align}\label{100}
		\mathbb{E}\Big[\big|\hat{g}\big(\Theta_{\mathcal{V}_{(1,0)}^1}^J(0,I),(y,B)\big)\big|\Big]
		&= g\mathbb{E}\Big[\Big|-\nabla_2v_T^1(0)y_1-(\nabla_2v_T^2(0)+\nu_1+\nu_2)y_2\notag\\
		&\quad\quad\quad+a(\nabla_{22}v_T^2(0)-\nabla_{21}v_T^1(0))-b\nabla_{22}v_T^1(0)+c(v_T^1(0)-\nabla_{21}v_T^2(0))\Big|\Big]\notag\\
		&\leq C,
	\end{align}
	and 
	\begin{align}\label{111}
		\mathbb{E}\Big[\big|\hat{g}\big(\Theta_{\mathcal{V}_{(1,1)}^1}^J(0,I),(y,B)\big)\big|\Big]
		&= g\mathbb{E}\Big[\Big|(v_T^1(0)+v_T^2(0))(y_1-y_2)+(\nu_1+\nu_2)(c-2a-b)\notag\\
		&\quad\quad\quad+(c-b)(\nabla_1v_T^1(0)+\nabla_2v_T^2(0))\Big|\Big]\notag\\
		&\leq C.
	\end{align}
	Combining this with $|c|\le C$, we immediately see that there exists $\eta^*(<\eta_2)$ such that for any $\eta\in (0,\eta^*)$, both $|a-b|\le C$ and $|c-2a-b|\le C$ hold. Thus for any $\eta\in (0,\eta^*)$, we have that
	$$
	|c|+|a-b|+|c-2a-b|\le C.
	$$
	Using the above bound, we immediately obtain $|B|\le C$. Thereby completing the proof of the lemma.
\end{proof}
\section{Approximate controllability of nonlinear Lagrangian flow}\label{Approximate controllability}
In this section, we primarily establish the approximate controllability condition $(C)$, which facilitates the exclusion of two types of almost surely continuous invariant structures. This constitutes part of the proof of Theorem \ref{posivie-lambda}. Furthermore, by combining the approximate controllability condition with the dissipativity of the Boussinesq equation, we complete the proof of weak irreducibility,namely Proposition \ref{irreducibility}. 

The proof of part (b) of condition $(C)$ is divided into two stages. First, we achieve controllability by identifying a suitable smooth control that steers the solution of the controlled system from the initial state to the target state within time $T(>0)$. Subsequently, we characterize the closeness between the solution of the controlled system \eqref{Con-2}-\eqref{Con-ini1} and that of the original system \eqref{Con-1}-\eqref{Con-ini} at time $T$, thereby establishing the approximation property. Combining these two components proves part (b) of condition $(C)$. Similar arguments are then provided for part (a).

As noted in the Introduction, for the Boussinesq system under study, the stochastic forcing acts directly only on the temperature equation and subsequently influences the velocity equation via the buoyancy term $\mathbf{g}\theta$. Consequently, the smooth controls we construct act directly solely on the temperature equation while indirectly achieving controllability of the velocity equation.
This differs fundamentally from the $Qh(t)$-type control constructed in \cite{NS}, as our smooth control necessarily exhibits spatial dependence (i.e., takes the form $\sum_{\substack{k \in \hat{\mathcal{Z}}, l \in \{0,1\}}}\sigma_k^l(x)h_k^l(x,t)$). This $x$-dependence will introduce additional complications when establishing the approximation property (c.f. Proposition \ref{V2-thm5.2}).

Recall that our motivation for constructing such smooth controls stems from the analysis of the nonlinear term in the Navier-Stokes equations presented in \cite{JEMS}. The underlying idea is to construct them based on shear flows and cellular flows, which nullify the nonlinear term. Despite the Boussinesq equations being coupled with a temperature equation compared to the Navier-Stokes equations, they retain the same form of nonlinear term. Consequently, we can likewise employ similar shear flows and cellular flows to construct smooth controls. Furthermore, it is evident that for a class of dissipative equations possessing an Euler-type nonlinearity, shear flows and cellular flows can be used to facilitate the achievement of approximate controllability.

To be precise, the aim of this section is to verify the approximate controllability of the following extended nonlinear system:
\begin{equation}\label{Con-1}
	\begin{aligned}
		d\mathbf{u}+(\nabla^\top\Delta^{-1}\mathbf{u}\cdot \nabla\mathbf{u}-\nu_1\Delta\mathbf{u})dt&=g\nabla^\top\Delta^{-1}\partial_x\theta dt,\\
		d\theta+(\mathbf{u}\cdot \nabla\theta-\nu_2\Delta\theta)dt&=\nabla^\top\Delta^{-1}\sigma_\theta dW,\\
		\dot{x}(t) &= \mathbf{u}(x(t)), \\
		\dot{v}(t) &= v\cdot \nabla \mathbf{u}(x(t))\cdot v^\top v^\top \\
		\dot{A}(t) &= A(t)\nabla \mathbf{u}(x(t)),
	\end{aligned}
\end{equation}
supplemented with the initial condition
\begin{align}\label{Con-ini}
	\mathbf{u}(0)=\mathbf{u}_0,\, \,\theta(0)=\theta_0,\,\, A(0)=\mathrm{Id}_{\mathbb{R}^2}.
\end{align}
Here, we give the corresponding control system for the equations \eqref{Con-1}-\eqref{Con-ini} as follows:
\begin{equation}\label{Con-2}
	\begin{aligned}
		d\mathbf{u}^{h}+(\nabla^\top\Delta^{-1}\mathbf{u}^{h}\cdot\nabla\mathbf{u}^{h}-\nu_1\Delta\mathbf{u}^{h})dt&=g\nabla^\top\Delta^{-1}\partial_x\theta^{h} dt,\\
		d\theta^{h}+(\mathbf{u}^{h}\cdot \nabla\theta^{h}-\nu_2\Delta\theta^{h})dt&=\nabla^\top\Delta^{-1}\sigma_\theta hdt,\\
		\dot{x}^{h}(t) &= \mathbf{u}^{h}(x^{h}(t)), \\
		\dot{v}^{h}(t) &= v^{h}\cdot \nabla \mathbf{u}^{h}(x^{h}(t))\cdot {(v^{h})}^\top {(v^{h})}^\top \\
		\dot{A}^{h}(t) &= A^{h}(t)\nabla \mathbf{u}^{h}(x^{h}(t)),
	\end{aligned}
\end{equation}
supplemented with the initial condition
\begin{align}\label{Con-ini1}
	\mathbf{u}^{h}(0)=\mathbf{u}_0,\, \,\theta^{h}(0)=\theta_0,\,\, A^{h}(0)=\mathrm{Id}_{\mathbb{R}^2},
\end{align}
where $\sigma_\theta h:=\sum_{\substack{k \in \hat{\mathcal{Z}}, l \in \{0,1\}}}\sigma_k^l(x)h_k^l(x,t)\in C_t^{\infty}(\mathbb{R}_+;L^2)$ is a smooth control, $\hat{\mathcal{Z}}$ is symmetric and $(1,0),\,(0,1),\,(0,2),$ $(1,1),\,(-1,1)\in \hat{\mathcal{Z}}$, as will be verified in Proposition \ref{V2-thm5.1}.
\begin{proof}[Verification of (b) for approximate controllability]
	The proof is divided into three steps.\\
	\textbf{Step 1 (decay of $(\mathbf{u},\theta)$).} The dissipative nature of the Boussinesq equation implies that when $\sigma_\theta h=0$, the solution $U_t^0$ decays exponentially to zero in $\hat{H}^4$. Indeed, observe that $\frac{d}{dt}\|U_t^0\|_{\hat{H}}\le -\frac{\kappa}{2}\|U_t^0\|_{D(A^{1/2})}$. Then, interpolating between this inequality and the growth bound established by repeatedly applying Proposition \ref{priori-eatimates}, one obtains the desired conclusion that it decays to zero. \\
	\textbf{Step 2 (controllability).} Moving $x_0$ to $x'$ and $v_0$ to $v'$ by Proposition \ref{V2-thm5.1}. Thus, there exists a smooth control \(\sigma_\theta h(x,t)\) such that $\big(\mathbf{u}^h_{1},\theta^h_{1}, x^h_{1},v^h_{1}\big)=(\mathbf{0},0,x',v')$.\\
	\textbf{Step 3 (approximate).} By virtue of the stability, we can derive that $\mathbf{u}_1$ and $\mathbf{u}_1^h$, $\theta_1$ and $\theta_1^h$ are approximated in the $L^\infty([0,1];H^5(\mathbb{T}^2))$-norm, $x_1$ and $x_1^h$, $v_1$ and $v_1^h$ are approximated in the $\mathbb{R}^2$-norm. Thereby verifying the approximate controllability condition (b) by Proposition \ref{V2-thm5.2}.
\end{proof}

Similarly, we can verify that the approximate controllability condition (a) holds by Proposition \ref{V2-thm5.3}.
The precise formulations and proofs of Propositions \ref{V2-thm5.1}–\ref{V2-thm5.3} are presented next.


\begin{proposition}[Controllability  of nonlinear Lagrangian flow]\label{V2-thm5.1}
For any initial data \((\mathbf{0},0, x_0,v_0,\mathrm{Id}_{\mathbb{R}^2})\\ \in H^5(\mathbb{T}^2)\times H^5(\mathbb{T}^2)\times\mathbb{T}^2 \times S^1 \times \mathrm{SL}_2(\mathbb{R}) \), and any target \((x', v') \in \mathbb{T}^2 \times S^1\), there exists and a smooth control \(\sigma_\theta h^x(x,t)\in C^{\infty}([0,1];L^2)\) such that  
	\begin{equation*}
		\big(\mathbf{u}^h_{1},\theta^h_{1}, x^h_{1},v^h_{1}\big)=(\mathbf{0},0,x',v').
	\end{equation*}
Furthermore, $\sigma_\theta h^x(x,t)$ can be chosen to depend smoothly on $x$, $x'$, $v$, $v'$ and supported only in frequencies $|k|_\infty \le 2$.
\end{proposition}

\begin{proof}
	We begin by considering the displacement of particle positions. Observing that the particle position $x\in \mathbb{T}^2$ is two-dimensional, we address its horizontal and vertical components separately. Now, we begin by proving that the horizontal component of $x$ can be moved to the target state. Let $x_0=(a_0,b_0),\, x'=(a_1,b_1)$. For $t\in(0,1/4)$, suppose the velocity field is given by the shear flow 
	\[
	\mathbf{u}_{t}^h(y)=f_{1/4}(t)\begin{pmatrix}
		\cos(y_{2}-b_{0}) \\ 
		0
	\end{pmatrix},
	\]
	where $f_{1/4}\in C_c^\infty(0,1/4)$ and $\int_0^{1/4} f_{1/4}(t)dt=a_1-a_0$. Then $\mathbf{u}_{t}^h(y)$ satisfies the following properties:\\
	(a) $\mathbf{u}_0=\mathbf{u}_{1/4}=\mathbf{0},\, \theta_0=\theta_{1/4}=0$,\\
	(b) the solution $x_t^h$ of 
	\[
	\begin{cases} 
		\dot{x}^h(t) = \mathbf{u}^h(t, x^h(t)), \\ 
		x^h(0) = x_0 
	\end{cases}
	\]
	satisfies $x^h_{1/4}=(a_1,b_0)$,\\
	(c) $\theta^h(t, x)=\frac{1}{g}(f'_{1/4}(t)+\nu_1f_{1/4}(t))x_1\sin(x_2-b_0)$,\\
	(d) $\big(\mathbf{u}^h(t, x),\theta^h(t, x)\big)$ is a solution of the Boussinesq system \eqref{Boussinesq} with control $\sigma_\theta h^1(t,x)$ given explicitly by
	$$
	\sigma_\theta h^1(t,x)=\frac{1}{g}\big[(f''_{1/4}(t)+\nu_1f'_{1/4}(t))x_1\sin(x_2-b_0)+(f'_{1/4}(t)+\nu_1f_{1/4}(t))(\frac{1}{2}\sin(2x_2-2b_0)+\nu_2x_2\sin(x_2-b_0))\big].
	$$
It follows immediately that $\sigma_\theta h^1(t,x)\in C^{\infty}([0,1/4];L^2)$ and is supported only on frequencies $|k|_\infty \le 2$.
	
	To verify that control $\sigma_\theta h^1$ does not affect the matrix flow $A_t$, we note that the vector field $\mathbf{u}^h(t, x)$ defined above satisfies
	\[
	\nabla\mathbf{u}^h(t,x) = f_{1/4}(t) \begin{pmatrix} 
		0 & -\sin(x_2 - b_0) \\ 
		0 & 0 
	\end{pmatrix},
	\]
	then we obtain that
	$$
	\dot{A}^h_t=\nabla\mathbf{u}^h(t,x^h(t))=0,
	$$
	which yields that $\dot{A}^h_t=\mathrm{Id}_{\mathbb{R}^2},\, t\in [0,1/4]$.
	
	Next, while holding fixed both the horizontal component of particle positions and the matrix flow, we displace the vertical component via shear flow. Similarly, for $t\in(1/4,1/2)$, suppose the velocity field is given by the shear flow 
	\[
	\mathbf{u}_{t}^h(y)=f_{1/2}(t)\begin{pmatrix}
		0 \\ 
		\cos(y_{1}-a_{1})
	\end{pmatrix},
	\]
	where $f_{1/2}\in C_c^\infty(1/4,1/2)$ and $\int_{1/4}^{1/2} f_{1/2}(t)dt=b_1-b_0$. Then $\mathbf{u}_{t}^h(y)$ satisfies the following properties:\\
	(a) $\mathbf{u}_{1/4}=\mathbf{u}_{1/2}=\mathbf{0},\, \theta_{1/4}=\theta_{1/2}=0$,\\
	(b) the solution $x^h_t$ of 
	\[
	\begin{cases} 
		\dot{x}^h(t) = \mathbf{u}^h(t, x^h(t)), \\ 
		x^h(1/4) = (a_1,b_0) 
	\end{cases}
	\]
	satisfies $x_{1/2}=(a_1,b_1)$,\\
	(c) $\theta^h(t, x)=\frac{1}{g}(f'_{1/2}(t)+\nu_1f_{1/2}(t))\cos(x_1-a_1)$,\\
	(d) $\big(\mathbf{u}^h(t, x),\theta^h(t, x)\big)$ is a solution of the Boussinesq system \eqref{Boussinesq} with control $\sigma_\theta h^2(t,x)$ given explicitly by
	$$
	\sigma_\theta h^2(t,x)=\frac{1}{g}\big[(f''_{1/2}(t)+\nu_1f'_{1/2}(t))\cos(x_1-a_1)-(f'_{1/2}(t)+\nu_1f_{1/2}(t))\nu_2\cos(x_1-a_1)\big].
	$$
	It follows immediately that $\sigma_\theta h^2(t,x)\in C^{\infty}([1/4,1/2];L^2)$ and is supported only on the frequencies $|k|_\infty \le 1$. By the same approach, we can verify that control $\sigma_\theta h^2$ does not affect the matrix flow $A_t,\, t\in [1/4,1/2]$.
	
	Subsequently, we extend the aforementioned control scheme to address the projective vector field while preserving particle positions invariant. The methodology remains analogous, with the key distinction that we now employ a cellular flow to construct the velocity field. Assuming we have previously displaced $v_0$ to a new position, let the updated value be denoted as $v_{1/2}^h$. For $t\in(1/2,1)$, suppose the velocity field is given by the cellular flow 
	\[
	\mathbf{u}_{t}^h(y)=f_{1}(t)\begin{pmatrix}
		-\sin(y_2-b_1) \\ 
		\sin(y_{1}-a_{1})
	\end{pmatrix},
	\]
	where $f_{1}\in C_c^\infty(1/2,1)$ and $\int_{1/2}^{1} f_{1}(t)dt=\angle(v',v_{1/2}^h)$. Then $\mathbf{u}_{t}^h(y)$ satisfies the following properties:\\
	(a) $\mathbf{u}_{1/2}=\mathbf{u}_{1}=\mathbf{0},\, \theta_{1/2}=\theta_{1}=0$,\\
	(b) the solution $v^h_t$ of 
	\[
	\begin{cases} 
		\dot{v}^h(t) = v^{h}\cdot \nabla \mathbf{u}^{h}(x^{h}(t))\cdot {(v^{h})}^\top {(v^{h})}^\top, \\ 
		v^h(1/2) = v_{1/2}^h 
	\end{cases}
	\]
	satisfies $v^h_{1}=v'$,\\
	(c) $\theta^h(t, x)=\frac{1}{g}(f'_{1}(t)+\nu_1f_{1}(t))\big(-x_1\cos(x_2-b_1)+\sin(x_1-a_1)\big)$,\\
	(d) $\big(\mathbf{u}^h(t, x),\theta^h(t, x)\big)$ is a solution of the Boussinesq system \eqref{Boussinesq} with control $\sigma_\theta h^v(t,x)$ given explicitly by
	\begin{align*}
		\sigma_\theta h^v(t,x)&=\frac{1}{g}\big[(f''_{1}(t)+\nu_1f'_{1}(t))(-x_1\cos(x_2-b_1)+\sin(x_1-a_1))\\
		&\quad+f_{1}(t)(f'_{1}(t)+\nu_1f_{1}(t))\big(\cos(x_1-b_1)-\cos(x_1-a_1)+x_1\sin(x_1-a_1)\big)\sin(x_2-b_1)\\
		&\quad+\nu_2(f'_{1}(t)+\nu_1f_{1}(t))(-\sin(x_1-a_1)+x_1\cos(x_2-b_1))\big].
	\end{align*}
	It follows immediately that $\sigma_\theta h^v(t,x)\in C^{\infty}([1/2,1];L^2)$ and is supported only on frequencies $|k|_\infty \le 2$. It can be ascertained at this stage that the matrix flow remains unchanged. This completes the proof of the proposition.
\end{proof}

\begin{proposition}[Stability]\label{V2-thm5.2}
	For all \( t > 0 \) and \( \varepsilon > 0 \), there exists \( \varepsilon' > 0 \) such that for all \((x, v), (x', v') \in \mathbb{T}^2 \times S^{1}\) and all \( \mathbf{u} \in B_{\varepsilon'}( \mathbf{0}) \),
	\[
	\mathbb{P}\bigl((\mathbf{u}_t, \theta_t, x_t) \in B_{\varepsilon}(\mathbf{0}) \times B_{\varepsilon}(0) \times B_{\varepsilon}(x')\bigm| (\mathbf{u}(0),\theta(0),x(0))=(\mathbf{u}_0, \theta_0, x_0)\bigr) > 0,
	\]
	\[
	\mathbb{P}\bigl((\mathbf{u}_t, \theta_t, x_t, v_t) \in B_{\varepsilon}(\mathbf{0}) \times B_{\varepsilon}(0) \times B_{\varepsilon}(x') \times B_{\varepsilon}(v')\bigm| (\mathbf{u}(0),\theta(0),x(0),v(0))=(\mathbf{u}_0, \theta_0, x_0, v_0)\bigr) > 0.
	\]
\end{proposition}
\begin{proof}
	Note that Proposition \ref{V2-thm5.1} has already established that the solution of the control equation can steer the initial state $(x_0,v_0)$ to the target state $(x',v')$. Therefore, we need only consider the stability of the system with respect to external forces. Here, $\sigma_\theta h^x(x_1,t)$ denotes the smooth control found in Proposition \ref{V2-thm5.1} corresponding to the desired endpoints $x_0,\, x'$, given explicitly by:
	\begin{align*}
		\sigma_\theta h^x(x_1,t)&=[(f_{1/4}''(t)+\nu_1f_{1/4}'(t))\cos b_0+\nu_2(f_{1/4}'(t)+\nu_1f_{1/4}(t))\cos b_0]\cdot x_1\sin x_2\\
		&\quad-[(f_{1/4}''(t)+\nu_1f_{1/4}'(t))\sin b_0+\nu_2(f_{1/4}'(t)+\nu_1f_{1/4}(t))\sin b_0]\cdot x_1\cos x_2\\
		&\quad +\frac{1}{2}(f_{1/4}'(t)+\nu_1f_{1/4}(t))\cos(2b_0)\cdot \sin(2x_2)-\frac{1}{2}(f_{1/4}'(t)+\nu_1f_{1/4}(t))\sin(2b_0)\cdot \cos(2x_2)\\
		&\quad +[(f_{1/2}''(t)+\nu_1f_{1/2}'(t))\cos a_1-\nu_2(f_{1/2}'(t)+\nu_1f_{1/2}(t))\cos a_1]\cdot \cos x_1\\
		&\quad +[(f_{1/2}''(t)+\nu_1f_{1/2}'(t))\sin a_1-\nu_2(f_{1/2}'(t)+\nu_1f_{1/2}(t))\sin a_1]\cdot \sin x_1.
	\end{align*}
	Note that $\sigma_\theta h^x(x_1,t)$ can in fact be expressed as the following:
	$$
	\sigma_\theta h^x(x_1,t)=\sum_{\substack{j \in \hat{\mathcal{Z}}-\{(0,1)\}, m \in \{0,1\}}} \sigma_j^m(x)h_j^m(t)+ \sum_{m \in \{0,1\}}\sigma_{(0,1)}^m(x)h_{(0,1)}^m(t)x_1,
	$$
	where $h_j^m(t)$ may be constantly zero $\sigma_j^m(x)$ is defined as in \eqref{temp-base}. Furthermore, we denote
	\begin{align*}
		\sigma_\theta h^x(x_1,t)&=h_{(0,1)}^1(t)x_1\sin x_2+h_{(0,1)}^0(t)x_1\cos x_2+h_{(0,2)}^1(t)\sin(2x_2)\\
		&\quad+h_{(0,2)}^0(t)\cos(2x_2)+h_{(1,0)}^0(t)\cos x_1+h_{(1,0)}^1(t)\sin x_1.
	\end{align*}
	Now, the first step is to prove that for all $\varepsilon$, there holds
	\begin{align}\label{JEMS-7.2}
		\mathbf{P}\bigl(\|U_t - U_t^h\|_{L^\infty([0,1];\hat{H}^4)} \lesssim \varepsilon\bigr) > 0,
	\end{align}
	where $U_t^h$ is the solution to the control equation \eqref{Con-1}-\eqref{Con-ini} after applying the vorticity transformation. Then from the mild form
	\begin{align}\label{Delta_t}
		U_t - U_t^h=\int_0^te^{-(t-s)A}\big[\big(B(U_s^h)-B(U_s)\big)-\big(GU_s^h-GU_s\big)\big]ds+\Gamma_t-\int_0^te^{-(t-s)A}\sigma_\theta h^x(x_1,s)ds,
	\end{align}
	where $\Gamma_t:=\int_0^te^{-A(t-s)}\sigma_\theta dW_s$. Note that this stochastic convolution $\Gamma_t$ is only supported on $|k|_\infty\le 1$, while the smooth control is only supported on 
	$|k|_\infty\le 2$. For all $\varepsilon>0$, we have that
	\begin{align}\label{JEMS-7.3}
		\mathbf{P}\Bigl( \sup_{t \in (0,1)} \Bigl\| \Gamma_t - \int_0^t e^{-(t-s)A} \sigma_\theta h^x(x_1,s) \,\mathrm{d}s \Bigr\|_{L^\infty([0,1];\hat{H}^4)} < \varepsilon \Bigr) > 0.	
	\end{align}
	Indeed, observe that
	\begin{align*}
		\Gamma_t=\int_0^t e^{-(t-s)A}(\alpha_1\cos x_1\,\mathrm{d}W^1_s+\alpha_2\sin x_1\,\mathrm{d}W^2_s+\alpha_3\cos x_2\,\mathrm{d}W^3_s+\alpha_4\sin x_2\,\mathrm{d}W^4_s),
	\end{align*}
	then, one has 
	\begin{align*}
		\Delta_t:&=\Gamma_t - \int_0^t e^{-(t-s)A} \sigma_\theta h^x(x_1,s) \,\mathrm{d}s\\
		&=\int_0^t e^{-(t-s)A}\alpha_1\cos x_1(\mathrm{d}W^1-h^1(s)\mathrm{d}s)+\alpha_2\sin x_1(\mathrm{d}W^2-h^2(s)\mathrm{d}s)\\
		&\quad+\big(\alpha_3\cos x_2\mathrm{d}W^3-h^3(s)x_1\cos x_2\mathrm{d}s\big)+\big(\alpha_4\sin x_2\mathrm{d}W^4-h^4(s)x_1\sin x_2\mathrm{d}s\big)\\
		&\quad-\sin(2y_2)h^5(s)\mathrm{d}s-\cos(2y_2)vh^6(s)\mathrm{d}s,
	\end{align*}
	where
	\begin{align*}
		h^1(s)&:=\frac{\cos a_1}{\alpha_1 g}\big((f''_{1/2}(s)+\nu_1f'_{1/2}(s))-\nu_2(f'_{1/2}(s)+\nu_1f_{1/2}(s))\big),\\
		h^2(s)&:=\frac{\sin a_1}{\alpha_2 g}\big((f''_{1/2}(s)+\nu_1f'_{1/2}(s))-\nu_2(f'_{1/2}(s)+\nu_1f_{1/2}(s))\big),\\
		h^3(s)&:=-\frac{\sin b_0}{g}\big((f''_{1/4}(s)+\nu_1f'_{1/4}(s))+\nu_2(f'_{1/4}(s)+\nu_1f_{1/4}(s))\big),\\
		h^4(s)&:=-\frac{\cos b_0}{g}\big((f''_{1/4}(s)+\nu_1f'_{1/4}(s))+\nu_2(f'_{1/4}(s)+\nu_1f_{1/4}(s))\big),\\
		h^5(s)&:=\frac{1}{2g}(f'_{1/4}(s)+\nu_1f_{1/4}(s))\cos(2b_0),\\
		h^6(s)&:=-\frac{1}{2g}(f'_{1/4}(s)+\nu_1f_{1/4}(s))\sin(2b_0).
	\end{align*}
	Note that $h^i\in C^{\infty}([0,1];L^2)$ and $x\in \mathbb{T}^2$. It then follows from the Girsanov theorem, the smoothing properties of the heat semigroup, and the regularity of the stochastic convolution that there exists a probability measure $\mathbb{Q}$ such that for any $p\ge 1$,
	$$
	\mathbb{E}^{\mathbb{Q}}\big[\sup_{t\in[0,1]}\|\Delta_t\|^p_{\hat{H}^4}\big]\le C(p).
	$$
	By employing Markov's inequality and the positivity of the Wiener measure, we obtain equation \eqref{JEMS-7.3}. Next, by employing the generalized Gagliardo-Nirenberg inequality, parabolic regularity, and combined with equation \eqref{Delta_t}, we obtain that \eqref{JEMS-7.2}. 
	
	Next, considering the approximation between particle positions $x_t$ and $x_t^h$, projection flow $v_t$ and $v_t^h$, we take trajectories of sample $w$ for which \eqref{JEMS-7.2} holds. We then obtain
	$$
	\frac{\mathrm{d}}{\mathrm{d}t}(x_t^h - x_t) = \mathbf{u}_t^h(x_t^h) - \mathbf{u}_t(x_t) = \bigl(\mathbf{u}_t^h(x_t^h) - \mathbf{u}_t^h(x_t)\bigr) + \bigl(\mathbf{u}_t^h(x_t) - \mathbf{u}_t(x_t)\bigr),
	$$
	and 
	\begin{align*}
		\frac{\mathrm{d}}{\mathrm{d}t}(v_t^h - v_t) &= v_t^{h}\cdot \nabla \mathbf{u}_t^{h}(x^{h}_t)\cdot {(v_t^{h})}^\top {(v_t^{h})}^\top - v_t\cdot \nabla \mathbf{u}_t(x_t)\cdot {(v_t)}^\top {(v_t)}^\top \\
		&= \bigl(v_t^{h}\cdot \nabla \mathbf{u}_t^{h}(x^{h}_t)\cdot {(v_t^{h})}^\top {(v_t^{h})}^\top -v_t^{h}\cdot \nabla \mathbf{u}_t^{h}(x_t)\cdot {(v_t^{h})}^\top {(v_t^{h})}^\top \bigr) \\
		&\quad+ \bigl(v_t^{h}\cdot \nabla \mathbf{u}_t^{h}(x_t)\cdot {(v_t^{h})}^\top {(v_t^{h})}^\top-v_t\cdot \nabla \mathbf{u}^{h}_t(x_t)\cdot {(v_t)}^\top {(v_t)}^\top\bigr)\\
		&\quad+\bigl(v_t\cdot \nabla \mathbf{u}^{h}_t(x_t)\cdot {(v_t)}^\top {(v_t)}^\top-v_t\cdot \nabla \mathbf{u}_t(x_t)\cdot {(v_t)}^\top {(v_t)}^\top\bigr).
	\end{align*}
	Finally, proceeding similarly to \cite[Lemma 7.3]{JEMS}, we apply the stability properties of $U_t$ to establish the conclusion. This completes the proof.
\end{proof}
\begin{proposition}[Approximate controllability of matrix flow]\label{V2-thm5.3}
	For all \( M > 0 \) and \( \varepsilon > 0 \),
	\begin{equation}\label{eq:7.4}
		\mathbf{P}\Bigl(\bigl(\mathbf{u}_1, \theta_1,x_1, A_1\bigr) \in B_\varepsilon(\mathbf{0}) \times B_\varepsilon(0) \times \bigl\{A \in \mathrm{SL}_d(\mathbb{R}) : \lvert A \rvert > M\bigr\} \,\bigm|\, \bigl(\mathbf{u}_0, \theta_0, x_0, A_0\bigr) = (\mathbf{0},0, 0, \mathrm{Id}_{\mathbb{R}^2})\Bigr) > 0.
	\end{equation}
\end{proposition}
\begin{proof}
	The control and approximation procedures are analogous to those in Propositions \ref{V2-thm5.1}-\ref{V2-thm5.2}. Here we adopt the corresponding cellular flow as the velocity field,
	$$
	\mathbf{u}^h_t = f_A \biggl( \begin{pmatrix} \sin(y_2 - b) \\ \sin(y_1 - a) \end{pmatrix} \biggr),
	$$
	where $f_A\in C_c^\infty(0,1)$ and $\int_0^1f_A(s)ds=\log M$. The proof is concluded by analogous selection of the smooth control combined with stability arguments.
\end{proof}

We now succinctly outline the proofs of weak irreducibility for the base, Lagrangian, and projective processes.
\begin{proof}[Proof of Proposition \ref{irreducibility}]
By utilizing the dissipativity of the Boussinesq equation described in Step 1 of verifying condition $(b)$ for approximate controllability $(C)$, along with the approximate controllability itself, we may complete the proof of this proposition following the classical irreducibility argument in \cite[Section 7]{JEMS}.
\end{proof}
\appendix

\section{Moment bounds on stochastic Boussinesq equations}\label{priori}
In this section, we establish the a priori estimates stated in Proposition \ref{priori-eatimates} for the stochastic Boussinesq system and present several moment estimates previously established in \cite{JFA}.
For completeness and for the reader’s convenience, we provide detailed arguments here, though we do not emphasize their originality. In contrast to the stochastic Navier-Stokes equations and other nonlinear stochastic partial differential equations with dissipative (parabolic) structures, the process of performing energy estimates for the Boussinesq equations requires compensating for the buoyancy term $g\partial_x\theta$. 
We handle the temperature and momentum equations with different weighting strategies. As above the dependence on physical parameter in constants is suppressed in what follows. Denote
\[
\varkappa := \frac{\nu_1 \nu_2}{g^2}.
\]
Then, cf. \eqref{JFA-2.4} and \eqref{2.3},
\begin{equation}\label{A.1}
	\|U\|_{\hat{H}}^2 = \varkappa \|\omega\|_{L^2}^2 + \|\theta\|_{L^2}^2, \quad \|U\|_{\hat{H}^1}^2 = \varkappa \|\nabla \omega\|_{L^2}^2 + \|\nabla \theta\|_{L^2}^2. 
\end{equation}
Also recall that our domain is $\mathbb{T}^2 = \mathbb{R}^2 / (2\pi\mathbb{Z}^2),$ and therefore the Poincaré inequality takes the form $\|U\|_{\hat{H}} \leq \|U\|_{\hat{H}^1}$. We first recall a standard exponential martingale estimate we will be using often.

\begin{lemma}\label{A.2}
	Let \( M_t \) a continuous \( L^2 \)-martingale. Then for all \( 0 \leq \eta \leq A \),
	\[
	\mathbb{E} \exp \left( \sup_{t \geq 0} \eta M_t - \eta A \langle M, M \rangle_t \right) \leq 2 \mathbb{E} e^{\eta M_0}.
	\]
\end{lemma}

Firstly, we show an easier version of \eqref{L2-eatimates}.
\begin{lemma}\label{A.3}
	There exists \( C > 0 \) such that for all \( \eta \leq C^{-1} \),
	\[
	\mathbb{E} \exp \left( \eta \| U_t \|_{\hat{H}}^2 \right) \leq C \exp \left( \eta e^{-C^{-1}t} \| U_0 \|_{\hat{H}}^2 \right).
	\]
\end{lemma}

\begin{proof}[Proof of \eqref{L2-eatimates}]
For any \( f \in \hat{H}^1 \) denote
\begin{equation}\label{DA}
	\|f\|_{D(A^{1/2})}^2 := \varkappa\nu\|\nabla f_1\|_{L^2}^2 + \mu\|\nabla f_2\|_{L^2}^2.
\end{equation}
By It\^o's formula, the process \( U_t=(\omega_t,\theta_t) \) satisfies:
\begin{equation}\label{2}
	\|\omega_t\|_{L^2}^2-\|\omega_s\|_{L^2}^2=-2\nu_1\int_s^t\|\omega_r\|_{H^1}^2dr+2\int_s^t\langle g\partial_x\theta,\omega_r\rangle dr,
\end{equation}
\begin{equation}\label{3}
	\|\theta_t\|_{L^2}^2-\|\theta_s\|_{L^2}^2=-2\nu_2\int_s^t\|\theta_r\|_{H^1}^2dr+2\int_s^t\langle \sigma_\theta,\theta\rangle dW_r+\|\sigma_\theta\|^2(t-s).
\end{equation}
Now we weight differently the equations, multiplying \eqref{2} by \(\varkappa\) and adding to \eqref{3} we obtain
\begin{equation*}
	\|U_t\|_{\hat{H}}^2-\|U_s\|_{\hat{H}}^2=-2\int_s^t\|U_r\|_{D(A^{\frac{1}{2}})}^2dr+2\varkappa\int_s^t\langle g\partial_x\theta,\omega_r\rangle dr+2\int_s^t\langle \sigma_\theta,\theta\rangle dW_r+\|\sigma_\theta\|^2(t-s).
\end{equation*}	
Since by the Poincaré inequality
\begin{align*}
2\varkappa g|\langle \partial_x\theta, \omega\rangle| 
&\leq \nu_1\varkappa\|\omega\|_{L^2}^2 + \nu_2\|\nabla\theta\|_{L^2}^2 \leq \nu_1\varkappa\|\nabla\omega\|_{L^2}^2 + \nu_2\|\nabla\theta\|_{L^2}^2\\
&=\|U\|_{D(A^{1/2})}^2,\\
\kappa\|U\|^2 &\leq \|U\|_{D(A^{1/2})}^2,\quad \langle \sigma_\theta,\theta\rangle\le \|\sigma_\theta\|\|U\|_{\hat{H}}.
\end{align*}
Then for any $\eta>0$, we have
\begin{align*}
	\eta\|U_t\|_{\hat{H}}^2+\frac{\eta\kappa}{2}\int_s^t\|U_r\|_{\hat{H}^1}^2dr-\eta(t-s)\|\sigma_\theta\|^2
	&\le \eta\|U_s\|_{\hat{H}}^2+2\eta\int_s^t\langle \sigma_\theta,\theta\rangle dW_r-\frac{\eta\kappa}{2}\int_s^t\|U_r\|_{\hat{H}^1}^2dr\\
	&\le M_t^s-C^{-1}{\langle M^s,M^s\rangle}_t,
\end{align*}
where $$M_t^s:=\eta\|U_t\|_{\hat{H}}^2+2\eta\int_s^t\langle \sigma_\theta,\theta\rangle dW_r$$
and we compute 
\begin{align*}
	\langle M^s, M^s \rangle_t \leqslant C \int_s^t \|U_r\|_{\hat{H}}^2 \, \mathrm{d}r \leqslant C\frac{\eta\kappa}{2} \int_s^t \|U_r\|_{\hat{H}^1}^2 \, \mathrm{d}r.
\end{align*}
Thus Lemma \ref{A.2} gives 
$$
\mathbb{E} \exp\left(\eta\bigg( \sup_{t \geq s} \|U_t\|_{\hat{H}}^2 + \frac{\kappa}{2} \int_s^t \|U_r\|_{\hat{H}^1}^2 \, \mathrm{d}r - C(t-s)\bigg)\right) \leqslant 2\mathbb{E} \exp\left(\eta\|U_s\|_{\hat{H}}^2\right).
$$
Using Lemma \ref{A.3} to bound the right-hand side and rearranging somewhat, we conclude.
\end{proof}

We now consider controlling higher derivatives of $(\omega_t,\theta_t)$. Differentiating the equation \eqref{Boussinesq-w} for $(\omega_t,\theta_t)$, we see
\begin{align}\label{4}
	\frac{\mathrm{d}}{\mathrm{d}t} \nabla^n \omega_t = \nu \Delta \nabla^n \omega_t- \Delta^{-1} \nabla^{\perp} \omega_t \cdot \nabla \nabla^n \omega_t - \sum_{j=0}^{n-1} \binom{n}{j} \Delta^{-1} \nabla^{\perp} \nabla^{n-j} \omega_t \cdot \nabla \nabla^j \omega_t+g\partial_x\nabla^n \theta_t,
\end{align}
\begin{align}\label{5}
	\frac{\mathrm{d}}{\mathrm{d}t} \nabla^n \theta_t = \nu \Delta \nabla^n \theta_t- \Delta^{-1} \nabla^{\perp} \omega_t \cdot \nabla \nabla^n \theta_t - \sum_{j=0}^{n-1} \binom{n}{j} \Delta^{-1} \nabla^{\perp} \nabla^{n-j} \omega_t \cdot \nabla \nabla^j \theta_t+ \nabla^n \sigma_\theta \mathrm{d}W.
\end{align}

Before continuing, we note the following bound from \cite[Lemma A.3]{NS}. It is noteworthy that we have rectified a typographical error present in \cite{NS} in this context.
\begin{lemma}\label{A.4}
	Let \( f \in C^{\infty}(\mathbb{T}^2) \), then there exists \( C(n) > 0 \) such that  
	\[
	2 \sum_{j=0}^{n-1} \binom{n}{j} \int |\nabla^n f| \, |\Delta^{-1} \nabla^\perp \nabla^{n-j} f| \, |\nabla \nabla^j f| \, dx \leqslant \frac{1}{2} \|f\|_{H^{n+1}} + C \|f\|_{L^2}^{2n+4}.
	\]
\end{lemma}

We now proceed to establish an easier version of \eqref{Hn-eatimates1}, analogous to Lemma \ref{A.3}. To facilitate this, we first introduce an auxiliary result pivotal to the subsequent analysis.
\begin{lemma}\emph{\cite[Lemma5.1]{JFA-55}}\label{A-55}
	Let \( U \) be a real-valued semimartingale  
	\[
	\mathrm{d}U(t, w) = F(t, w) \, \mathrm{d}t + G(t, w) \, \mathrm{d}B(t, w),
	\]  
	where \( B \) is a standard Brownian motion. Assume that there exist a process \( Z \) and positive constants \( b_1, b_2, b_3 \), with \( b_2 > b_3 \), such that  
	\[
	F(t, w) \leq b_1 - b_2 Z(t, w), \quad U(t, w) \leq Z(t, w), \quad G(t, w)^2 \leq b_3 Z(t, w) \quad \text{a.s.}
	\]  
	Then, the bound  
	\[
	\exp\left( U(t) + \frac{b_2 e^{-b_2 t/4}}{4} \int_0^t Z(s) \, \mathrm{d}s \right) \leq \frac{b_2 \exp(2b_1/b_2)}{b_2 - b_3} \exp\left(U(0)e^{-(b_2/2)t}\right)
	\]  
	holds for any \( t \geq 0 \).
\end{lemma}

The easier version of \eqref{Hn-eatimates1} as follows.
\begin{lemma}\label{A.5}
There exists \( C(n) > 0 \) such that for all \( 0 \leq \eta \leq C^{-1} \) and all \( t \geq 0 \),
	\begin{equation*}
		\mathbb{E} \exp \left( \eta \| U_t \|_{\hat{H}^n}^{\frac{2}{n+2}} \right) \leq C \exp \left( e^{-C^{-1} t} \eta \| U_0 \|_{\hat{H}^n}^{\frac{2}{n+2}} + C \eta \| U_0 \|_{\hat{H}}^2 \right).
	\end{equation*}
\end{lemma}

\begin{proof}
For any $f\in \hat{H}^{n+1}$ denote
$$
\|f\|_{D^n(A^{1/2})}^2 := \varkappa\nu\|\nabla f_1\|_{H^n}^2 + \mu\|\nabla f_2\|_{H^n}^2.
$$
By It\^o's formula and \eqref{4}-\eqref{5}, we have
\begin{align} \label{A.6}
	\mathrm{d}{\|\omega_t\|}_{H^n}^2=&-2\nu_1{\|\omega_t\|}_{H^{n+1}}^2\mathrm{d}t-2\sum_{j=0}^{n-1}\binom{n}{j}\int\nabla^n \omega_t:\Delta^{-1} \nabla^{\perp} \nabla^{n-j} \omega_t \cdot \nabla \nabla^j \omega_t\mathrm{d}x \mathrm{d}t \notag \\
	&+2g\int\nabla^n \omega_t:\partial_x\nabla^n \theta_t\mathrm{d}x \mathrm{d}t,
\end{align}
and
\begin{align} \label{A.7}
	\mathrm{d}{\|\theta_t\|}_{H^n}^2=&-2\nu_2{\|\theta_t\|}_{H^{n+1}}^2\mathrm{d}t-2\sum_{j=0}^{n-1}\binom{n}{j}\int\nabla^n \theta_t:\Delta^{-1} \nabla^{\perp} \nabla^{n-j} \omega_t \cdot \nabla \nabla^j \theta_t\mathrm{d}x \mathrm{d}t \notag \\
	&+2\sum_{\substack{j \in \mathcal{Z} \\ m \in \{0,1\}}}\alpha_j^m\left(\int\nabla^n \theta_t:\nabla^n\sigma_j^m\mathrm{d}x\right)\mathrm{d}W^{j,m}+{\|\sigma_\theta\|}^2\mathrm{d}t.
\end{align}
Multiplying \eqref{A.6} by \(\varkappa\) and adding to \eqref{A.7} we obtain
\begin{align}\label{U-n}
	\mathrm{d}\left(\eta{\|U_t\|}_{\hat{H}^n}^2\right)=&\eta\left(-2{\|U_t\|}_{D^n(A^{1/2})}^2+2g\varkappa\int\nabla^n \omega_t:\partial_x\nabla^n \theta_t\mathrm{d}x+{\|\sigma_\theta\|}^2\right)\mathrm{d}t \notag\\
	&-2\sum_{j=0}^{n-1}\binom{n}{j}\eta\varkappa\int\nabla^n \omega_t:\Delta^{-1} \nabla^{\perp} \nabla^{n-j} \omega_t \cdot \nabla \nabla^j \omega_t\mathrm{d}x \mathrm{d}t \notag\\
	&-2\sum_{j=0}^{n-1}\binom{n}{j}\int\nabla^n \theta_t:\Delta^{-1} \nabla^{\perp} \nabla^{n-j} \omega_t \cdot \nabla \nabla^j \theta_t\mathrm{d}x \mathrm{d}t \notag\\
	&+2\sum_{\substack{j \in \mathcal{Z} \\ m \in \{0,1\}}}\alpha_j^m\left(\int\nabla^n \theta_t:\nabla^n\sigma_j^m\mathrm{d}x\right)\mathrm{d}W^{j,m}.
\end{align}
By the Poincaré inequality,
\begin{align}\label{U-n-1}
	2g\varkappa|\langle \nabla^n \omega_t, \partial_x\nabla^n \theta_t\rangle| \leq \|U\|_{D^n(A^{1/2})}^2,\notag\\
	\kappa\|U\|_{\hat{H}^n}^2 \leq \|U\|_{D^n(A^{1/2})}^2,\quad \langle \nabla^n \theta_t,\nabla^n\sigma_j^m\rangle\le \|\sigma_\theta\|\|U\|_{\hat{H}^n}.
\end{align}
Note that
\begin{align*}
	\int|\nabla^n \omega_t||\Delta^{-1} \nabla^{\perp} \nabla^{n-j} \omega_t||\nabla \nabla^j \omega_t|\mathrm{d}x\le C{\|\omega\|}_{W^{n,4}}{\|\omega\|}_{W^{n-k,p}}{\|\omega\|}_{W^{k,q}}
\end{align*}
and 
\begin{align*}
	\int|\nabla^n \theta_t||\Delta^{-1} \nabla^{\perp} \nabla^{n-j} \omega_t||\nabla \nabla^j \theta_t|\mathrm{d}x\le C{\|\theta\|}_{W^{n,4}}{\|\omega\|}_{W^{n-k,p}}{\|\theta\|}_{W^{k,q}},
\end{align*}
where 
\begin{align}\label{pq1}
q=\frac{2n-k}{4n}\,\,\text{and}\,\, p=\frac{n+k}{4n},
\end{align}
such that
\begin{align}\label{pq2}
\frac{1}{4}+\frac{1}{p}+\frac{1}{q}=1
\end{align}
Then 
\begin{align*}
	&\Big|2\sum_{j=0}^{n-1}\binom{n}{j}\eta\varkappa\int\nabla^n \omega_t:\Delta^{-1} \nabla^{\perp} \nabla^{n-j} \omega_t \cdot \nabla \nabla^j \omega_t\mathrm{d}x \mathrm{d}t\\
	\quad&+2\sum_{j=0}^{n-1}\binom{n}{j}\int\nabla^n \theta_t:\Delta^{-1} \nabla^{\perp} \nabla^{n-j} \omega_t \cdot \nabla \nabla^j \theta_t\mathrm{d}x \mathrm{d}t\Big|\\
	\quad&\le C{\|U\|}_{\widehat{W}^{n,4}}{\|U\|}_{\widehat{W}^{n-k,p}}{\|U\|}_{\widehat{W}^{k,q}},
\end{align*}
where $p$ and $q$ satisfy equations \eqref{pq1}-\eqref{pq2}, and 
$$
\widehat{W}^{s,q}:=\left\{U:={(\omega,\theta)}^T\in {W^{s,q}(\mathbb{T}^2)}^2:\int_{\mathbb{T}^2}\omega dx=\int_{\mathbb{T}^2}\theta dx=0\right\} \quad \text{for any } s,q \geq 0.
$$
Subsequently, following the proof methodology of \cite[Lemma A.3]{NS}, we derive:
\begin{align}\label{NSF}
	&\Big|2\sum_{j=0}^{n-1}\binom{n}{j}\eta\varkappa\int\nabla^n \omega_t:\Delta^{-1} \nabla^{\perp} \nabla^{n-j} \omega_t \cdot \nabla \nabla^j \omega_t\mathrm{d}x \mathrm{d}t\notag\\
	\quad&+2\sum_{j=0}^{n-1}\binom{n}{j}\int\nabla^n \theta_t:\Delta^{-1} \nabla^{\perp} \nabla^{n-j} \omega_t \cdot \nabla \nabla^j \theta_t\mathrm{d}x \mathrm{d}t\Big|\notag\\
	\quad&\le \frac{\kappa}{2}\|U\|_{\hat{H}^{n+1}} + C_\kappa \|U\|_{\hat{H}}^{2n+4}.
\end{align}
Then, combining equations \eqref{U-n}-\eqref{U-n-1}, we obtain:
$$
\mathrm{d}\left(\eta{\|U_t\|}_{\hat{H}^n}^2\right)\le \eta\left((\frac{\kappa}{2}-1){\|U_t\|}_{D^n(A^{\frac{1}{2}})}^2+C_\kappa \|U_t\|_{\hat{H}}^{2n+4}+{\|\sigma_\theta\|}^2\right)\mathrm{d}t+2\eta\langle\nabla^n\sigma_\theta,\nabla^n\theta\rangle\mathrm{d}W.
$$
Thus, We have for \( Z_t := \frac{\eta}{\kappa} \|U_t\|_{D^n(A^{1/2})}^2 \) and \( V_t := \eta \|U_t\|_{\hat{H}^n}^2 \) that \( V \leq Z \), and 
\begin{align*}
\eta\left((\frac{\kappa}{2}-1){\|U_t\|}_{D^n(A^{\frac{1}{2}})}^2+C_\kappa \|U_t\|_{L^2}^{2n+4}+{\|\sigma_\theta\|}^2\right)
&\le \eta C_\kappa \|U_t\|_{L^2}^{2n+4}+\eta \|\sigma_\theta\|^2-\frac{\kappa(2-\kappa)}{2}Z_t,\\
4{\eta}^2{\big|\langle\nabla^n\sigma_\theta,\nabla^n\theta\rangle\big|}^2\le 4{\eta}^2{\|\sigma_\theta\|}^2{\|U_t\|}_{\hat{H}^n}^2&\le 4{\eta}{\|\sigma_\theta\|}^2\cdot \frac{\eta}{\kappa}{\|U_t\|}_{D^n(A^{\frac{1}{2}})}^2.
\end{align*}
Combining Lemma \ref{A.3}, \eqref{noise-form}, and Lemma \ref{A-55}, we conclude that for any $\eta<\frac{\kappa(2-\kappa)}{8\|\sigma_\theta\|^2}(\kappa<2)$, the following holds:
\begin{align*}
&\mathbb{E} \exp \left( \eta \| U_t \|_{\hat{H}^n}^2 + \frac{\kappa(2-\kappa)}{8} e^{-\frac{\kappa(2-\kappa)}{8}t} \int_0^t Z_s \, ds \right) \\
&\leq \frac{\kappa(2-\kappa)}{\kappa(2-\kappa)-8\eta\|\sigma_\theta\|^2}\mathbb{E}\exp\left(\frac{4\eta\left(C_\kappa\|U_t\|_{L^2}^{2n+4}+\|\sigma_\theta\|^2\right)}{\kappa(2-\kappa)}\right) \exp \left( \eta e^{-\frac{\kappa(2-\kappa)}{4}t} \| U_0 \|_{\hat{H}^n}^2 \right).
\end{align*}
Note that $\kappa Z_s = \eta \| U_s \|_{D^n(A^{1/2})}^2 \geq \eta \kappa \| U_s \|_{\hat{H}^{n+1}}^2$, then we get
\begin{align*}
	&\mathbb{E} \exp \left( \eta \| U_t \|_{\hat{H}^n}^2 + \eta\frac{\kappa(2-\kappa)}{8} e^{-\frac{\kappa(2-\kappa)}{8}t} \int_0^t \|U_s\|_{\hat{H}^{n+1}}^2 \, ds \right) \\
	&\leq \frac{\kappa(2-\kappa)}{\kappa(2-\kappa)-8\eta\|\sigma_\theta\|^2}\mathbb{E}\exp\left(\frac{4\eta\left(C_\kappa\|U_t\|_{\hat{H}}^{2n+4}+\|\sigma_\theta\|^2\right)}{\kappa(2-\kappa)}\right) \exp \left( \eta e^{-\frac{\kappa(2-\kappa)}{4}t} \| U_0 \|_{\hat{H}^n}^2 \right).
\end{align*}
Note that this immediately implies
\begin{align}\label{+}
	\mathbb{E} \exp \left(\eta \| U_t \|_{\hat{H}^n}^2-e^{-\frac{\kappa(2-\kappa)}{4}t}\eta\| U_0 \|_{\hat{H}^n}^2-C\eta\|U_t\|_{\hat{H}}^{2n+4} \right)_+\le C,
\end{align}
where \( x_+ := x \lor 0 \). Our goal is now to use \eqref{L2-eatimates} to remove the term involving \(\| U_t \|_{\hat{H}} \). To that end, we compute
\begin{align}\label{AA.3}
	&\mathbb{E} \exp \left( \eta \| U_t \|_{\hat{H}^{n}}^{\frac{2}{n+2}} \right) \notag \\
	&\leq \mathbb{E} \exp \left( \left( \eta^{n+2} \| U_t \|_{\hat{H}^n}^2 - e^{-\frac{\kappa(2-\kappa)}{4} t} \eta^{n+2} \| U_0 \|_{\hat{H}^n}^2 - C \eta^{n+2} \| U_t \|_{\hat{H}}^{2n+4} \right)_+^{\frac{1}{n+2}} \right) \notag \\
	&\quad \times \exp \left( e^{-C^{-1}t} \eta \| U_0 \|_{\hat{H}^n}^{\frac{2}{n+2}} + C \eta \| U_t \|_{\hat{H}}^{2} \right) \notag\\
	&\leq C \left( \mathbb{E} \exp \left( 2\eta^{n+2} \| U_t \|_{\hat{H}^n}^2 - 2e^{-\frac{\kappa(2-\kappa)}{4} t} \eta^{n+2} \| U_0 \|_{\hat{H}^n}^2 - C \eta^{n+2} \| U_t \|_{\hat{H}}^{2n+4} \right) \right)^{1/2} \notag\\
	&\quad \times \exp \left( e^{-C^{-1}t} \eta \| U_0 \|_{\hat{H}^n}^{\frac{2}{n+2}} \right) \left( \mathbb{E} \exp \left( C \eta \| U_t \|_{\hat{H}}^{2} \right) \right)^{1/2} \notag \\
	&\leq C \exp \left( e^{-C^{-1}t} \eta \| U_0 \|_{\hat{H}^n}^{\frac{2}{n+2}} \right) \left( \mathbb{E} \exp \left( C \eta \| U_t \|_{\hat{H}}^{2}\right) \right)^{1/2},
\end{align}
where we use \eqref{+}. Subsequently utilizing \eqref{L2-eatimates}, we obtain 
\begin{align}\label{AA.4}
	\mathbb{E} \exp \left( C \eta \| U_t \|_{\hat{H}}^{2}\right) \le Ce^{Ct}\exp(C\eta\|U_0\|_{\hat{H}}^2).
\end{align}
Combining \eqref{AA.3} and \eqref{AA.4}, we conclude.
\end{proof}

\begin{proof}[Proof of \eqref{Hn-eatimates1}]
By It\^o's formula, \eqref{4}-\eqref{5}, \eqref{NSF} and then Integrating, we see that
\[
	\| U_t \|_{\hat{H}^n}^2 + \frac{\kappa}{2} \int_s^t \| U_r \|_{\hat{H}^{n+1}}^2 \, dr - C \int_s^t \| U_r \|_{\hat{H}}^{2n+4} \, dr - C(t-s) - \| U_s \|_{\hat{H}^n}^2 \leq M_t^s - \frac{1}{C} \langle M^s, M^s \rangle_t,
\]
where
\[
	M_t^s := \int_s^t \langle\nabla^n \theta_r,\nabla^n\sigma_\theta \rangle dW_r.
\]
Analogously, we establish the conclusion.
\end{proof}
The \eqref{Hn-eatimates2} can be proven by employing \eqref{L2-eatimates} and \eqref{Hn-eatimates1} with techniques analogous to those in \cite{NS}.


Finally, we present several moment estimates previously established in \cite[Lemma.A.1]{JFA}, which are frequently employed throughout the manuscript.
\begin{lemma}
There exists $\eta^*>0$ such that:\\
(i) For any $s \geq 0$, $p \geq 2$, and $\eta \in (0, \eta^*]$ there exists $C = C(\eta, s, T, p)$ such that
\begin{align}\label{JFA-A.5}
\mathbb{E} \left( \sup_{t \in [T/2, T]} \|U(t)\|_{\hat{H}^s}^p \right) \leq C \exp\bigl( \eta \|U_0\|_{\hat{H}}^2 \bigr). 
\end{align}
(ii) For any $p \geq 2$, $s \geq 0$, $\eta > 0$, and $T > 0$ there is $C = C(s, T, p, \eta)$ such that
\begin{align}\label{JFA-A.6}
\mathbb{E} \left( \|U\|_{C^{1/4}([T/2, T], \hat{H}^s)}^p \right) \leq C \exp\bigl( \eta \|U_0\|_{\hat{H}}^2 \bigr). 
\end{align}
\end{lemma}




  
  \section*{Declarations}
  
  \noindent{Availability of data:} No new data were generated or analysed in support of this search.
  
  \noindent{Conflict of interests:} 
  The authors declare that there are no conflict of interests, we do not have any possible conflicts of interest.
  
  \noindent{Funding:} The manuscript is supported by National Natural Science Foundation of China (No. 12571189).


\end{document}